\documentclass[3p, times, singlecolumn]{elsarticle}

\usepackage[utf8]{inputenc}
\usepackage[english]{babel}
\usepackage{graphicx}
\usepackage{subcaption}
\usepackage[textformat=period]{caption}
\usepackage{xcolor}
\usepackage{amsmath, amsthm, amssymb}
\usepackage{bm}
\usepackage{pdflscape}
\usepackage{algorithm}
\usepackage{algpseudocode}
\usepackage[numbers]{natbib}
\usepackage{float}
\usepackage[colorlinks]{hyperref}
\usepackage[symbols,nogroupskip,sort=none]{glossaries-extra}

\usepackage{multirow}

\usepackage[
	detect-family=true,
	detect-inline-family=math,
	exponent-product=\cdot,
	output-decimal-marker={.}, 
	range-phrase={~to~},		
	range-units=single,		
	retain-explicit-plus,	
	retain-unity-mantissa=false,
	retain-zero-exponent=true,	
	separate-uncertainty=true,
]{siunitx}
\usepackage[]{chemformula}
\usepackage{blindtext}
\usepackage{overpic}
\usepackage{booktabs}
\usepackage{longtable}
\usepackage{mdframed}

\newtheorem{thm}{Theorem}

\newtheorem{prop}[thm]{Proposition}
\newtheorem{lm}[thm]{Lemma}

\newtheorem{ass}[thm]{Assumption}

\newcommand{\revtext}[1]{\textcolor{black}{#1}}

\newcommand{\R}{\mathbb{R}}
\newcommand{\finter}{f^{\mathrm{i}}}
\newcommand{\ginter}{g^{\mathrm{i}}}
\newcommand{\fpert}{f^{\mathrm{p}}}
\newcommand{\gpert}{g^{\mathrm{p}}}
\newcommand{\frel}{f^{\mathrm{r}}}
\newcommand{\grel}{g^{\mathrm{r}}}

\newcommand{\pref}{\theta_{\mathrm{ref}}}
\newcommand{\prefidx}[1]{\theta_{\mathrm{ref}_{#1}}}
\newcommand{\palt}{\theta_{\mathrm{alt}}}
\newcommand{\C}{i \in \{1,\dots,C\}}

\newcommand{\fvlei}{f_{\mathrm{vle,\mathit{i}}}}
\newcommand{\fhold}{f_{\mathrm{holdup}}}
\newcommand{\fhl}{f_{\mathrm{hl}}}
\newcommand{\fhv}{f_{\mathrm{hv}}}
\newcommand{\hli}{h_{\mathrm{l},\mathit{i}}}
\newcommand{\hvi}{h_{\mathrm{v},\mathit{i}}}
\newcommand{\hl}{\bm{h}_{\mathrm{l}}}
\newcommand{\hv}{\bm{h}_{\mathrm{v}}}
\newcommand{\hvapi}{h_{\mathrm{vap},i}}
\newcommand{\nl}{n_{\mathrm{l}}}
\newcommand{\nv}{n_{\mathrm{v}}}
\newcommand{\napp}{n^{\mathrm{app}}}
\newcommand{\xapp}{\bm{x}^{\mathrm{app}}}

\newcommand{\nappinit}{n^{\mathrm{app},0}}
\newcommand{\xappinit}{\bm{x}^{\mathrm{app},0}}
\newcommand{\xist}{\xi^{\mathrm{st}}}
\newcommand{\etast}{\eta^{\mathrm{st}}}
\newcommand{\xiin}{\xi^{\mathrm{in}}}
\newcommand{\etain}{\eta^{\mathrm{in}}}
\newcommand{\tpre}{t_{\mathrm{pre}}}
\newcommand{\tend}{t_{\mathrm{end}}}
\newcommand{\rhol}{\rho_{\mathrm{l}}}
\newcommand{\rhov}{\rho_{\mathrm{v}}}
\newcommand{\cli}{c_{\mathrm{l},\mathit{i}}}
\newcommand{\Tcond}{T^{\mathrm{cond}}}
\newcommand{\xcond}{\bm{x}^{\mathrm{cond}}}
\newcommand{\Psati}{P_{\mathrm{sat},\mathit{i}}}
\newcommand{\aux}{\mathrm{aux}}
\newcommand{\ydef}{\mathrm{ydef}}
\newcommand{\enthdef}{\mathrm{edef}}
\newcommand{\xsum}{\mathrm{xsum}}
\newcommand{\ysum}{\mathrm{ysum}}
\newcommand{\hold}{\mathrm{hold}}
\newcommand{\slack}{\mathrm{slack}}
\newcommand{\aebal}{\mathrm{aebal}}
\newcommand{\devmeangamma}[1]{\Delta^{\mathrm{m}}_{\gamma_{#1}}}
\newcommand{\devmeanT}{\Delta^{\mathrm{m}}_T}
\newcommand{\devworstgamma}[1]{\Delta^{\mathrm{w}}_{\gamma_{#1}}}
\newcommand{\devworstT}{\Delta^{\mathrm{w}}_T}
\newcommand{\devaltx}[3]{\Delta_{x_{#1}^{#2,#3}}}
\newcommand{\devaltgamma}[3]{\Delta_{\gamma_{#1}^{#2, #3}}}

\begin{document}

\begin{frontmatter}

\title{An equation-based batch distillation simulation to evaluate the effect of multiplicities in thermodynamic activity coefficients}

\author[ITWM-address]{Jennifer Werner}
\ead{jennifer.werner@itwm.fraunhofer.de}
\author[ITWM-address]{Jochen Schmid\corref{mycorrespondingauthor}}
\cortext[mycorrespondingauthor]{Corresponding author}
\ead{jochen.schmid@itwm.fraunhofer.de}
\author[CMU-address]{Lorenz T. Biegler}
\ead{lb01@andrew.cmu.edu}
\author[ITWM-address]{Michael Bortz}
\ead{michael.bortz@itwm.fraunhofer.de}
\address[ITWM-address]{Fraunhofer Institute for Industrial Mathematics (ITWM), Kaiserslautern, Germany}
\address[CMU-address]{Chemical Engineering Department, Carnegie Mellon University, Pittsburgh, Pennsylvania, USA}

\begin{abstract}
In this paper, we investigate the influence of multiplicities in activity coefficients on batch distillation processes. In order to do so, we develop a rigorous simulation of batch distillation processes based on the MESH equations. In particular, we propose a novel index reduction method to transform the original index-$2$ system into a well-posed differential-algebraic system of index $1$. With the help of this simulation, we then explore whether the alternative NRTL parameters, which yield indistinguishable activity coefficients and VLE diagrams when compared to the reference data, can produce distinguishable profiles in dynamic simulations. As it turns out, this can  happen in general and we explain the reasons behind the dynamic distinguishability.
\end{abstract}

\begin{keyword}
batch distillation \sep
index reduction \sep
infinite-reflux initialization \sep 
non-random two-liquid model \sep 
vapor-liquid-equilibrium \sep 
multiplicities \sep 
activity coefficients
\end{keyword}



\end{frontmatter}

\section{Introduction}

It is well-known that the parameters of activity coefficient models are not unique (see for example \revtext{\cite{miyahara1970evaluation, silverman1977number, gau2000reliable, tassios1979number, vasiliu2021multiple}}). In fact, for some mixtures, the non-random two-liquid (NRTL) model \revtext{\cite{renon1968local, gmehling2019chemical}} has up to five solutions, as is shown in \cite{werner2023multiplicities}. \revtext{It is also shown in \cite{werner2023multiplicities}} that the resulting alternative NRTL parameters have little to no effect on the vapor-liquid equilibria (VLE). A natural question is what influence these alternative NRTL parameters have on certain chemical-engineering processes based on vapor-liquid equilibria, such as a batch distillation process. Answering this question is the goal of this work. In order to achieve this, we develop a rigorous simulation for batch distillation processes based on a novel index reduction technique.

As is well-known, the rigorous modeling of a batch distillation process relies on dynamic models described by a set of coupled differential and algebraic equations (DAE).
The first rigorous, multi-component batch distillation model based on the assumption of equilibrium stages was developed by \cite{meadows1963multicomponent}. In \cite{distefano1968mathematical}, the model was extended and a computer-based method for solving the equations was developed, which was further improved in \cite{boston1981foundations}. The modeling, simulation, and optimization of batch distillation processes is intensively studied in the literature \revtext{\cite{Doherty1978, VanDongen1985, cervantes1998, cervantes2000, CERVANTES200041, biegler2002, ragunathan2004, eckert2008mathematical, lopez2016rigorous, bortz2019estimating, Mohring2022, qian2023nonlinear}} and the established models are discussed in many textbooks, see for example \cite{seader2006separation, biegler2010nonlinear, biegler1997systematic, king2013separation, diwekar2011batch}.
In most models, the dynamics of the process is described  by the MESH equations, which comprise mass balances, thermodynamic equilibrium conditions, summation conditions over the phase concentrations, and heat balances on each stage of the considered distillation column (see \cite{seader2006separation, biegler2010nonlinear}). It is well-known that these equations constitute a differential-algebraic system of index $2$ \revtext{\cite{gani1992modelling, peng2003dynamic, lopez2016rigorous, qian2023nonlinear}} which, without suitable transformations, cannot be solved by most standard DAE solvers.

Since, on the other hand, index-$1$ systems can be reliably solved by most standard solvers \revtext{\cite{hairer1991ii}}, it is important to transform the MESH equations to a differential-algebraic system of index $1$. A general principle to achieve this \revtext{\cite{kunkel2006differential, bachmann1990methods}} is to differentiate (a subset of) algebraic equations, substitute  (a subset of) differential equations, and delete the corresponding differential equations to obtain a square, index-$1$ DAE system. In most papers on batch distillation simulations~\revtext{\cite{cervantes1998, cervantes2000, CERVANTES200041, biegler2002, ragunathan2004, eckert2008mathematical, qian2023nonlinear}}, the index reduction is achieved by differentiating the vapor summation equation. In the present paper, we propose an index reduction method based on differentiating the liquid summation equation instead. A similar index reduction has also been indicated in~\cite{lopez2016rigorous}, but an important singularity issue seems to have been overlooked there. In the approach suggested here, we avoid singularity issues by suitable perturbations and we rigorously establish the index-$1$ property of the reduced differential-algebraic system we propose. Apart from its conceptual simplicity and transparency, a major advantage of the proposed index reduction approach is that it does not require the differentiation of any vapor-liquid equilibrium model. In the more traditional approaches, by contrast, the partial derivatives of the equilibrium model are needed to even formulate the reduced system equations. In particular, exchanging the vapor-liquid equilibrium model is considerably more difficult for those approaches. 

Another important issue in the solution of differential-algebraic equations in general and of the MESH equations in particular is the initialization with feasible starting values and the computation of consistent initial values. In the present paper, we use an infinite-reflux initialization inspired by \cite{fletcher2000initialising} combined with a boundary-value problem formulation for the transition from infinite refluxes at a pre-initial time instant to the prescribed finite reflux ratio value at the initial time. 
We implemented our index-reduced simulation model and the initialization methods in python using the packages pyomo and pyomo.dae \cite{bynum2021pyomo, hart2011pyomo, Nicholson2018}. 

With this simulation at hand, we then investigate the question of whether the alternative NRTL parameters, which lead to stationarily indistinguishable activity coefficients and VLE diagrams, can yield distinguishable outputs in the dynamic batch distillation. In fact, we show in the study results that there can be significant differences in the dynamic profiles of the liquid concentrations and the activity coefficients when alternative NRTL parameters are used. We discuss these differences in detail and furthermore analyze why these differences occur.

We organize the article as follows. In Section \ref{sec:simulation} we introduce our simulation in detail. Sections~\ref{sec:setting} and~\ref{sec:system-equations} formally define the DAE system describing the batch distillation model. In Section \ref{sec:index-reduction} we introduce our index reduction method, while in Section \ref{sec:initial-value-problem} we detail our initialization method for the index-reduced DAE system. In Section \ref{sec:study}, in turn, we investigate the influence of alternative, stationarily indistinguishable, NRTL parameters on the solutions of the batch distillation model. We present three different examples and give an explanation on the cause of these effects.

\section{Simulation of batch distillation columns}
\label{sec:simulation}

\revtext{
In this section, we begin by recording the setting and the differential-algebraic system equations for our batch distillation simulation (Sections~\ref{sec:setting} and \ref{sec:system-equations}). We then introduce our index reduction approach (Section~\ref{sec:index-reduction}) and our overall solution strategy for the initial-value problem for the index-reduced system (Section~\ref{sec:initial-value-problem}).
}

\subsection{Setting}
\label{sec:setting}

We consider batch distillation processes in distillation columns with $S \ge 2$ stages for multi-component mixtures consisting of $C \ge 2$ components. Stage $1$ is the pot, which is heated by means of a reboiler, and stage $S$ is the head, which is connected to a condenser. 
And from the condenser, in turn, the produced distillate is partly refluxed to the head of the column and partly effluxed into fraction containers. A schematic of the batch distillation columns considered here is depicted in Figure~\ref{fig:distillation_column_schematic}. 
As usual in stage-wise batch distillation modeling, we assume that thermodynamic equilibrium holds on all stages and that  neighboring stages are connected by liquid downstreams and vapor upstreams. \revtext{Also, the composition of each liquid downstream is the same as that of the liquid phase on the stage above, and the composition of each vapor upstream is the same as that of the vapor phase on the stage below.} Specifically, we adopt the following assumptions throughout this paper. 

\begin{ass} \label{ass:setting}
We have
\begin{itemize}
    \item[(i)] a zero vapor holdup on all stages at all times, 
    \item[(ii)] exactly two phases in equilibrium with each other, namely one vapor and one liquid phase, on all stages at all times (no liquid-liquid phase splitting),
    \item[(iii)] the same pressure on all stages (no pressure drop),   
    \item[(iv)] total condensation in the condenser at all times.
\end{itemize}
\end{ass}

In keeping with the zero vapor holdup assumption (Assumption~\ref{ass:setting}~(i)), we do not have to distinguish between the total molar holdup $n^j$ and the liquid molar holdup $\nl^j$ on the stages -- they are assumed to be the same 
\begin{align} \label{eq:zero-vapor-holdup}
n^j = \nl^j + \nv^j = \nl^j \qquad (j \in \{1,\dots,S\}).
\end{align}
Assumption~\ref{ass:setting}~(ii) means that we confine ourselves to vapor-liquid equilibria and do not model vapor-liquid-liquid equilibria for simplicity. In particular, we only need one liquid composition vector $\bm{x}^j$ and one vapor composition vector $\bm{y}^j$ on all stages. 
Assumption~\ref{ass:setting}~(iii), in turn, means that we do not model pressure drop in the column but instead assume one single column pressure
\begin{align}
P = P^j \qquad (j \in \{1,\dots,S\})
\end{align}
on all stages. And finally, the total condensation assumption (Assumption~\ref{ass:setting}~(iv)) means that the vapor upstream from the top-most stage $S$ is completely condensed. In particular, the composition $\xcond$ in the condenser (and thus the composition of the distillate stream) is equal to the vapor composition on stage $S$:
\begin{align}
\xcond = \bm{y}^S.
\end{align}
As a consequence of the above assumptions, the state of the considered distillation process is completely described by the 
variables 
\begin{align} \label{eq:state-variables}
n^j, \quad \bm{x}^j, \quad \bm{y}^j, \quad H^j, \quad T^j, \quad L^j, \quad V^j 
\end{align}
(namely, liquid moles, liquid composition, vapor composition, liquid enthalpy, temperature on stage $j$, liquid downstream from stage $j+1$ to stage $j$, and vapor upstream from stage $j$ to stage $j+1$). In order to influence the trajectories of the state variables, one can prescribe the values of four control variables (as functions of time),
namely the column pressure $P$, the reboiler heat duty $Q$, the condenser temperature $\Tcond$, and the efflux ratio $\epsilon$. Similar to the reflux ratio $R := F/E$, the efflux ratio 
\begin{align}
\epsilon := E/V^S
\end{align}
is defined as the percentage of the vapor stream $V^S$ that is withdrawn from the column. In these relations, 
\begin{align}
E = \epsilon V^S
\qquad \text{and} \qquad
F = L^S = (1-\epsilon) V^S
\end{align}
denote the flowrates of the efflux and reflux streams, respectively. In particular, efflux and reflux ratio are related as follows:
\begin{align} \label{eq:efflux-and-reflux-ratio}
R = \frac{1-\epsilon}{\epsilon} = 1/\epsilon - 1 \in [0,\infty] 
\qquad \text{and} \qquad
\epsilon = \frac{1}{1+R} \in [0,1].
\end{align}

\begin{figure}
	\centering
	\includegraphics[width=0.7\textwidth]{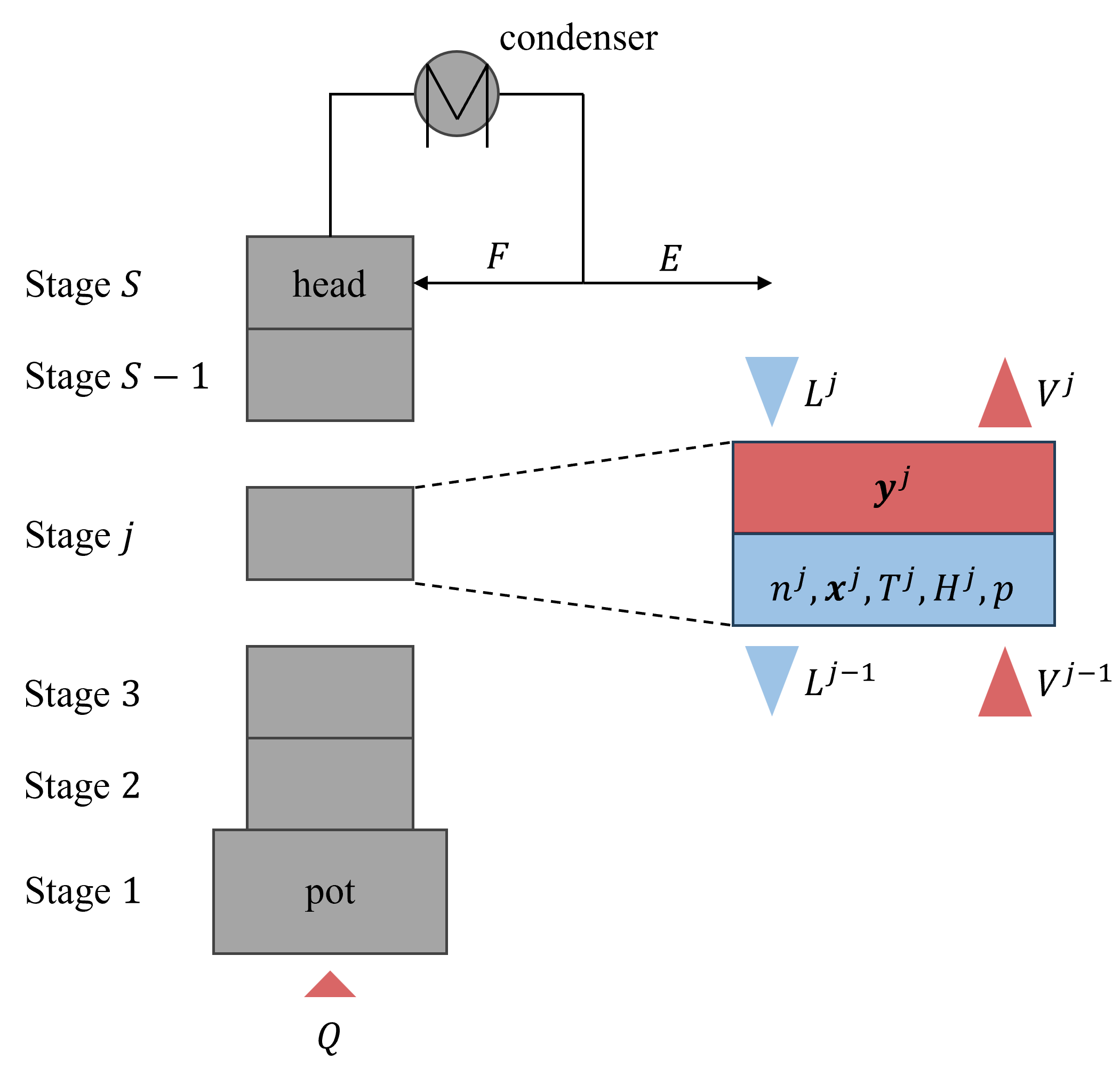}
	\caption{Schematic of the considered batch distillation column}
	\label{fig:distillation_column_schematic}
\end{figure}

\subsection{System equations} 
\label{sec:system-equations}

As usual in stage-wise distillation column modeling, we model the evolution of the batch distillation process 
by means of the mole balances, the enthalpy balances, the thermodynamic equilibrium conditions, and the liquid and vapor summation conditions on every equilibrium stage of the distillation column. In the setting described above, these balance and equilibrium and summation equations can be cast in the form~\eqref{eq:tot-mole,tot-and-modif-comp-mole-balance}-\eqref{eq:holdup,tot-and-modif-comp-mole-balance} below (Proposition~\ref{prop:equivalence-of-basic-and-modified-basic-system}). Specifically, under the assumptions made above, the total mole balances around the column stages take the form
\begin{subequations} \label{eq:tot-mole,tot-and-modif-comp-mole-balance}
	\begin{align}
		\dot{n}^1 & = L^1 - V^1  \label{eq:tot-mole-1,tot-and-modif-comp-mole-balance} \\ 
		\dot{n}^j & = L^j - V^j - L^{j-1} + V^{j-1} \label{eq:tot-mole-j,tot-and-modif-comp-mole-balance} \\
		\dot{n}^S &= -\epsilon V^S - L^{S-1} + V^{S-1}, \label{eq:tot-mole-S,tot-and-modif-comp-mole-balance}   
    \end{align}
\end{subequations}
the component mole balances can be cast in the form 
\begin{subequations} \label{eq:comp-mole,tot-and-modif-comp-mole-balance}
    \begin{align}
    		\dot{x}_i^1 &= \big( L^1 (x_i^2 - x_i^1) -V^1 (y_i^1 - x_i^1) \big)/n^1 \qquad (i\in \{1,\dots,C\})  \label{eq:comp-mole-1,tot-and-modif-comp-mole-balance} \\
        \dot{x}_i^j &= \big( L^j (x_i^{j+1} - x_i^j) - V^j (y_i^j - x_i^j) + V^{j-1} (y_i^{j-1} - x_i^j) \big)/n^j \qquad (i\in \{1,\dots,C\}) \label{eq:comp-mole-j,tot-and-modif-comp-mole-balance} \\
        \dot{x}_i^S &= \big( -\epsilon V^S (y_i^S - x_i^S) + V^{S-1} (y_i^{S-1} - x_i^S) \big)/n^S, \qquad (i\in \{1,\dots,C\}), \label{eq:comp-mole-S,tot-and-modif-comp-mole-balance}
    \end{align}
\end{subequations}
and the enthalpy balances take the form
\begin{subequations} \label{eq:enth,tot-and-modif-comp-mole-balance}
    \begin{align}
    		\dot{H}^1 & = L^1 \fhl(T^2,\bm{x}^2) - V^1 \fhv(T^1,\bm{y}^1) + Q \label{eq:enth-1,tot-and-modif-comp-mole-balance} \\
    		\dot{H}^j & = L^j \fhl(T^{j+1},\bm{x}^{j+1}) - V^j \fhv(T^j,\bm{y}^j) - L^{j-1} \fhl(T^j,\bm{x}^j) + V^{j-1}\fhv(T^{j-1},\bm{y}^{j-1}) \label{eq:enth-j,tot-and-modif-comp-mole-balance} \\
        \dot{H}^S &= (1-\epsilon)V^S \fhl(\Tcond,\bm{y}^S) - V^S \fhv(T^S,\bm{y}^S) \nonumber\\
        &\quad - L^{S-1} \fhl(T^S,\bm{x}^S) + V^{S-1}\fhv(T^{S-1},\bm{y}^{S-1}). \label{eq:enth-S,tot-and-modif-comp-mole-balance} 
    \end{align}
\end{subequations}
Additionally, the summation conditions and the equilibrium- and enthalpy-defining equations take the form
\begin{align}
		\sum_{i=1}^C x_i^j &= 1 \qquad (j \in \{1,\dots,S\}) \label{eq:x-sum,tot-and-modif-comp-mole-balance} \\
        y_i^j &= \fvlei(P,T^j,\bm{x}^j) \qquad (i \in \{1,\dots,C\}, j \in \{1,\dots,S\}) \label{eq:y-def,tot-and-modif-comp-mole-balance} \\
        H^j &= n^j \fhl(T^j,\bm{x}^j) \qquad (j \in \{1,\dots,S\}) \label{eq:enth-def,tot-and-modif-comp-mole-balance}
\end{align}
for every stage $j \in \{1,\dots,S\}$. 
And finally, we have the holdup equations
\begin{align} \label{eq:holdup,tot-and-modif-comp-mole-balance}
		n^j & = \fhold(L^{j-1}), 
\end{align}
which relate the liquid holdup on stage $j \in \{2,\dots,S\}$ to the liquid downstream from that stage.
\revtext{
In the above equations, the models $\fvlei$, $\fhv$, $\fhl$ and $\fhold$ describe, respectively, the vapor-liquid equilibria on the column stages, the vapor and liquid molar enthalpies of the streams in the column, and the liquid holdups on the column stages. All we need, for our solution approach to work, is that these thermodynamic models are explictly defined and continuously differentiable.
\begin{ass} \label{ass:submodels}
We have explicit and continuously differentiable models $\fvlei$, $\fhv$, $\fhl$, and $\fhold$ modeling, respectively,  
\begin{itemize}
\item[(i)] the $i$th vapor mole fraction of a mixture at vapor-liquid equilibrium as a function of $(P,T,\bm{x})$
\item[(ii)] the vapor and liquid molar enthalpies of a stream as a function of $(T,\bm{y})$ and $(T,\bm{x})$, respectively 
\item[(iii)] the liquid holdup on a column stage as a function of the liquid downstream $L$ from that stage.
\end{itemize}
\end{ass}
\noindent In our simulation examples, we define $\fvlei$ in terms of Antoine's and Raoult's extended laws and the non-random two-liquid model (\ref{app:vle}). Also, $\fhl$ and $\fhv$ we define in terms of the liquid heat capacities of the pure components and of the Clausius-Clapeyron approximation (\ref{app:molar_liquid_vapor_enthalpies}). And for $\fhold$, in turn, we use a simple proportionality relation here (\ref{app:holdup}). 
}

It should be noticed that the system equations~\eqref{eq:tot-mole,tot-and-modif-comp-mole-balance}-\eqref{eq:holdup,tot-and-modif-comp-mole-balance} do not contain a vapor summation condition explicitly, but only the liquid summation condition~\eqref{eq:x-sum,tot-and-modif-comp-mole-balance}. 
It is well-known~\cite{seader2006separation} and straightforward to verify that the vapor summation condition
\begin{align}  \label{eq:y-sum,tot-and-modif-comp-mole-balance}
	\sum_{i=1}^C y_i^j &= 1 \qquad (j \in \{1,\dots,S\})
\end{align}
automatically follows from the other equations, namely from the component mole balance equations~\eqref{eq:comp-mole,tot-and-modif-comp-mole-balance} in conjunction with the total mole balance and the liquid summation equations~\eqref{eq:tot-mole,tot-and-modif-comp-mole-balance} and~\eqref{eq:x-sum,tot-and-modif-comp-mole-balance} (Proposition~\ref{prop:equivalence-of-basic-and-modified-basic-system}).  

It should also be noticed that the system equations~\eqref{eq:tot-mole,tot-and-modif-comp-mole-balance}-\eqref{eq:holdup,tot-and-modif-comp-mole-balance} are differential-algebraic equations of a particularly simple form. Specifically, they consist of ordinary differential equations~\eqref{eq:tot-mole,tot-and-modif-comp-mole-balance}-\eqref{eq:enth,tot-and-modif-comp-mole-balance} for the differential variables 
\begin{align} \label{eq:diff-variables,tot-and-modif-comp-mole-balance}
n^j, \quad \bm{x}^j, \quad H^j
\qquad (j \in \{1,\dots,S\})
\end{align}
and of algebraic equations~\eqref{eq:x-sum,tot-and-modif-comp-mole-balance}-\eqref{eq:holdup,tot-and-modif-comp-mole-balance} for the algebraic variables 
\begin{align} \label{eq:alg-variables,tot-and-modif-comp-mole-balance}
V^j,\quad \bm{y}^j, \quad T^j \qquad (j \in \{1,\dots,S\}) 
\qquad \text{and} \qquad
L^j \qquad (j \in \{1,\dots,S-1\}).
\end{align}
In short, the system equations~\eqref{eq:tot-mole,tot-and-modif-comp-mole-balance}-\eqref{eq:holdup,tot-and-modif-comp-mole-balance} are differential-algebraic equations of the following semi-explicit form:
\begin{align}
\dot{\xi} &= f(\xi,\eta, u) \label{eq:ode} \\
0 &= g(\xi,\eta, u) \label{eq:alg}
\end{align}
where $\xi$ denotes the vector of differential variables and $\eta$ denotes the vector of algebraic variables. 
(As usual, a state variable is called a differential variable of a differential-algebraic system iff it shows up with a time derivative somewhere in the differential-algebraic system. A state variable is called algebraic iff it is not differential.) \revtext{Also, 
\begin{align}
u := (\epsilon, P, Q, \Tcond)
\end{align}
denotes the vector of control variables.} 
If the Jacobian $D_\eta g(\xi,\eta,u)$ of the algebraic equations w.r.t.~the algebraic variables was invertible along solutions $t \mapsto (\xi(t),\eta(t))$ of~\eqref{eq:ode}-\eqref{eq:alg}, then the system~\eqref{eq:ode}-\eqref{eq:alg} would be of (differentiation) index $1$ and could therefore be solved efficiently and stably by many standard solvers \revtext{\cite{hairer1991ii}}. (Indeed, if the matrix $D_\eta g(\xi(t),\eta(t),u(t))$ was invertible for every $t$ in the solution interval $I$, then we could reduce the system~\eqref{eq:ode}-\eqref{eq:alg} to an ordinary differential equation for the entire state vector $\zeta = (\xi,\eta)$ simply by differentiating the algebraic equations~\eqref{eq:alg} once along the solution $t \mapsto (\xi(t),\eta(t))$ and by then exploiting the assumed invertibility.) Yet, for the batch distillation equations~\eqref{eq:tot-mole,tot-and-modif-comp-mole-balance}-\eqref{eq:holdup,tot-and-modif-comp-mole-balance}, the Jacobian $D_\eta g(\xi,\eta,u)$ of the algebraic equations w.r.t.~the algebraic variables is glaringly non-invertible \revtext{(Figure~\ref{fig:incidence-matrix-not-index-reduced-for-S=3}). In fact, this Jacobian has} 
\begin{itemize}
\item \revtext{$S$ zero rows because the liquid summation equations~\eqref{eq:x-sum,tot-and-modif-comp-mole-balance} do not depend on any of the algebraic variables~\eqref{eq:alg-variables,tot-and-modif-comp-mole-balance}, and}
\item \revtext{$S$ zero columns because the algebraic vapor stream variables do not enter any of the algebraic equations~\eqref{eq:x-sum,tot-and-modif-comp-mole-balance}-\eqref{eq:holdup,tot-and-modif-comp-mole-balance}.}
\end{itemize}
Consequently, the batch distillation system equations~\eqref{eq:tot-mole,tot-and-modif-comp-mole-balance}-\eqref{eq:holdup,tot-and-modif-comp-mole-balance} are of index $2$ and not amenable to standard solvers \revtext{\cite{hairer1991ii}}.

\begin{figure}
	\centering
	\includegraphics[width=0.5\textwidth]{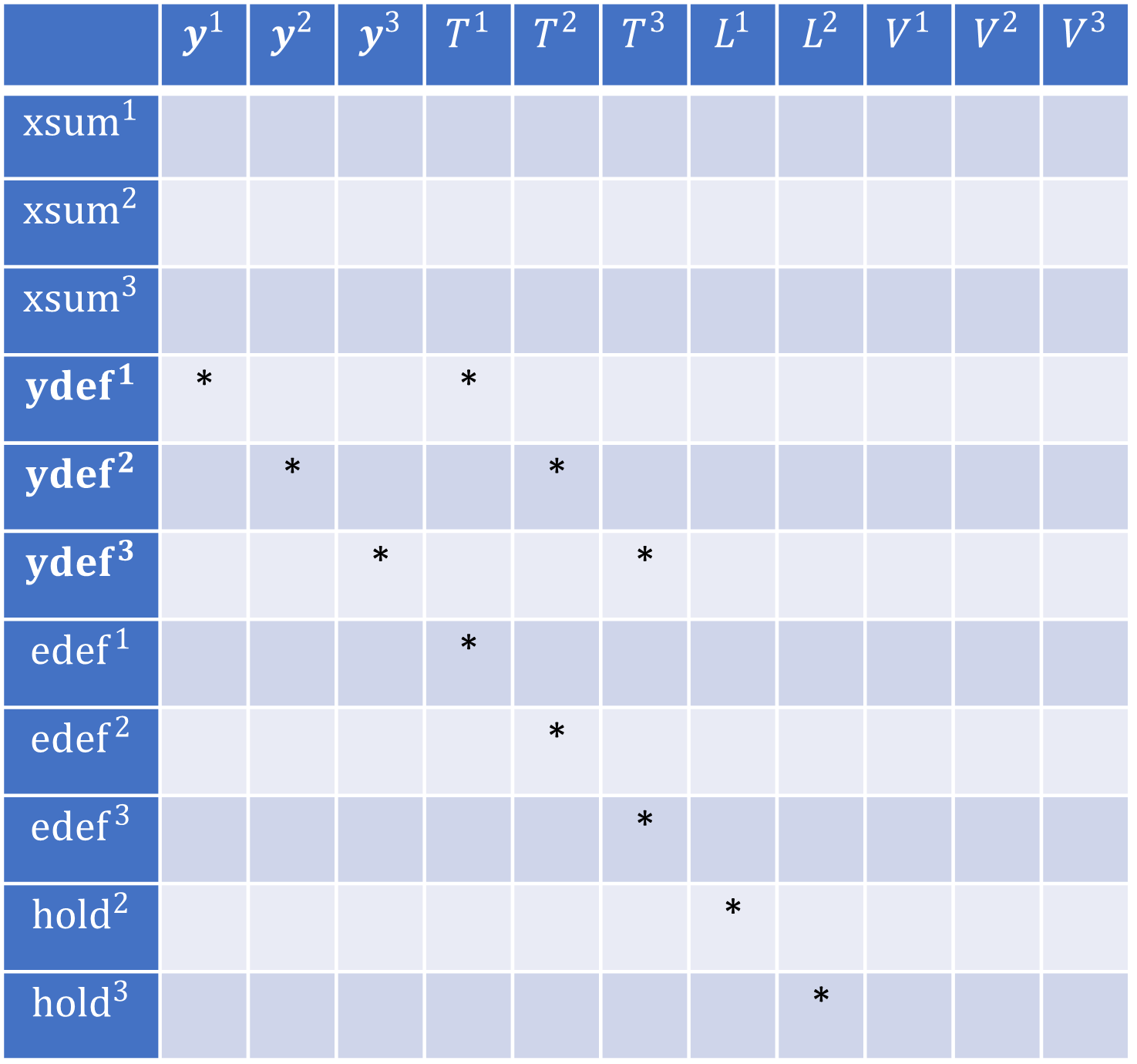}
	\caption{Structure of the Jacobian $D_\eta g(\zeta)$ of the algebraic equations of~\eqref{eq:tot-mole,tot-and-modif-comp-mole-balance}-\eqref{eq:holdup,tot-and-modif-comp-mole-balance} w.r.t.~the algebraic variables in the special case of $S=3$ stages. An empty cell stands for a zero entry or block of that matrix, while a cell with an asterisk stands for a generally non-zero entry or block}
	\label{fig:incidence-matrix-not-index-reduced-for-S=3}
\end{figure}

\subsection{Index reduction}
\label{sec:index-reduction}

In order to stably and efficiently solve the index-$2$ differential-algebraic system equations~\eqref{eq:tot-mole,tot-and-modif-comp-mole-balance}-\eqref{eq:holdup,tot-and-modif-comp-mole-balance}, it is important to turn them into a semi-explicit system of index $1$, that is, a system of the form~\eqref{eq:ode}-\eqref{eq:alg} such that the Jacobian $D_\eta g(\xi,\eta,u)$ becomes invertible along solutions of the system equations \revtext{\cite{hairer1991ii}}. A general possible procedure to achieve this is to differentiate a suitable subset of the algebraic equations and to substitute a subset of the differential equations \revtext{\cite{kunkel2006differential, bachmann1990methods}}. In this manner, one obtains new algebraic equations and new algebraic variables and, when chosen appropriately, these can then remedy the singularity of the original index-$2$ system.

In the literature on batch distillation simulations, there essentially exist two approaches to reduce the index to $1$, which are based on differentiating the vapor summation equations or, respectively,  the liquid summation equations. \revtext{To the best of our knowledge, the first approach was initially introduced in \cite{cervantes1998}. It has since been frequently used in the literature~\cite{cervantes2000, CERVANTES200041, biegler2002, ragunathan2004, eckert2008mathematical, qian2023nonlinear}, whereas} the second approach seems to be much less used. In fact, we found only one reference~\cite{lopez2016rigorous} where this second approach is indicated. In that paper, however, an important singularity issue seems to have been overlooked. At least, this singularity issue is not addressed there at all and, as it stands, the model proposed there is actually still of index $2$ (by the very same arguments as for Proposition~\ref{prop:singularity-of-unperturbed-system} below).      

In the present paper, we propose an index reduction method that is based on differentiating the liquid summation equations as well, but steers clear of the singularity issue in~\cite{lopez2016rigorous}. Specifically, our approach proceeds as follows. 
\revtext{In order to get rid of the zero-row problem (first bullet point at the end of Section~\ref{sec:system-equations}), we drop the component mole balance equations~\eqref{eq:comp-mole,tot-and-modif-comp-mole-balance} for component $C$. As a consequence, the liquid mole fractions $x_C^1, \dots, x_C^S$ are no longer differential variables, but they become algebraic variables and the derivatives of the liquid summation equations~\eqref{eq:x-sum,tot-and-modif-comp-mole-balance} w.r.t.~these new algebraic variables become non-zero. In particular, this fixes the zero-row problem in Figure~\ref{fig:incidence-matrix-not-index-reduced-for-S=3}. In return for the new algebraic variables $x_C^1, \dots, x_C^S$, however, we have to add $S$ new algebraic equations  because otherwise the Jacobian of the algebraic equations w.r.t.~the algebraic variables is no longer a square matrix. In our approach, we obtain these new algebraic equations by differentiating the liquid summation equations~\eqref{eq:x-sum,tot-and-modif-comp-mole-balance} and by inserting, into the resulting equations, the right-hand sides of the component mole balance equations~\eqref{eq:comp-mole,tot-and-modif-comp-mole-balance}.} 
In this manner, we arrive at $S$ new algebraic equations, namely
\begin{subequations} \label{eq:aux-alg-eq,x-sum-differentiation}
	\begin{align}
		\sum_{i=1}^C \big( L^1 (x_i^2 - x_i^1) -V^1 (y_i^1 - x_i^1)\big)/n^1 &= 0 \label{eq:aux-alg-eq-1,x-sum-differentiation} \\
		\sum_{i=1}^C \big( L^j (x_i^{j+1} - x_i^j) - V^j (y_i^j - x_i^j) + V^{j-1} (y_i^{j-1} - x_i^j) \big)/n^j &= 0
\label{eq:aux-alg-eq-j,x-sum-differentiation} \\
		\sum_{i=1}^C \big( \epsilon V^S (x_i^S - y_i^S) + V^{S-1} (y_i^{S-1} - x_i^S) \big)/n^S &= 0 \label{eq:aux-alg-eq-S,x-sum-differentiation}
	\end{align}
\end{subequations}
Summarizing, we arrive at an intermediary system of equations which differs from the original system~\eqref{eq:tot-mole,tot-and-modif-comp-mole-balance}-\eqref{eq:holdup,tot-and-modif-comp-mole-balance} only in that the component mole balance equations are now imposed only for the first $C-1$ components and that it features the new algebraic equations~\eqref{eq:aux-alg-eq,x-sum-differentiation}. \revtext{In short, the intermediary system can be written as a semi-explicit differential-algebraic system
\begin{align}
\dot{\xi} &= \finter(\xi,\eta,u) \label{eq:ode,intermediary}\\
0 &= \ginter(\xi,\eta,u) \label{eq:alg,intermediary}
\end{align}
with differential variables $\xi$ given by~\eqref{eq:diff-variables,tot-and-modif-comp-mole-balance} except $x_C^1, \dots, x_C^S$ and with algebraic variables $\eta$ given by~\eqref{eq:alg-variables,tot-and-modif-comp-mole-balance} plus $x_C^1, \dots, x_C^S$.}  
It is straightforward to see that this intermediary system is equivalent to the original system (Proposition~\ref{prop:equivalence-of-modified-basic-and-unperturbed-system}). 
\revtext{Also, the vapor stream variables $V^1, \dots, V^S$ now explicitly enter the new algebraic equations~\eqref{eq:aux-alg-eq,x-sum-differentiation}, so that the zero-column problem (second bullet point at the end of Section~\ref{sec:system-equations}) is addressed at least at first glance. At a closer look, however, the zero-column problem persists along solutions (Proposition~\ref{prop:singularity-of-unperturbed-system}) and thus} the intermediary system is still singular along solutions. \revtext{We therefore have to regularize the intermediary system, and we propose two approaches to do so: a perturbation-based approach using suitable perturbation parameters (Section~\ref{sec:perturbation-parameter-approach}) and a relaxation-based approach using suitable slack variables (Section~\ref{sec:slack-variable-approach}). In our implementation, both regularization approaches come into play: the first one in the simulation phase and the second one in the initialization phase (Section~\ref{sec:initial-value-problem}).} 

\subsubsection{A perturbation-based approach}
\label{sec:perturbation-parameter-approach}

\revtext{In our first regularization approach, we suitably perturb some of the equations of the intermediary system, using suitable perturbation parameters $\delta_i^j \in [0,\infty)$ for $i \in \{1,\dots,C\}$ and $j \in \{1,\dots,S\}$. Specifically, we perturb the component mole balance equations and the new algebraic equations of the intermediary system~\eqref{eq:ode,intermediary}-\eqref{eq:alg,intermediary} in a physically motivated fashion (\ref{app:physical-motivation}), based on} a suitable relaxation of our zero vapor holdup assumption (Assumption~\ref{ass:setting}~(i)). In this manner, we arrive at the following perturbed system~\eqref{eq:tot-mole,x-sum-differentiation-perturbed}-\eqref{eq:aux-alg-eq,x-sum-differentiation-perturbed}. It consists of the unperturbed total mole balances 
\begin{subequations} \label{eq:tot-mole,x-sum-differentiation-perturbed}
	\begin{align}
		\dot{n}^1 & = L^1 - V^1  \label{eq:tot-mole-1,x-sum-differentiation-perturbed} \\ 
		\dot{n}^j & = L^j - V^j - L^{j-1} + V^{j-1} \label{eq:tot-mole-j,x-sum-differentiation-perturbed} \\
		\dot{n}^S &= -\epsilon V^S - L^{S-1} + V^{S-1}, \label{eq:tot-mole-S,x-sum-differentiation-perturbed}     
    \end{align}
\end{subequations}
the following perturbed component mole balances 
\begin{subequations} \label{eq:comp-mole,x-sum-differentiation-perturbed}
    \begin{align}
		\dot{x}_i^1 & = \big( L^1 (x_i^2 - x_i^1 - \delta_i^1) -V^1 (y_i^1 - x_i^1 - \delta_i^1)\big)/n^1 \qquad (i\in \{1,\dots,C-1\})  \label{eq:comp-mole-1,x-sum-differentiation-perturbed} \\        
        \dot{x}_i^j & = \big( L^j (x_i^{j+1} - x_i^j - \delta_i^j) - V^j (y_i^j - x_i^j - \delta_i^j) + V^{j-1} (y_i^{j-1} - x_i^j - \delta_i^j) \big)/n^j \qquad (i\in \{1,\dots,C-1\}) \label{eq:comp-mole-j,x-sum-differentiation-perturbed} \\
        \dot{x}_i^S &= \big( \epsilon V^S (x_i^S + \delta_i^S - y_i^S) + V^{S-1} (y_i^{S-1} - x_i^S - \delta_i^S) \big)/n^S \qquad (i\in \{1,\dots,C-1\}), \label{eq:comp-mole-S,x-sum-differentiation-perturbed}
    \end{align}
\end{subequations}
and the unperturbed enthalpy balances
\begin{subequations} \label{eq:enth,x-sum-differentiation-perturbed}
    \begin{align} 
    		\dot{H}^1 & = L^1 \fhl(T^2,\bm{x}^2) - V^1 \fhv(T^1,\bm{y}^1) + Q \label{eq:enth-1,x-sum-differentiation-perturbed} \\
    		\dot{H}^j & = L^j \fhl(T^{j+1},\bm{x}^{j+1}) - V^j \fhv(T^j,\bm{y}^j) - L^{j-1} \fhl(T^j,\bm{x}^j) + V^{j-1}\fhv(T^{j-1},\bm{y}^{j-1}) \label{eq:enth-j,x-sum-differentiation-perturbed} \\
        \dot{H}^S &= (1-\epsilon)V^S \fhl(\Tcond,\bm{y}^S) - V^S \fhv(T^S,\bm{y}^S) \nonumber\\
        &\quad - L^{S-1} \fhl(T^S,\bm{x}^S) + V^{S-1}\fhv(T^{S-1},\bm{y}^{S-1}). \label{eq:enth-S,x-sum-differentiation-perturbed} 
    \end{align}
\end{subequations}
Additionally, it includes the liquid summation conditions, the equilibrium- and enthalpy-defining equations, and the holdup equations 
\begin{align}
		x_C^j &= 1 - \sum_{i=1}^{C-1} x_i^j \qquad (j \in \{1,\dots,S\}) \label{eq:x-sum,x-sum-differentiation-perturbed} \\
        y_i^j &= \fvlei(P,T^j,\bm{x}^j) \qquad (i \in \{1,\dots,C\}, j \in \{1,\dots,S\}) \label{eq:y-def,x-sum-differentiation-perturbed} \\
        H^j &= n^j \fhl(T^j,\bm{x}^j) \qquad (j \in \{1,\dots,S\}) \label{eq:enth-def,x-sum-differentiation-perturbed} \\
        n^j & = \fhold(L^{j-1}) \qquad (j \in \{2,\dots,S\}) \label{eq:holdup-j,x-sum-differentiation-perturbed}
\end{align} 
in unperturbed form. And finally, the perturbed system features the new algebraic equations~\eqref{eq:aux-alg-eq,x-sum-differentiation} introduced in the intermediary system in the following perturbed form:
\begin{subequations} \label{eq:aux-alg-eq,x-sum-differentiation-perturbed}
	\begin{align}
		\sum_{i=1}^C \big( L^1 (x_i^2 - x_i^1 - \delta_i^1) -V^1 (y_i^1 - x_i^1 - \delta_i^1)\big)/n^1 &= 0 \label{eq:aux-alg-eq-1,x-sum-differentiation-perturbed} \\
		\sum_{i=1}^C \big( L^j (x_i^{j+1} - x_i^j - \delta_i^j) - V^j (y_i^j - x_i^j - \delta_i^j) + V^{j-1} (y_i^{j-1} - x_i^j - \delta_i^j) \big)/n^j &= 0
\label{eq:aux-alg-eq-j,x-sum-differentiation-perturbed} \\
		\sum_{i=1}^C \big( \epsilon V^S (x_i^S + \delta_i^S - y_i^S) + V^{S-1} (y_i^{S-1} - x_i^S - \delta_i^S) \big)/n^S &= 0. \label{eq:aux-alg-eq-S,x-sum-differentiation-perturbed}
	\end{align}
\end{subequations}
\revtext{Clearly, this perturbed system~\eqref{eq:tot-mole,x-sum-differentiation-perturbed}-\eqref{eq:aux-alg-eq,x-sum-differentiation-perturbed} is again a semi-explicit differential-algebraic system
\begin{align}
\dot{\xi} &= \fpert_\delta(\xi,\eta,u) \label{eq:ode,perturbed}\\
0 &= \gpert_\delta(\xi,\eta,u) \label{eq:alg,perturbed}
\end{align}
whose differential variables $\xi$ and algebraic variables $\eta$ are the same as for the intermediary system~\eqref{eq:ode,intermediary}-\eqref{eq:alg,intermediary}. Also, the perturbed system reduces to the intermediary system in the special case where the perturbation parameters $\delta = (\delta_i^j)$ are chosen to be all zero.}
As shown in Proposition~\ref{prop:regularity-of-perturbed-system}, our perturbed system~\eqref{eq:tot-mole,x-sum-differentiation-perturbed}-\eqref{eq:aux-alg-eq,x-sum-differentiation-perturbed} is indeed non-singular along its solution trajectories if, for instance, the perturbation parameters $\delta_i^j$ are chosen to be all positive and independent of the stage: 
\begin{align}
    \delta_i^1 = \dotsb = \delta_i^S > 0 \qquad (i \in \{1,\dots,C\})
\end{align}
or, more generally, to be increasing with the stage index $j$ as in~\eqref{eq:perturbations-increasingly-ordered}. 
Consequently, under these assumptions, the perturbed system is of index $1$ and can thus be stably and efficiently solved with standard solvers. It should be noticed, however, that the perturbed system is no longer exactly equivalent to the original system equations~\eqref{eq:tot-mole,tot-and-modif-comp-mole-balance}-\eqref{eq:holdup,tot-and-modif-comp-mole-balance}. In fact, along the solutions of the perturbed system, the vapor summation equations~\eqref{eq:y-sum,tot-and-modif-comp-mole-balance} are violated by an amount proportional to the perturbation parameter. Specifically, the vapor summation equations along solutions become~\eqref{eq:ysum-perturbed}.  
As long as the perturbation parameters are chosen to be small, however, the violation of the vapor summation equation is negligible. 

\subsubsection{A relaxation-based approach} 
\label{sec:slack-variable-approach}

\revtext{In our second regularization approach, we suitably relax the vapor summation equations implicit in the intermediary system, 
using suitable slack variables. Specifically, we add to the intermediary system~\eqref{eq:ode,intermediary}-\eqref{eq:alg,intermediary} the following relaxed vapor summation equations
\begin{align} \label{eq:ysum,relaxed}
\sum_{i=1}^C y_i^j = 1 + s^j \qquad (j \in \{1,\dots,S\})
\end{align}
containing $S$ additional algebraic slack variables $s^1, \dots, s^S \in \R$. In this manner, we arrive at a relaxed system consisting of the equations of the intermediary system and the additional algebraic equations~\eqref{eq:ysum,relaxed}. Clearly, this relaxed system is again a semi-explicit differential-algebraic system
\begin{align}
\dot{\xi} &= \frel(\xi,\eta,s,u) \label{eq:ode,relaxed}\\
0 &= \grel(\xi,\eta,s,u) \label{eq:alg,relaxed}
\end{align} 
whose differential variables $\xi$ are the same as for the intermediary system and whose algebraic variables $(\eta,s)$ are those of the intermediary system plus the newly introduced slack variables $s = (s^1,\dots,s^S)$. As shown in Proposition~\ref{prop:regularity-of-relaxed-system}, our relaxed system~\eqref{eq:ode,relaxed}-\eqref{eq:alg,relaxed} is indeed non-singular along  its solution trajectories if the slack variables do not vanish
\begin{align}
s^1(t), \dots, s^S(t) \ne 0
\end{align}
for any $t$ in the solution interval. Consequently, under these assumptions, the relaxed system is of index $1$ and can thus be stably and efficiently solved with standard solvers. 
}

\subsubsection{Comparison with the traditional index-reduction approach} 
\label{sec:comparison-with-traditional-approach}

\revtext{
Compared to the traditional index-reduction approach from the literature (\ref{app:index-reduction-y-sum-differentiation}), our approach offers several remarkable advantages.
}

\revtext{
A clear practical advantage of our approach is that the resulting perturbed and relaxed system equations~\eqref{eq:ode,perturbed}-\eqref{eq:alg,perturbed} and \eqref{eq:ode,relaxed}-\eqref{eq:alg,relaxed} are both  considerably simpler than the alternative system equations~\eqref{eq:tot-mole,y-sum-differentiation}-\eqref{eq:vlesum,y-differentiation} resulting from the traditional approach. Just compare the enthalpy balance equations~\eqref{eq:enth,x-sum-differentiation-perturbed} of our perturbed and relaxed systems with the fairly complicated algebraicized enthalpy balance equations~\eqref{eq:enth,y-sum-differentiation} of the alternative system. In particular, these algebraicized enthalpy balances contain partial derivatives of the  thermodynamic submodels $\fvlei$ and $\fhl$ via~\eqref{eq:a(P,T,x)-def} and~\eqref{eq:b(P,T,x)-def}, whereas our equations only contain the submodels themselves. As a consequence, exchanging these submodels 
is considerably simpler in our approach. 
}

\revtext{
A clear theoretical advantage of our approach is that the conditions under which the resulting perturbed and relaxed system equations~\eqref{eq:ode,perturbed}-\eqref{eq:alg,perturbed} and \eqref{eq:ode,relaxed}-\eqref{eq:alg,relaxed} are non-singular, are very transparent and natural. See the remarks after Proposition~\ref{prop:regularity-of-perturbed-system}. In contrast, the conditions under which the alternative system equations~\eqref{eq:tot-mole,y-sum-differentiation}-\eqref{eq:vlesum,y-differentiation} are non-singular, 
are fairly complicated and intransparent. And, as a consequence, it is not so clear when exactly the traditional index-reduction approach really results in an index-reduced system. See the remarks after Proposition~\ref{prop:regularity-of-alternative-system}. 
}

\revtext{
An apparent disadvantage of our approach is that the perturbed and relaxed system equations~\eqref{eq:ode,perturbed}-\eqref{eq:alg,perturbed} and \eqref{eq:ode,relaxed}-\eqref{eq:alg,relaxed} are not exactly equivalent to the original system equations because the vapor summation equations are slightly violated. As has already been pointed out above, though, this violation is negligible when the perturbation parameters are chosen small, for example, as in~\eqref{eq:perturbation-parameter-choice}. See \ref{app:satisfaction-of-vapor-summation-equation}. 
}

\subsection{Initial-value problem} \label{sec:initial-value-problem}

\revtext{
In this section, we explain how we solve the initial-value problem for the perturbed system~\eqref{eq:ode,perturbed}-\eqref{eq:alg,perturbed} with initial conditions on the total apparatus holdup $\napp$ and on the total apparatus composition $\xapp$, which are defined by
\begin{align}
\napp := n^1 + \dotsb + n^S
\qquad \text{and} \qquad
\xapp := (n^1 \bm{x}^1 + \dotsb + n^S \bm{x}^S)/\napp.
\end{align}
Specifically, we explain our solution strategy for the differential-algebraic initial-value problem
\begin{align}
\dot{\xi} &= \fpert_\delta(\xi,\eta,u) \label{eq:ivp-ode,perturbed}\\
0 &= \gpert_\delta(\xi,\eta,u) \label{eq:ivp-alg,perturbed}
\end{align}
on the time horizon $[0,\tend]$ with initial conditions
\begin{align} \label{eq:ivp-initial-conditions,apparatus}
\napp(0) = \nappinit \qquad \text{and} \qquad \xapp(0) = \xappinit.
\end{align}
In particular, we explain our initialization strategy, 
that is, our strategy of computing good initial guesses for the state variables $\xi$ and $\eta$ of this initial-value problem~\eqref{eq:ivp-ode,perturbed}-\eqref{eq:ivp-initial-conditions,apparatus}. As is well-known~\cite{ascher1998computer, brenan1995numerical, biegler2010nonlinear, may2022optimal, grossmann2005optimal}, good initial guesses are decisive for the robust numerical solution of differential-algebraic initial-value problems. 
In essence, our initialization strategy consists of two parts: a stationary initialization with infinite reflux (Section~\ref{sec:stationary-initialization}) inspired by \cite{fletcher2000initialising} and an instationary initialization with finite reflux (Section~\ref{sec:instationary-initialization}). As is clear from~\eqref{eq:efflux-and-reflux-ratio}, infinite reflux $R = \infty$ is equivalent to zero efflux $\epsilon = 0$. 
} 

\subsubsection{Stationary initialization with infinite reflux}
\label{sec:stationary-initialization}

\color{black}
In our stationary initialization, we compute 
a stationary solution of the intermediary unperturbed system~\eqref{eq:ode,intermediary}-\eqref{eq:alg,intermediary} with initial condition~\eqref{eq:ivp-initial-conditions,apparatus} and with constant control input
\begin{align} \label{eq:u-const-u_0}
u \equiv u_0 := (0,P_0,Q_0,\Tcond_0)
\end{align}
featuring a zero efflux (infinite reflux) value and the values $(P_0,Q_0,\Tcond_0) := (P(0),Q(0),\Tcond(0))$ of the other control inputs prescribed at the beginning $0$ of the simulation phase. In other words, we compute a solution of the stationary system
\begin{gather}
0 = \fpert_0(\xi,\eta,u_0) \label{eq:ode,stationary-unperturbed}\\
0 = \gpert_0(\xi,\eta,u_0) \label{eq:alg,stationary-unperturbed}\\
\napp(0) = \nappinit \qquad \text{and} \qquad \xapp(0) = \xappinit \label{eq:initial-conditions,apparatus,stationary-unperturbed}
\end{gather}
with the constant infinite-reflux control input $u_0$ defined above. In order to do so, we use two loops:
\begin{itemize}
\item In an inner loop, we compute, for every given pot composition candidate $\bm{x}^1$, a solution to the stationary system~\eqref{eq:ode,stationary-unperturbed}-\eqref{eq:alg,stationary-unperturbed} without the extra conditions~\eqref{eq:initial-conditions,apparatus,stationary-unperturbed}.
\item In an outer loop, we iteratively adapt the pot composition candidate $\bm{x}^1$ such that the solution to~\eqref{eq:ode,stationary-unperturbed}-\eqref{eq:alg,stationary-unperturbed} computed in the inner loop also satisfies the extra conditions~\eqref{eq:initial-conditions,apparatus,stationary-unperturbed}.  
\end{itemize}
In detail, the inner and outer loops are described in Algorithms~\ref{alg:inner-loop} and~\ref{alg:outer-loop}. See also \ref{app:infinite-reflux-initialization}. Algorithm~\ref{alg:inner-loop} is well-defined as soon as all the denominators
\begin{align} \label{eq:inner-algo-well-defined}
\fhl(T^{j+1},\bm{y}^j) - \fhv(T^j,\bm{y}^j) \ne 0
\end{align}
in the computation formulas for the vapor streams are non-zero. As explained after Proposition~\ref{prop:output-of-inner-algo-is-stationary-solution}, this is the case under very natural conditions and thus poses no real restrictions. 
\color{black}

\begin{algorithm}
\caption{Inner loop of the stationary initialization} \label{alg:inner-loop}
\revtext{
\textbf{Input:} $\bm{x}^1$ and $P_0, Q_0, \Tcond_0 =: T^{S+1}$
} 
\begin{algorithmic}
\For{$j \in \{1,\dots,S\}$}
\State Compute bubble-point temperature $T^j$ by solving $\sum_{i=1}^C \fvlei(P_0,T^j,\bm{x}^j) = 1$
\State $y_i^j \gets \fvlei(P_0,T^j,\bm{x}^j)$ for all $i \in \{1,\dots,C\}$
\State $\bm{x}^{j+1} \gets \bm{y}^j$
\EndFor
\For{$j \in \{1,\dots,S\}$}
\State Compute the flow rates $V^j$ and $L^j$ by solving the enthalpy balance equation
\If{$j = 1$}
\State \revtext{$V^1 \gets -\frac{Q_0}{\fhl(T^2,\bm{y}^1) - \fhv(T^1,\bm{y}^1)}$}
\Else
\State \revtext{$V^j \gets V^{j-1} \frac{\fhl(T^j,\bm{y}^{j-1}) - \fhv(T^{j-1},\bm{y}^{j-1})}{\fhl(T^{j+1},\bm{y}^j) - \fhv(T^j,\bm{y}^j)}$}
\EndIf
\State \revtext{$L^j \gets V^j$}
\EndFor
\For{$j \in \{2,\dots,S\}$}
\State Compute the stage holdup $n^j$ and the liquid enthalpy $H^j$ 
\State $n^j \gets \fhold(L^{j-1})$
\State \revtext{$H^j \gets n^j \fhl(T^j,\bm{x}^j)$}
\EndFor
\end{algorithmic}
\revtext{
\textbf{Output:} a solution to~\eqref{eq:ode,stationary-unperturbed}-\eqref{eq:alg,stationary-unperturbed} 
}
\end{algorithm}

\color{black}
\noindent Algorithm~\ref{alg:outer-loop} computes a solution to~\eqref{eq:ode,stationary-unperturbed}-\eqref{eq:initial-conditions,apparatus,stationary-unperturbed} by solving the constrained optimization problem
\begin{align} \label{eq:outer-algo}
\min_{\bm{x}^1 \in \mathcal{X}} \big| \xapp(\bm{x}^1) - \xappinit \big|^2 \quad \text{s.t.} \quad n^1(\bm{x}^1) \ge 0.
\end{align}
In this problem, $\mathcal{X} := \{ \bm{x} \in [0,1]^C: \sum_{i=1}^C x_i = 1\}$ is the set of admissible composition vectors and the quantities $n^1(\bm{x}^1)$ and $\xapp(\bm{x}^1)$ are defined as follows: 
\begin{align} \label{eq:n^1-and-xapp-as-functions-of-x^1}
n^1(\bm{x}^1) := \nappinit - \sum_{j=2}^S n^j(\bm{x}^1)
\qquad \text{and} \qquad
\xapp(\bm{x}^1) := \sum_{j=1}^S n^j(\bm{x}^1) \bm{x}^j(\bm{x}^1) / \nappinit,
\end{align} 
where $\bm{x}^j(\bm{x}^1)$ and $n^j(\bm{x}^1)$ for $j \in \{2,\dots,S\}$ denote the results of Algorithm~\ref{alg:inner-loop} with input $\bm{x}^1$. We implemented Algorithm~\ref{alg:outer-loop} in python using the constrained optimization solver COBYLA \cite{powell1994direct} within the package NLopt \cite{NLopt}.
\color{black}

\begin{algorithm}
\caption{Outer loop of the stationary initialization} \label{alg:outer-loop}
\revtext{
\textbf{Input:} $\nappinit$, $\xappinit$, $P_0, Q_0, \Tcond_0$, an initial pot composition candidate $\bm{x}_0^1$ 
}
\begin{algorithmic}
\State $\text{tol} \gets 10^{-10}$
\State $\text{res} \gets \infty$
\State \revtext{$\text{pen} \gets 10^2$}
\State $\bm{x}^1 \gets \bm{x}^1_0$
\While{$\text{res} > \text{tol}$}
\State Compute stage holdups $n^j = n^j\left(\bm{x}^1\right)$ for $j \in \{2,\dots,S\}$ with Algorithm~\ref{alg:inner-loop}
\State $n^1 \gets \nappinit - \sum_{j=2}^S n^j$
\State Compute apparatus composition $\xapp = \xapp(\bm{x}^1)$ with Algorithm~\ref{alg:inner-loop} using~\eqref{eq:n^1-and-xapp-as-functions-of-x^1} 
\If{$n^1 \ge 0$}
\State $\text{res} \gets \big| \xapp(\bm{x}^1) - \xappinit \big|^2$
\State \revtext{Update $\bm{x}^1$ within $\mathcal{X}$ using COBYLA}	
\ElsIf{$n^1 < 0$}
\State $\text{res} \gets \big| \xapp(\bm{x}^1) - \xappinit \big|^2 + \text{pen}$
\State \revtext{Update $\bm{x}^1$ within $\mathcal{X}$ using COBYLA}
\EndIf
\EndWhile
\end{algorithmic}
\revtext{
\textbf{Output:} a solution to~\eqref{eq:ode,stationary-unperturbed}-\eqref{eq:initial-conditions,apparatus,stationary-unperturbed}
}
\end{algorithm}

\subsubsection{Instationary initialization with finite reflux}
\label{sec:instationary-initialization}

\revtext{
In our instationary initialization, we compute initial guesses for the state variables $\xi$ and $\eta$ of the initial-value problem~\eqref{eq:ivp-ode,perturbed}-\eqref{eq:ivp-initial-conditions,apparatus} to be solved. As these initial guesses, however, we cannot simply take the result $(\xist,\etast)$ of the stationary initialization (Section~\ref{sec:stationary-initialization}). After all, $(\xist,\etast)$ is a solution to the unperturbed system with infinite reflux, but the initial-value problem to be solved involves the perturbed system with finite reflux. In order to smoothly go 
from the unperturbed system with infinite reflux to the perturbed system with finite reflux, we use the slack variables of our relaxed system~\eqref{eq:ode,relaxed}-\eqref{eq:alg,relaxed}. Specifically, we solve a suitable boundary-value problem for the relaxed system, namely 
\begin{align}
\dot{\xi} &= \frel(\xi,\eta,s,\epsilon,\hat{u}) \label{eq:bvp-ode,relaxed}\\
0 &= \grel(\xi,\eta,s,\epsilon,\hat{u}) \label{eq:bvp-alg,relaxed}
\end{align} 
on an initialization time horizon $[\tpre,0]$ with the following boundary conditions on the efflux and on the apparatus holdup and composition:
\begin{gather}
\epsilon(\tpre) = 0 \qquad \text{and} \qquad \epsilon(0) = \epsilon_0 \label{eq:bvp-boundary-condition,efflux} \\
\napp(0) = \nappinit \qquad \text{and} \qquad \xapp(0) = \xappinit. \label{eq:bvp-boundary-condition,apparatus}
\end{gather} 
Similar to the stationary initialization, $\epsilon_0 := \epsilon(0)$ and 
\begin{align}
\hat{u} \equiv \hat{u}_0 := (P_0,Q_0,\Tcond_0) := (P(0),Q(0),\Tcond(0))
\end{align}
denote the values of the control input prescribed at the beginning of the simulation phase $[0,\tend]$. It should be noticed that in the boundary-value problem~\eqref{eq:bvp-ode,relaxed}-\eqref{eq:bvp-boundary-condition,apparatus}, the efflux ratio $\epsilon$ is a free state variable and not a prescribed control input like in~\eqref{eq:ivp-ode,perturbed}-\eqref{eq:ivp-initial-conditions,apparatus}. 
As initial guesses for the state variables $\xi$ and $\eta$ 
at time $\tpre$, we use the results $\xist$ and $\etast$ of the stationary initialization (while as an initial guess for the state variable $s$, we take $0$, of course). 
We solve the boundary-value problem~\eqref{eq:bvp-ode,relaxed}-\eqref{eq:bvp-boundary-condition,apparatus} using the packages pyomo and pyomo.dae \cite{bynum2021pyomo, hart2011pyomo, Nicholson2018} in conjunction with knitro~\cite{byrd2006k}. In particular, this means that the boundary-value problem is treated as the constraints of a dynamic optimization problem with a dummy objective function $0$. In our simulation examples (Section~\ref{sec:study}), we set tight bounds on the slack variables as additional constraints. 
}

\subsubsection{Solving the initial-value problem}
\label{sec:solving-the-ivp}

\revtext{
In order to solve the initial-value problem~\eqref{eq:ivp-ode,perturbed}-\eqref{eq:ivp-initial-conditions,apparatus}, we use the results $\xiin$ and $\etain$ of the instationary initialization (Section~\ref{sec:instationary-initialization}) in a suitable manner. Specifically, instead of~\eqref{eq:ivp-ode,perturbed}-\eqref{eq:ivp-initial-conditions,apparatus} with its apparatus-based initial conditions, we consider the auxiliary initial-value problem
\begin{align}
\dot{\xi} &= \fpert_\delta(\xi,\eta,u) \label{eq:ivp-ode,perturbed,auxiliary}\\
0 &= \gpert_\delta(\xi,\eta,u) \label{eq:ivp-alg,perturbed,auxiliary}
\end{align}
with initial conditions directly on the differential state variables $\xi$, that is,
\begin{align} \label{eq:ivp-initial-conditions,auxiliary}
\xi(0) = \xi^0. 
\end{align}
As the initial value $\xi^0$ for the differential variables, we take 
the value $\xi^0 := \xiin(0)$ of the solution to the boundary-value problem~\eqref{eq:bvp-ode,relaxed}-\eqref{eq:bvp-boundary-condition,apparatus} at time $0$. Similarly, we take the value $\etain(0)$ as an initial guess for the algebraic variables of~\eqref{eq:ivp-ode,perturbed,auxiliary}-\eqref{eq:ivp-alg,perturbed,auxiliary} at time $0$. In view of 
the boundary condition~\eqref{eq:bvp-boundary-condition,apparatus}, it is clear that every solution of the auxiliary initial-value problem~\eqref{eq:ivp-ode,perturbed,auxiliary}-\eqref{eq:ivp-initial-conditions,auxiliary} is also a solution to the original problem~\eqref{eq:ivp-ode,perturbed}-\eqref{eq:ivp-initial-conditions,apparatus} with its apparatus-based initial conditions. It should be noticed that, from a practical point of view, only the apparatus-based initial conditions are of relevance -- just because the initial apparatus conditions are known in real world, while the initial stage conditions (stage holdups, compositions, and enthalpies) making up the differential state variables $\xi$ are usually not. 
We solve the auxiliary initial-value problem~\eqref{eq:ivp-ode,perturbed,auxiliary}-\eqref{eq:ivp-initial-conditions,auxiliary} using the packages pyomo and pyomo.dae \cite{bynum2021pyomo, hart2011pyomo, Nicholson2018} in conjunction with ipopt~\cite{wachter2006implementation}. In our implementation, we make a uniform choice for the perturbation parameters $\delta = (\delta_i^j)$ of~\eqref{eq:ivp-ode,perturbed,auxiliary}-\eqref{eq:ivp-initial-conditions,auxiliary}, namely
\begin{align} \label{eq:perturbation-parameter-choice}
\delta_i^j := 10^{-6} \qquad (i \in \{1,\dots,C\}, j \in \{1,\dots,S\}). 
\end{align}
Alternatively, and a bit more systematically, one could also choose the perturbation parameter values $\delta_i^j$ with the help of the results $\etain$ and $s^{\mathrm{in}}$ of the instationary initialization. (Specifically, one could choose the $\delta_i^j$ such that the vapor summation equations are continuous at the transition from the initialization phase to the simulation phase, that is, such that $1+s^{\mathrm{in}}(0)$ is equal to the right-hand side of~\eqref{eq:ysum-perturbed} with the stream variable values $L^j$ and $V^j$ taken from $\etain(0)$.) At least in our simulation examples (Section~\ref{sec:study}), however, the simple uniform choice~\eqref{eq:perturbation-parameter-choice} works just as well. 
}

\section{Case study} \label{sec:study}

In this section, we discuss the effects of multiplicities in activity coefficients on batch distillation simulations. Specifically, we investigate to what extent stationarily indistinguishable activity coefficient parameter sets become distinguishable when used in a dynamic simulation. We first consider the vapor-liquid equilibrium and activity coefficient diagrams for three mixtures, and then proceed to present the dynamic results for these mixtures, meaning the batch distillation simulation results. We then conclude our study by discussing the effects of different activity coefficient parameter sets on the batch distillation simulation results. 

\subsection{Stationary results} \label{sec:stationary_results}

\revtext{As we will see in the following}, alternative NRTL parameters, resulting from the multiplicities, can lead to significantly different simulation results of a batch distillation process. To highlight these findings, we therefore consider a ternary, a quaternary mixture, and a mixture with five components, all including both zeotropic and azeotropic submixtures.
The first example mixture is acetone (1), methanol (2), and butanol (3); the second example mixture is acetone (1), methanol (2), butanol (3), and chloroform (4); and the third example mixture is 2-butanone (1), acetic acid (2), ethanol (3), ethyl acetate (4), toluene (5).

Applying the parameter estimation method and the solver used in~\cite{werner2023multiplicities}, we obtain (practically) indistinguishable NRTL parameter estimates that differ only within an objective value of $10^{-5}$ (see \eqref{eq:NRTL-parameter-estimation-problem}) for some of the binary submixtures of the three mixtures considered here, see Table~\ref{tab:nrtl_parameters}.

\begin{table}[h]
    \centering
        \begin{tabular}{|c|c|c|c|c|c|c|}
            \hline
            Submixture & type & $a_{ij}$ & $b_{ij}$ & $a_{ji}$ & $b_{ji}$ & parameters \\
            \hline
            \multicolumn{7}{|c|}{\textbf{Example 1}} \\
            \hline
            (1)-(2) & azeotropic & $0.0$ & $101.89$ & $0.0$ & $114.131$ & $\pref^{\{1,2\}}$ \\
            (1)-(2) & azeotropic & $-24.908$ & $8304.053$ & $25.36$ & $-8236.672$ & $\palt^{\{1,2\}}$\\ 
            (1)-(3) & zeotropic & $-8.888$ & $3077.28$ & $10.298$ & $-3326.5$ & $\pref^{\{1,3\}}$ \\ 
            (1)-(3) & zeotropic & $0.302$ & $-73.879$ & $-1.546$ & $730.087$ & $\palt^{\{1,3\}}$ \\
            (2)-(3) & zeotropic & $2.22$ & $-337.71$ & $-1.516$ & $242.624$ & $\pref^{\{2,3\}}$ \\
            (2)-(3) & zeotropic & $-4.747$ & $2061.016$ & $2.307$ & $-1074.57$ & $\palt^{\{2,3\}}$ \\
            \hline
            \multicolumn{7}{|c|}{\textbf{Example 2}} \\
            \hline
            (1)-(4) & azeotropic & $0.965$ & $-590.026$ & $0.538$ & $-106.422$ & $\pref^{\{1,4\}}$ \\
            (1)-(4) & azeotropic & $1.371$ & $-1021.811$ & $-1.420$ & $1143.750$ & $\palt^{\{1,4\}}$ \\
            (2)-(4) & azeotropic & $0.0$ & $-71.903$ & $0.0$ & $690.066$ & $\pref^{\{2,4\}}$ \\
            (3)-(4) & zeotropic & $0.921$ & $-410.590$ & $-4.426$ & $1899.050$ & $\pref^{\{3,4\}}$ \\
            \hline
            \multicolumn{7}{|c|}{\textbf{Example 3}} \\
            \hline
            (1)-(2) & zeotropic & $0.0$ & $544.662$ & $0.0$ & $-284.699$ & $\pref^{\{1,2\}}$ \\
            (1)-(3) & azeotropic & $0.7593$ & $-132.99$ & $-1.5609$ & $654.555$ & $\pref^{\{1,3\}}$ \\
            (1)-(4) & azeotropic & $0.0$ & $-105.570$ & $0.0$ & $173.945$ & $\pref^{\{1,4\}}$ \\
            (1)-(5) & zeotropic & $2.0126$ & $-804.3$ & $-2.7474$ & $1185.26$ & $\pref^{\{1,5\}}$ \\
            (2)-(3) & zeotropic & $0.0$ & $-252.482$ & $0.0$ & $225.476$ & $\pref^{\{2,3\}}$ \\
            (2)-(3) & zeotropic & $-2.74$ & $1572.853$ & $0.89$ & $-760.324$ & $\palt^{\{2,3\}}$ \\
            (2)-(4) & zeotropic & $0.0$ & $-235.279$ & $0.0$ & $515.821$ & $\pref^{\{2,4\}}$ \\
            (2)-(5) & azeotropic & $0.29395$ & $0.0$ & $1.60112$ & $0.0$ & $\pref^{\{2,5\}}$ \\
            (3)-(4) & azeotropic & $-1.151$ & $ 524.424$ & $ -0.243$ & $282.956$ & $\pref^{\{3,4\}}$\\
            (3)-(5) & azeotropic & $1.1459$ & $-113.466$ & $-1.7221$ & $992.737$ & $\pref^{\{3,5\}}$ \\
            (4)-(5) & zeotropic & $-3.544$ & $1438.45$ & $0.298$ & $-160.310$ & $\pref^{\{4,5\}}$\\
            \hline
        \end{tabular}
    \caption{NRTL parameters for examples 1-3; note that example 2 also includes the components 1, 2, and 3 from example 1. The reference NRTL parameters for examples 1-3 were obtained through Aspen}
    \label{tab:nrtl_parameters}
\end{table}

In essence, the parameter estimation method from~\cite{werner2023multiplicities} proceeds in the following steps for every binary submixture $\{i,j\} \subset \{1,\dots, C\}$ of the considered $C$-component mixture: 
\begin{itemize}
    \item A reference parameter value $\pref^{\{i,j\}} = (a_{ij}^{\mathrm{ref}}, b_{ij}^{\mathrm{ref}}, a_{ji}^{\mathrm{ref}}, b_{ji}^{\mathrm{ref}})$ is taken from the literature
    \item The binary mixture is considered at $n+1$ different compositions $\bm{x}^{(0)}, \dots, \bm{x}^{(n)}$, namely $\bm{x}^{(l)} := x_i^{(l)} \bm{e}_i + x_j^{(l)} \bm{e}_j := l/n \cdot \bm{e}_i + (1-l/n) \cdot  \bm{e}_j$ for $l \in \{0,\dots, n\}$
    \item The bubble-point temperature $T^{(l)} := T^{\pref^{\{i,j\}}}(\bm{x}^{(l)})$ of the binary mixture with composition $\bm{x}^{(l)}$ is computed using the reference parameter $\pref^{\{i,j\}}$. In order to do so, the vapor summation equation
    \begin{align}
        1 = \sum_{k=1}^C y_k^{\pref^{\{i,j\}}}(P,T,\bm{x}^{(l)}) = y_i^{\pref^{\{i,j\}}}(P,T,\bm{x}^{(l)}) + y_j^{\pref^{\{\i,j\}}}(P,T,\bm{x}^{(l)})
    \end{align}
    is solved for $T$, where $y_k^{\pref^{\{i,j\}}}(P,T,x) := \frac{P_{\text{sat}, k}(T)}{P} \gamma_k^{\pref^{\{i,j\}}}(T, \bm{x}) x_k$ as in~\eqref{eq:Raoult}.
    \item With the solution method described in~\cite{werner2023multiplicities}, the least-squares parameter estimation problem 
    \begin{align} \label{eq:NRTL-parameter-estimation-problem}
        \min_{\theta^{\{i,j\}} \in \Theta} \frac{1}{n} \sum_{l=0}^n \sum_{k \in \{i,j\}} \big( \ln \gamma_k^{\theta^{\{i,j\}}}(T^{(l)}, \bm{x}^{(l)}) - \ln \gamma_k^{\pref^{\{i,j\}}}(T^{(l)}, \bm{x}^{(l)})  \big)^2 
    \end{align}
    is solved for the pair-interaction parameters $\theta^{\{i,j\}} = (a_{ij}, b_{ij}, a_{ji}, b_{ji})$ from a reasonable parameter set $\Theta$ (while the non-randomness parameter $\alpha_{ij} := 0.3$ is always fixed to $0.3$ as recommended in~\cite{renon1968local}). In this case study, we set $n=100$ for all binary submixtures. 
\end{itemize}
Clearly, the reference parameter value $\pref^{\{i,j\}}$ is one solution to the parameter estimation problem~\eqref{eq:NRTL-parameter-estimation-problem}, but for some submixtures $\{i,j\}$ of some mixtures, there are more solutions than this trivial solution. In these cases, the parameter estimation problem~\eqref{eq:NRTL-parameter-estimation-problem} has multiple solutions or, in other words, these solutions are (practically) indistinguishable parameter estimates. Applying the method sketched above to the $3$-, $4$-, and $5$-component systems considered here, we obtain practically indistinguishable pair-interaction parameters $\pref^{\{i,j\}}$ and $\palt^{\{i,j\}}$ for the binary submixture(s) 
\begin{itemize}
    \item $\{1,2\}$, $\{1,3\}$, $\{2,3\}$ in the case of the $3$-component system
    \item $\{1,4\}$ in the case of the $4$-component system
    \item $\{2,3\}$ in the case of the $5$-component system.
\end{itemize}
See Table~\ref{tab:nrtl_parameters} for these indistinguishable reference and alternative parameter sets $\pref^{\{i,j\}}$ and $\palt^{\{i,j\}}$, respectively. Specifically, the sum of squared errors for the alternative parameter sets $\palt^{\{i,j\}}$ is less than $10^{-5}$ in all considered examples. 
It is clear from these practically zero sums of squared errors and from the relatively large number $n+1$ of different compositions taken into account in~\eqref{eq:NRTL-parameter-estimation-problem} that for the binary submixtures with indistinguishable interaction parameters, the binary (logarithmic) activity coefficient curves
\begin{align} \label{eq:ln-gamma-curves}
    \mathcal{X}_{ij} \ni \bm{x} \mapsto \ln\gamma_k^{\pref^{\{i,j\}}}(T^{\pref^{\{i,j\}}}(\bm{x}),\bm{x})
    \qquad \text{and} \qquad
    \mathcal{X}_{ij} \ni \bm{x} \mapsto \ln\gamma_k^{\palt^{\{i,j\}}}(T^{\palt^{\{i,j\}}}(\bm{x}),\bm{x})
\end{align}
are practically indistinguishable as well for $k \in \{i,j\}$. In fact, this is precisely what the bottom diagrams of Figures~\ref{fig:stationary_diagrams} and~\ref{fig:stationary_diagrams2} demonstrate. In the above relation, 
\begin{align}
\mathcal{X}_{ij} := \{t \bm{e}_i + (1-t) \bm{e}_j: t \in [0,1]\}
\end{align}
denotes the subset of all $C$-component composition vectors that actually involve only component $i$ and component $j$. \revtext{Here $\bm{e}_i$ denotes the unit vector in $\mathbb{R}^C$ whose entries are all zero except the $i$th entry which is one.} So, in the $3$-component special case, $\mathcal{X}_{ij}$ is just one side of the ternary composition triangle (simplex). Apart from the indistinguishability of the logarithmic activity coefficient curves, we also find that the bubble-point curves 
\begin{align} \label{eq:bubble-point-temperature-curves}
    \mathcal{X}_{ij} \ni \bm{x} \mapsto T^{\pref^{\{i,j\}}}(\bm{x})
    \qquad \text{and} \qquad
    \mathcal{X}_{ij} \ni \bm{x} \mapsto T^{\palt^{\{i,j\}}}(\bm{x})
\end{align}
and the vapor composition curves
\begin{align}
    \mathcal{X}_{ij} \ni \bm{x} \mapsto \bm{y}^{\pref^{\{i,j\}}}(\bm{x})
    \qquad \text{and} \qquad
    \mathcal{X}_{ij} \ni \bm{x} \mapsto \bm{y}^{\palt^{\{i,j\}}}(\bm{x})
\end{align}
based on the reference or alternative parameter sets $\pref^{\{i,j\}}$ or $\palt^{\{i,j\}}$, respectively, are practically indistinguishable as well. And hence the same holds true for the dew-point curves, see the top diagrams of Figures~\ref{fig:stationary_diagrams} and~\ref{fig:stationary_diagrams2}. In fact, as the figures show, the bubble- and dew-point curves are even less distinguishable than the activity coefficient curves. This is remarkable since the bubble- and dew-point temperatures are not (explicitly) taken into account in the parameter estimation problem~\eqref{eq:NRTL-parameter-estimation-problem}, whereas the activity coefficients obviously are.

\begin{figure}[H]
	\centering
	\makebox[\textwidth]{\includegraphics[width=\textwidth]{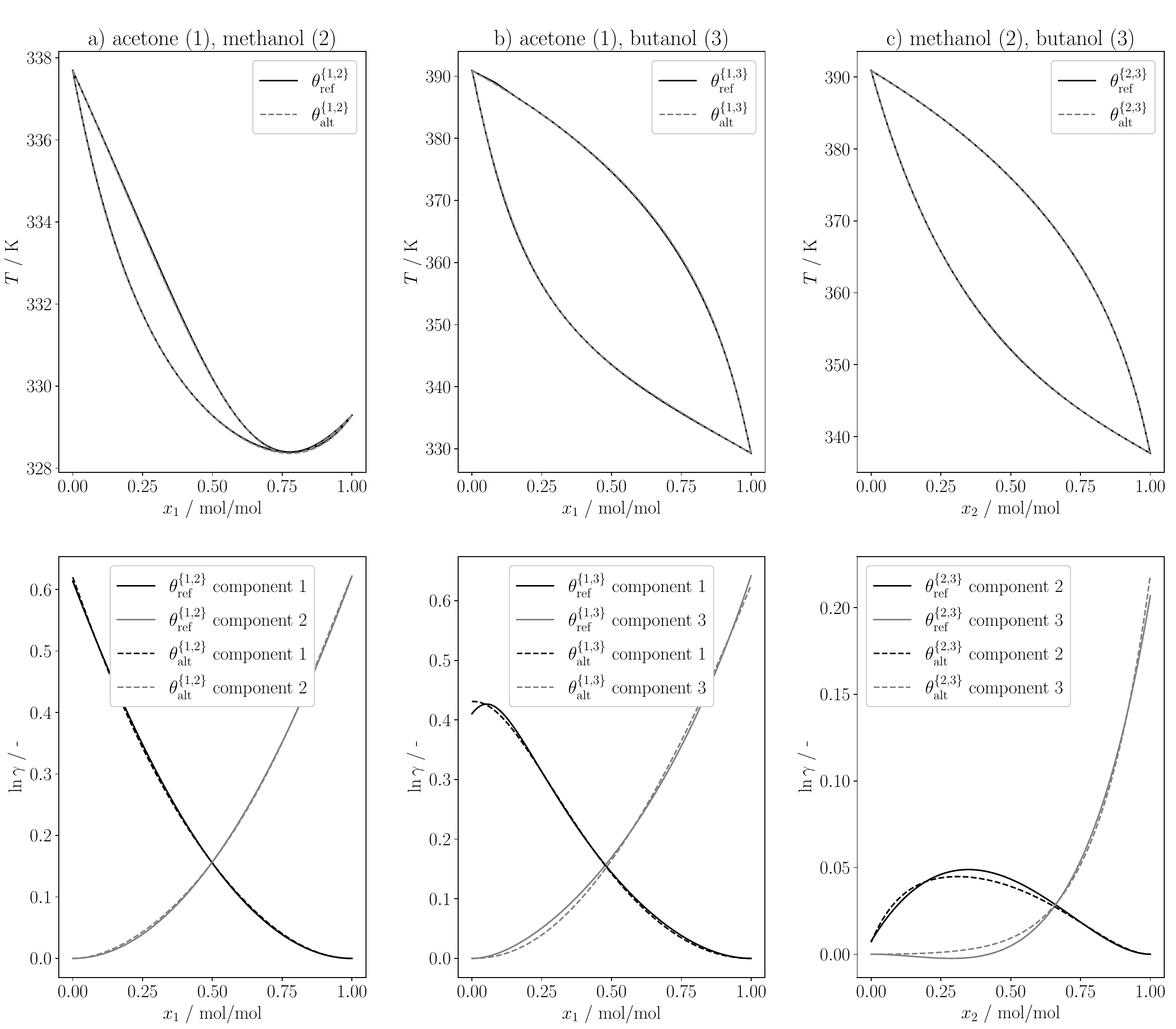}}
	\caption{VLE diagrams~\eqref{eq:bubble-point-temperature-curves} and logarithmic activity coefficient diagrams ~\eqref{eq:ln-gamma-curves} of binary submixtures a) (1)-(2), b) (1)-(3), and c) (2)-(3) of mixture 1; comparison between reference and alternative NRTL parameters}
	\label{fig:stationary_diagrams}
\end{figure}

\begin{figure}[H]
	\centering
	\makebox[\textwidth]{\includegraphics[width=\textwidth]{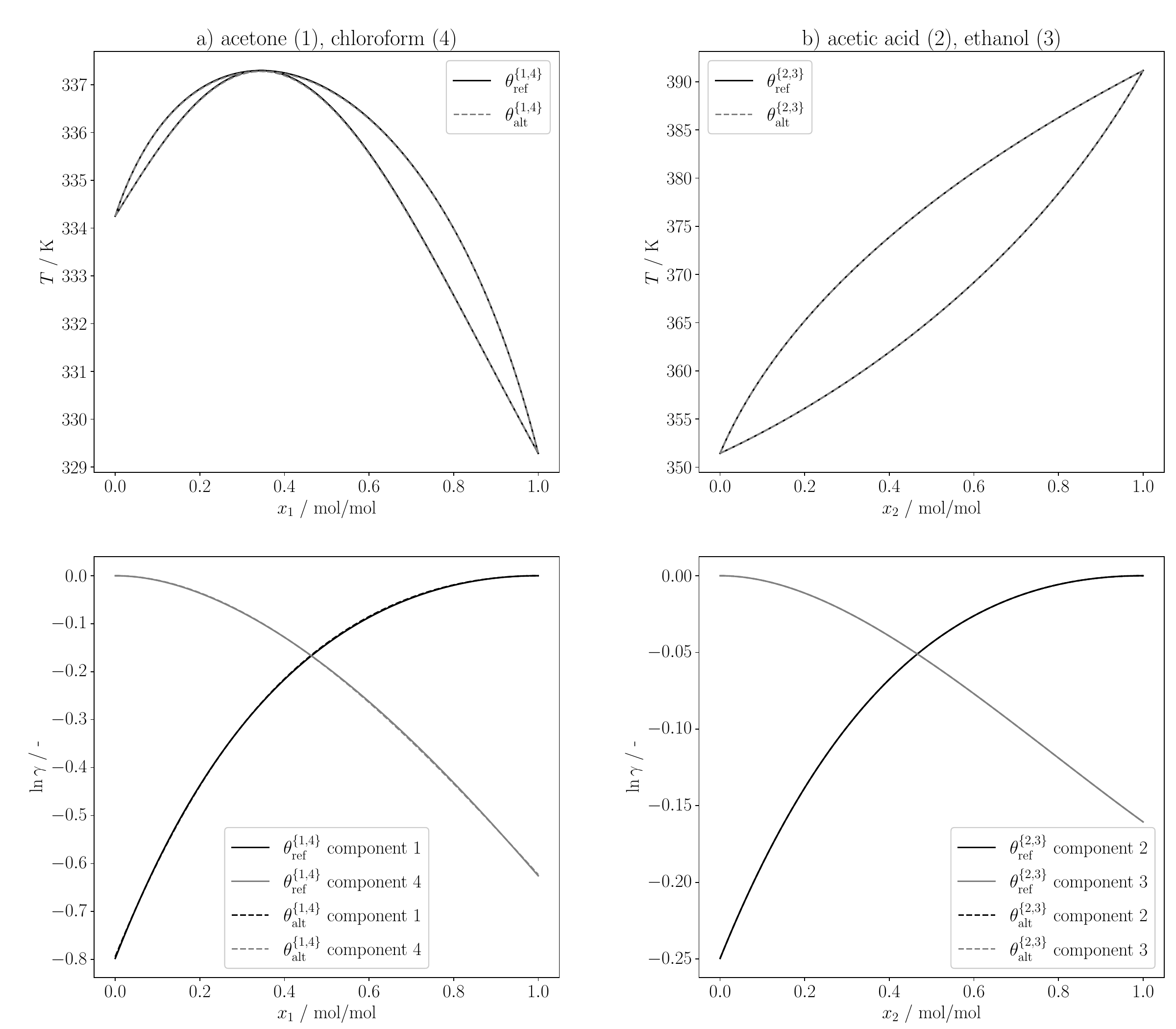}}
	\caption{VLE diagrams~\eqref{eq:bubble-point-temperature-curves} and logarithmic activity coefficient diagrams~\eqref{eq:ln-gamma-curves} of binary submixture a) (1)-(4) of mixture 2, and b) (2)-(3) of mixture 3; comparison between reference and alternative NRTL parameters}
	\label{fig:stationary_diagrams2}
\end{figure}

\subsection{Dynamic results} \label{sec:dynamic_results}

In this subsection we discuss the batch distillation simulation results for the three different examples described in the previous subsection using alternative NRTL parameters $\palt^{\{i,j\}}$ for some of the submixtures.

All simulations shown in this section are solved using knitro \cite{byrd2006k} during the initialization boundary problem (see Section \ref{sec:instationary-initialization}) on a negative time span $[-20, 0]$ s, and ipopt \cite{wachter2006implementation} during the simulation phase (see Section~\ref{sec:solving-the-ivp}). After initialization, the simulation is solved on the time horizon of $[0, 10000]$ s, which is divided into smaller sub-intervals of \SI{20}{s}. On each of those sub-intervals we use ipopt to solve the DAE system using the orthogonal collocation method \cite{cuthrell1989simultaneous} with the Lagrange-Radau discretization scheme, $12$ finite elements and $3$ collocation points. We furthermore set the ipopt option parameters mu\_init and tol to $10^{-3}$ and $10^{-5}$, respectively, and choose appropriate scaling factors for the variables and equations in our DAE system. Assuming the current sub-interval to be $[t_i, t_{i+1}]$, we initialize the variables at all discrete time steps\footnote{i. e. the union of all collocation points and the beginning and end of the finite element intervals} of the next sub-interval with the values of the current sub-interval at time $t_{i+1}$. The simulation stops as soon as the liquid molar holdup  $n^1$ in the pot is smaller than \SI{0.75}{\mole}, or one of the liquid mole fractions $x_i^j$ on one of the stages is smaller than \SI{e-12}{\mole\per\mole}.
The number of stages is set to 10 for all performed simulations. The control input values and initial values we use throughout the different simulations in this study are listed in Table \ref{tab:sim_parameters_example1}.

\begin{table}[h]
    \centering
        \begin{tabular}{|c|c|c|}
            \hline
            Control input / initial value & Unit & Value \\
            \hline
            pressure $P$ & Pa & $101330$ \\
            heat $Q$ & W & $2000$ \\ 
            efflux ratio $\epsilon$ & - & $0.5$ \\
            condenser temperature $\Tcond$ & K & $298.15$ \\
            initial total moles $\nappinit$ mixture 1 & mol & $20.717$ \\
            initial composition $\xappinit$ mixture 1 & [mol/mol, mol/mol, mol/mol] & $(0.3, 0.5, 0.2)$ \\
            initial total moles $\nappinit$ mixture 2 & mol & $13.145$ \\
            initial composition $\xappinit$ mixture 2 & [mol/mol, ..., mol/mol] & $(0.2, 0.2, 0.3, 0.3)$ \\
            initial total moles $\nappinit$ mixture 3 & mol & $14.390$ \\
            initial composition $\xappinit$ mixture 3 & [mol/mol, ..., mol/mol] & $(0.4, 0.3, 0.1, 0.1, 0.1)$ \\
            \hline
        \end{tabular}
    \caption{parameters and initial values used for the simulations performed in this study}
    \label{tab:sim_parameters_example1}
\end{table}

\begin{table}[h]
	\centering
	\begin{tabular}{|c|c|c|c|c|c|}
		\hline
		Deviation & $i=1$ & $i=2$ & $i=3$ & $i=4$ & $i=5$\\
		\hline
		$\devaltx{i}{1}{\mathrm{alt}_1}$ & 0.0061 & 0.0055 & 0.0026 & - & -\\
		$\devaltx{i}{1}{\mathrm{alt}'_1}$ & 0.0075 & 0.0066 & 0.0021 & - & - \\
		$\devaltx{i}{10}{\mathrm{alt}_1}$ & 0.0228 & 0.0380 & 0.0374 & - & - \\
		$\devaltx{i}{10}{\mathrm{alt}'_1}$ & 0.0230 & 0.0513 & 0.0495 & - & - \\
		$\devaltgamma{i}{1}{\mathrm{alt}_1}$ & 0.3377 & 0.0113 & 0.0311 & - & - \\
		$\devaltgamma{i}{1}{\mathrm{alt}'_1}$ & 0.3430 & 0.02217 & 0.0307 & - & - \\
		$\devaltgamma{i}{10}{\mathrm{alt}_1}$ & 0.5601 & 0.0180 & 0.0247 & - & - \\
		$\devaltgamma{i}{10}{\mathrm{alt}'_1}$ & 0.5746 & 0.0209 & 0.1452 & - & - \\
		\hline
		$\devaltx{i}{1}{\mathrm{alt}_2}$ & 0.0031 & 0.0043 & 0.0012 & 0.0073 & - \\
		$\devaltx{i}{10}{\mathrm{alt}_2}$ & 0.0182 & 0.0517 & 0.0680 & 0.0584 & - \\
		$\devaltgamma{i}{1}{\mathrm{alt}_2}$ & 0.1328 & 0.0718 & 0.0232 & 0.0253 & - \\
		$\devaltgamma{i}{10}{\mathrm{alt}_2}$ & 0.0920 & 0.7378 & 0.2001 & 0.0659 & - \\
		\hline
		$\devaltx{i}{1}{\mathrm{alt}_3}$ & 0.0128 & 0.0084 & 0.0183 & 0.0022 & 0.0086 \\
		$\devaltx{i}{10}{\mathrm{alt}_3}$ & 0.1015 & 0.1590 & 0.1210 & 0.0199 & 0.1529 \\
		$\devaltgamma{i}{1}{\mathrm{alt}_3}$ & 0.0081 & 0.0152 & 0.3107 & 0.0051 & 0.0468 \\
		$\devaltgamma{i}{10}{\mathrm{alt}_3}$ & 0.0877 & 0.2750 & 0.4756 & 0.0955 & 0.6392 \\
		\hline		
	\end{tabular}
	\caption{Metrics $\devaltx{i}{j}{\mathrm{alt}}$, and $\devaltgamma{i}{j}{\mathrm{alt}}$ for simulations $\text{alt}_1$, $\text{alt}'_1$, $\text{alt}_2$, and $\text{alt}_3$ compared to the respective reference simulations $\text{ref}_1$, $\text{ref}_2$, and $\text{ref}_3$}
	\label{tab:delta_dynamic_results}
\end{table}

\subsubsection{A three-component mixture} \label{sec:dynamic_results_example1}
The first example including the components acetone (1), methanol (2), and butanol (3) compares three simulations with the following NRTL parameter settings (see Table \ref{tab:nrtl_parameters} for the different pair-wise NRTL parameters):

\begin{itemize}
    \item Simulation $\text{ref}_1$: $\prefidx{1} = \left(\pref^{\{1,2\}}, \pref^{\{1,3\}}, \pref^{\{2,3\}} \right)$,
    \item Simulation $\text{alt}_1$: $\theta_{\text{alt}_1} = \left( \palt^{\{1,2\}}, \pref^{\{1,3\}}, \pref^{\{2,3\}} \right)$,
    \item Simulation $\text{alt}'_1$: $\theta_{\text{alt}'_1} = \left( \palt^{\{1,2\}}, \palt^{\{1,3\}}, \palt^{\{2,3\}} \right)$.
\end{itemize}

\noindent The results are shown in Figure \ref{fig:dynamic_diagram_example1}. The liquid composition trajectories 
\begin{align} \label{eq:liquid-composition-trajectories}
t \mapsto \bm{x}^{j,\theta}(t)
\end{align}
in the pot and head stages $j = 1$ and $j = 10$ are shown in \revtext{Figure \ref{fig:dynamic_diagram_example1} a) and b)}, while the corresponding activity coefficients 
\begin{align} \label{eq:activitiy-coefficient-trajectories}
t \mapsto \gamma_i^j(t) := \gamma_i^{\theta}(T^j(\bm{x}^{j,\theta}(t)), \bm{x}^{j,\theta}(t))
\qquad (j \in \{1,10\})
\end{align}
are plotted in \revtext{Figure \ref{fig:dynamic_diagram_example1} c) and d)}.
The results are plotted as a line, dashed line, and dotted line for simulations $\text{ref}_1$, $\text{alt}_1$, and $\text{alt}'_1$, respectively. The vapor mole fractions in the pot and head, the temperatures on the different stages, and the liquid holdup are shown in \ref{app:additional_results_mixture1} in Figures \ref{fig:dynamic_diagram_example1_vapor} and \ref{fig:dynamic_diagram_example1_temp_n}.

\revtext{We furthermore define the worst deviation from the reference simulation regarding the liquid composition trajectories \eqref{eq:liquid-composition-trajectories} for component $i$  on stage $j$ as}
	
\begin{align}
	\revtext{\devaltx{i}{j}{\mathrm{alt}} = \max_{t\in [0,t_\mathrm{end}]} | x_i^{j, \theta_{\text{alt}}}(t) - x_i^{j, \theta_{\text{ref}}}(t) |,  \qquad (\C, j \in \{1,10\}),}
\end{align}
\revtext{and regarding the activity coefficient trajectories \eqref{eq:activitiy-coefficient-trajectories} for component $i$ on stage $j$ as}

\begin{align}
	\revtext{\devaltgamma{i}{j}{\mathrm{alt}} = \max_{t \in [0,t_\mathrm{end}]} | \gamma_i^{j, \theta_{\text{alt}}}(t) - \gamma_i^{j, \theta_{\text{ref}}}(t) |,}
\end{align}
\revtext{with $j \in \{1,10\}$, $\C$, and with $\mathrm{alt} \in \{\mathrm{alt}_1, \mathrm{alt}'_1 \}$.}

There are significant differences in the activity coefficient trajectories visible in Figure \ref{fig:dynamic_diagram_example1}. Especially starting from \SI{300}{\second}, $\gamma_1^1$ decreases rapidly for simulations $\mathrm{alt}_1$ and $\mathrm{alt}'_1$, while this behavior is not visible in simulation 1. The same effect can be observed in the head stage, where there is a sharp decrease of the $\gamma^{10}_1$ at the same time, while $x^{10}_2$ and $x_3^{10}$ decrease and increase rapidly. \revtext{The worst deviation from the reference simulation $\text{ref}_1$ is maximal for component 1 of simulation $\mathrm{alt}'_1$ with $\devaltgamma{1}{1}{\mathrm{alt}'_1} = 0.3430$ and $\devaltgamma{1}{10}{\mathrm{alt}'_1} = 0.5746$.} \revtext{Regarding the liquid mole fractions, the worst deviation from the reference simulation on stage 1 is maximal for component 1 with $\devaltx{1}{1}{\mathrm{alt}'_1}=0.0075$. On stage $10$ the error is approximately three times higher with $\devaltx{1}{10}{\mathrm{alt}'_1} = 0.0230$.}

We also performed individual simulations, where we used the alternative NRTL parameters only for the submixtures (1)-(3), and (2)-(3), but comparing this to reference simulation $\text{ref}_1$, there was not a large visible difference between those simulations. Therefore, we can conclude for this example that the alternative NRTL parameters $\palt^{\{1,3\}}$, and $\palt^{\{2,3\}}$ on their own do not have a large effect on the simulation, whereas the alternative NRTL parameters $\palt^{\{1,2\}}$ have quite a large effect on the activity coeffiecient trajectories \revtext{with a maximum worst deviation of $\devaltgamma{1}{10}{\mathrm{alt}_1} = 0.5601$}, but \revtext{the effect of the sharp decrease in the activity coefficient trajectories of simulations $\mathrm{alt}_1$, and $\mathrm{alt}'_1$ is not reflected} on the other trajectories \revtext{with a maximum worst deviation of $\devaltx{2}{10}{\mathrm{alt}'_1}=0.0380$}.  Combining the different alternative NRTL parameters in simulation $\mathrm{alt}'_1$ intensifies the effect for this example as demonstrated in Figure \ref{fig:dynamic_diagram_example1} \revtext{and leads to maximum worst deviations of $\devaltgamma{1}{10}{\mathrm{alt}'_1} = 0.5746$, and $\devaltx{2}{10}{\mathrm{alt}'_1} = 0.0513$}. \revtext{Further worst deviations for simulations $\mathrm{alt}_1$, and $\mathrm{alt}'_1$ are listed in Table \ref{tab:delta_dynamic_results}.}

\begin{figure}[H]
	\centering
	\makebox[\textwidth]{\includegraphics[width=\textwidth]{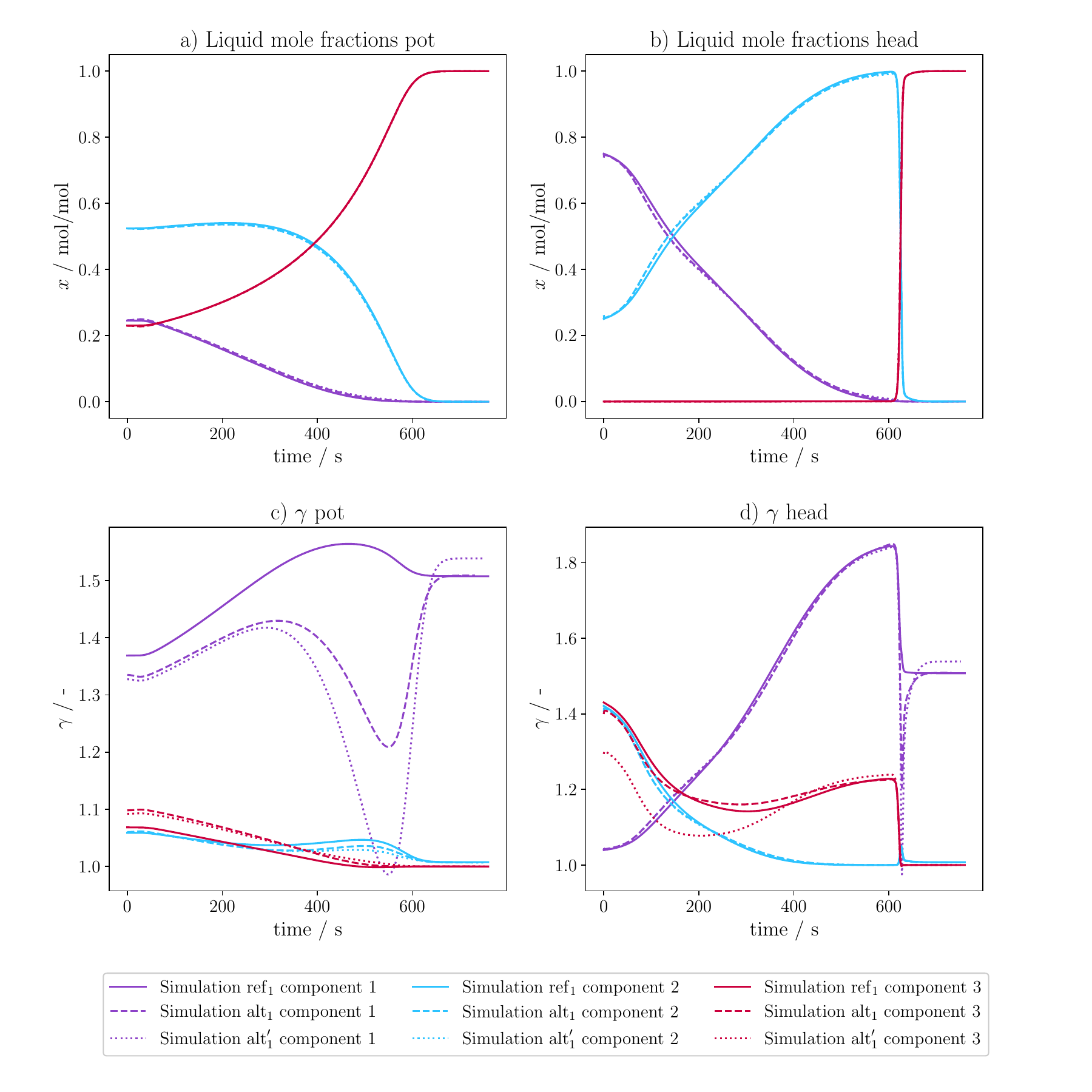}}
	\caption{Simulation results for example 1 acetone (1), methanol (2), and butanol (3); liquid composition trajectories~\eqref{eq:liquid-composition-trajectories} for pot and head stage are shown in a) and b), and activity coefficient trajectories~\eqref{eq:activitiy-coefficient-trajectories} for pot and head stage are shown in c) and d)}
	\label{fig:dynamic_diagram_example1}
\end{figure}

\subsubsection{A four-component mixture} \label{sec:dynamic_results_example2}
Example 2 includes the components acetone (1), methanol (2), butanol (3), and chloroform (4). We compare two different simulations with the following NRTL parameter settings (see Table \ref{tab:nrtl_parameters} for the NRTL parameters):

\begin{itemize}
    \item Simulation $\text{ref}_2$: $\prefidx{2} = \left(\pref^{\{1,2\}}, \pref^{\{1,3\}}, \pref^{\{1,4\}}, \pref^{\{2,3\}}, \pref^{\{2,4\}}, \pref^{\{3,4\}} \right)$,
    \item Simulation $\text{alt}_2$: $\theta_{\text{alt}_2} = \left( \palt^{\{1,2\}}, \palt^{\{1,3\}}, \palt^{\{1,4\}}, \palt^{\{2,3\}}, \pref^{\{2,4\}}, \pref^{\{3,4\}} \right)$.
\end{itemize}

The results are shown in Figure \ref{fig:dynamic_diagram_example2}. The liquid mole fractions $\bm{x}$ in the pot and head of the column \revtext{are shown in Figure \ref{fig:dynamic_diagram_example2} a) and b)}, \revtext{while Figure \ref{fig:dynamic_diagram_example2} c) and d)} shows the activity coefficient trajectories in the pot and head.

Again, we performed individual simulations, where we only changed the alternative NRTL parameter sets $\palt^{\{1,2\}}$, $\palt^{\{1,3\}}$, $\palt^{\{1,4\}}$, and $\palt^{\{2,3\}}$, respectively. In this case, the largest differences can be observed for $\palt^{\{1,2\}}$, and $\palt^{\{1,4\}}$; these submixtures ((1)-(2), and (1)-(4)) are both azeotropic. Combining multiple alternative NRTL parameter sets does not intensify the effects on the activity coefficient trajectories, as it does in the previous example. The effect on the activity coefficient trajectories is \revtext{in the same range} as \revtext{for} the previous example \revtext{with a maximum worst deviation of $\devaltgamma{2}{10}{\mathrm{alt}_2} = 0.7378$}, \revtext{as are the effects on the liquid mole fractions. The maximum worst deviation of stage $1$ can be observed for component $4$ with $\devaltx{4}{1}{\mathrm{alt}_2} = 0.0073$. Again, the worst deviation is one order of magnitude higher for component $4$ on stage $10$ with $\devaltx{4}{10}{\mathrm{alt}_2} = 0.0584$.} Additional results (vapor mole fractions, temperatures and liquid molar holdup) are shown in the  \ref{app:additional_restults_mixture2} in Figures \ref{fig:dynamic_diagram_example2_vapor} and \ref{fig:dynamic_diagram_example2_temp_n}\revtext{, and additional worst deviations for simulation $\mathrm{alt}_2$ are listed in Table \ref{tab:delta_dynamic_results}}.

\begin{figure}[H]
	\centering
	\makebox[\textwidth]{\includegraphics[width=\textwidth]{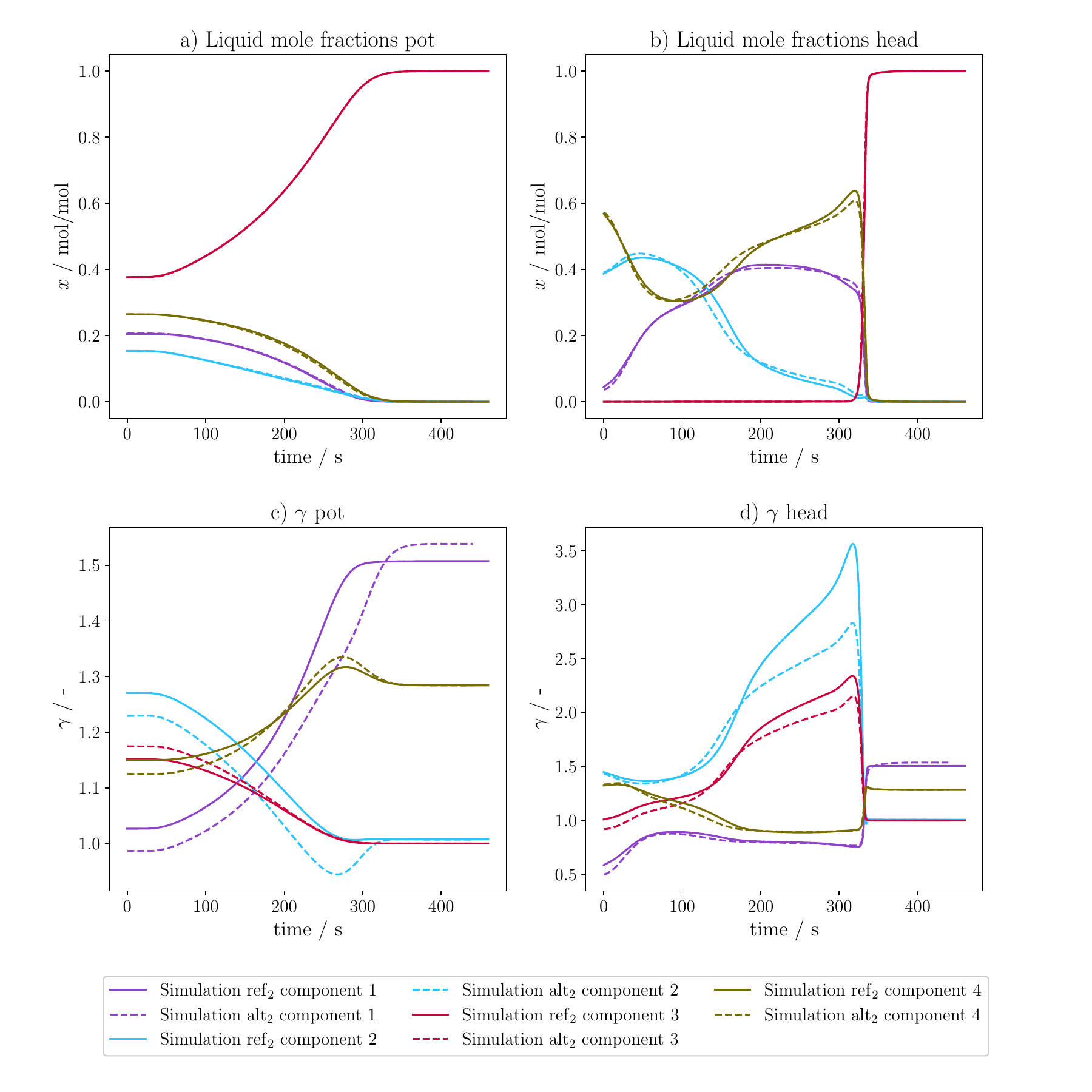}}
	\caption{Simulation results for example 2 acetone (1), methanol (2), butanol (3), and chloroform (4); liquid composition trajectories~\eqref{eq:liquid-composition-trajectories} for pot and head stage are shown in a) and b), and activity coefficient trajectories~\eqref{eq:activitiy-coefficient-trajectories} for pot and head stage are shown in c) and d)}
	\label{fig:dynamic_diagram_example2}
\end{figure}

\subsubsection{A five-component mixture} \label{sec:dynamic_results_example3}
The third example including the components 2-butanone (1), acetic acid (2), ethanol (3), ethyl acetate (4), and toluene (5) compares two different simulations with the following NRTL parameter settings:

\begin{itemize}
    \item Simulation $\text{ref}_3$: $\prefidx{3} = \left(\pref^{\{1,2\}}, \pref^{\{1,3\}}, \pref^{\{1,4\}}, \pref^{\{1,5\}}, \pref^{\{2,3\}}, \pref^{\{2,4\}}, \pref^{\{2,5\}}, \pref^{\{3,4\}}, \pref^{\{3,5\}}, \pref^{\{4,5\}} \right)$,
    \item Simulation $\text{alt}_3$: $\theta_{\text{alt}_3} = \left(\pref^{\{1,2\}}, \pref^{\{1,3\}}, \pref^{\{1,4\}}, \pref^{\{1,5\}}, \palt^{\{2,3\}}, \pref^{\{2,4\}}, \pref^{\{2,5\}}, \pref^{\{3,4\}}, \pref^{\{3,5\}}, \pref^{\{4,5\}} \right)$
\end{itemize}

The results are shown in Figure \ref{fig:dynamic_diagram_example3}, where we can again observe significant differences between the liquid mole fraction and activity coefficient trajectories \revtext{with maximum worse deviations of $\devaltx{2}{10}{\mathrm{alt}_3} = 0.1590$. Compared to the worst deviations of the liquid mole fractions on stage $10$ of the previous examples, this error is one order of magnitude higher. The worst deviation of the activity coefficient trajectories is maximal for component $3$ with $\devaltgamma{3}{10}{\mathrm{alt}_3} = 0.4756$}. Further results are shown in \ref{app:additional_results_mixture3} in Figures \ref{fig:dynamic_diagram_example3_vapor} and \ref{fig:dynamic_diagram_example3_temp_n}\revtext{, and further worst deviations for simulation $\mathrm{alt}_3$ are listed in Table \ref{tab:delta_dynamic_results}.}

\begin{figure}[H]
	\centering
	\makebox[\textwidth]{\includegraphics[width=\textwidth]{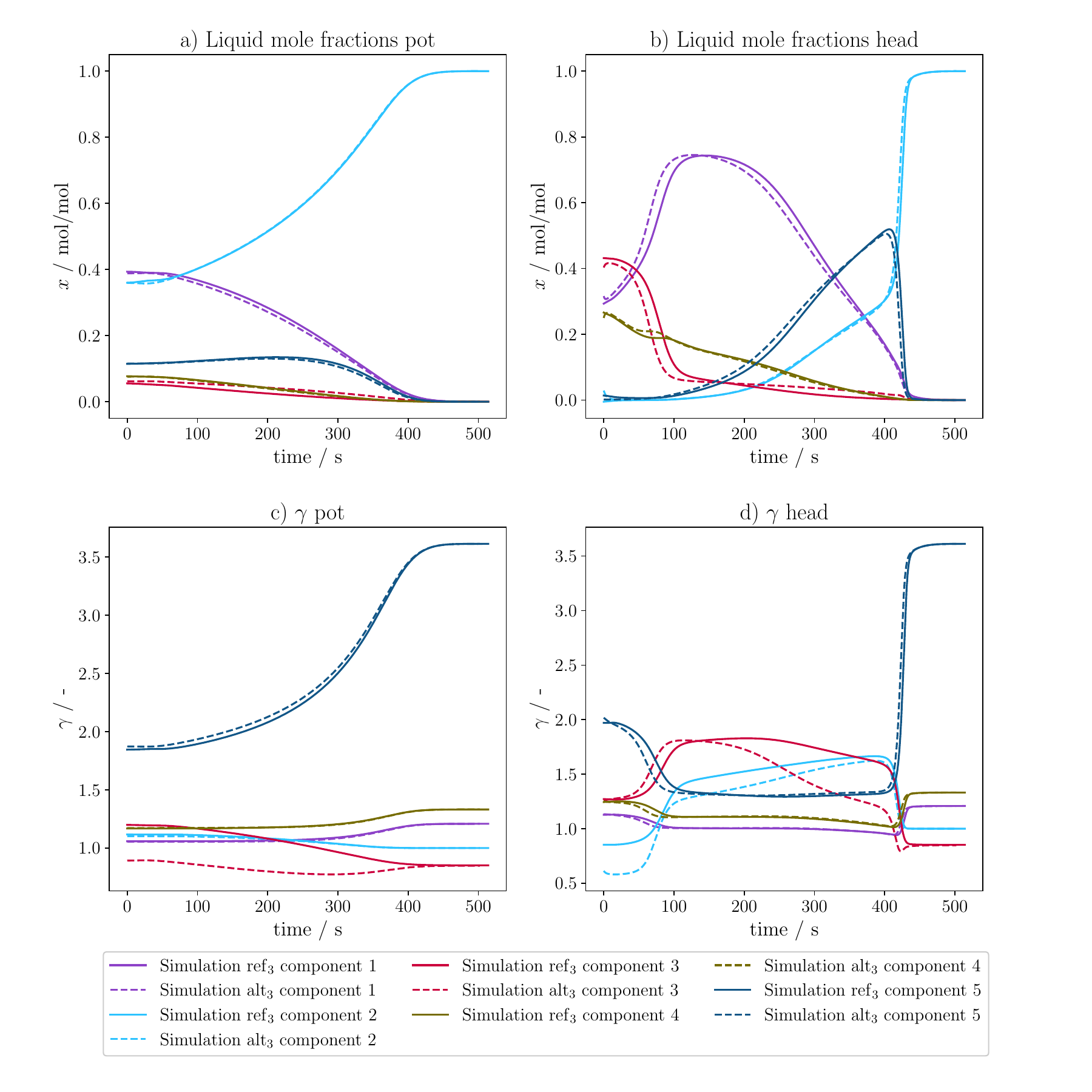}}
	\caption{Simulation results for example 3 2-butanone (1), acetic acid (2), ethanol (3), ethyl acetate (4), and toluene (5); liquid composition trajectories~\eqref{eq:liquid-composition-trajectories} for pot and head stage are shown in a) and b), and activity coefficient trajectories~\eqref{eq:activitiy-coefficient-trajectories} for pot and head stage are shown in c) and d)}
	\label{fig:dynamic_diagram_example3}
\end{figure}

\subsection{Discussion} \label{sec:study_discussion}
As demonstrated in Section \ref{sec:stationary_results}, the various NRTL parameter sets corresponding to the respective binary submixtures yield qualitatively the same activity coefficient and vapor-liquid equilibrium diagrams.

In Section \ref{sec:dynamic_results}, we observe that the application of different NRTL parameter sets in the batch distillation simulations introduced in Section \ref{sec:simulation} can lead to significant variations in the resulting trajectories. Especially the activity coefficients show large variations for some components. As the liquid mole fractions, vapor mole fractions, temperatures and activity coefficients are all linked through the extended Raoult law, and these in turn are linked to all other variables by the DAE system governing the batch distillation, these differences are observable in all related variables.

The results of our study indicate that these variations occur only in mixtures consisting of more than two components. Simulations of binary mixtures, regardless of the NRTL parameter set that was employed, consistently produce the same results within the squared error \eqref{eq:NRTL-parameter-estimation-problem} that is discussed in Section \ref{sec:stationary_results}\footnote{Simulations of the five binary submixtures with alternative NRTL parameters are shown in \ref{app:additional_results_binary}}.

The extent of these effects is influenced by the specific components included. The first and second examples (see Sections \ref{sec:dynamic_results_example1} and \ref{sec:dynamic_results_example2}) exhibit notable differences in the activity coefficient trajectories, while the differences in the other trajectories (liquid/vapor mole fractions, temperatures) are relatively minor. Conversely, the third example (see Section \ref{sec:dynamic_results_example3}) shows smaller variations in the activity coefficient trajectories, yet the differences in other trajectories remain significant. Overall, both the first and second examples show the largest discrepancies when alternative NRTL parameters for azeotropic submixtures are applied, whereas the third example (see Section \ref{sec:dynamic_results_example3}) shows considerable differences associated with the alternative NRTL parameter for a zeotropic submixture. \revtext{The worst deviations between the alternative and respective reference simulations are listed in Table \ref{tab:delta_dynamic_results}.}

The observed differences in the trajectories can be explained by the fact that $\palt^{\{i,j\}}$ was only determined by the activity coefficients of the binary submixture ($i$)-($j$) within the multi-component mixtures.
\revtext{As we demonstrated in Section \ref{sec:stationary_results}, the activity coefficient of component $k \in \{i,j\}$ depending on the bubble-point temperature $T^{\pref^{\{i,j\}}}(\bm{x})$ and the binary mixture composition $\bm{x}$ is approximately the same as the activity coefficient of component $k$ depending on $T^{\palt^{\{i,j\}}}(\bm{x})$ and $\bm{x}$. In short,}
\begin{equation} \label{eq:gamma_ref_equals_gamma_alt}
    \gamma_k^{\pref^{\{i,j\}}} \left( T^{\pref^{\{i,j\}}}(\bm{x}), \bm{x} \right) \approx \gamma_k^{\palt^{\{i,j\}}} \left( T^{\palt^{\{i,j\}}}(\bm{x}), \bm{x} \right) \qquad (\bm{x} \in \mathcal{X}_{ij}),
\end{equation}
\noindent for $k \in \{i,j\}$. \revtext{Note that this equation holds true for all binary compositions on the simplex edge $\mathcal{X}_{ij}$ when $k$ is either $i$ or $j$. It does not necessarily hold true for all other compositions in the ternary simplex.}


This approach does not include the interactions among components in a multi-component mixture. Consequently, we observe minimal differences in the stationary results of this study (see Section \ref{sec:stationary_results}), while significantly larger discrepancies arise in the dynamic results (refer to Section \ref{sec:dynamic_results}). This can be \revtext{further} demonstrated by analyzing the ternary diagram of the multi-component mixture 1 with the components acetone (1),  methanol (2), and butanol (3) in Figure \ref{fig:ternary_plots_example1}. In the dynamic results of this mixture presented in Section \ref{sec:dynamic_results_example1}, we observe the most significant differences in the activity coefficient of component (1) \revtext{for simulation $\mathrm{alt}'_1$ with a worst deviation of $\devaltgamma{1}{10}{\mathrm{alt}'_1} = 0.5746$}. To further investigate this, we sample liquid concentrations $\bm{x}$ from the ternary mixture, solve the VLE using $\pref$, and $\theta_{\text{alt}_1}$, and generate a heatmap for $i=1$ of the metric
\begin{equation}
    \bm{\Delta}_{\gamma_i} = \left(\left|\gamma_i^{\pref}\big((T^{\pref}(\bm{x}^l), \bm{x}^l\big) - \gamma_i^{\theta_{\text{alt}_1}}\big(T^{\theta_{\text{alt}_1}}(\bm{x}^l), \bm{x}^l\big)\right|\right)_{l \in \{1,\dots,N\}} \qquad (\C),
\end{equation}
with $N=860$, which is the number of points sampled from the ternary mixture. The ternary diagram with the heatmap of $\bm{\Delta}_{\gamma_1}$ is shown in Figure \ref{fig:ternary_plots_example1}. Additionally, the trajectories of the liquid mole fractions in both the pot and head stages of simulations $\text{ref}_1$, and $\text{alt}_1$ are visualized in the figure.

\begin{figure}[H]
	\centering
	\makebox[\textwidth]{\includegraphics[width=\textwidth]{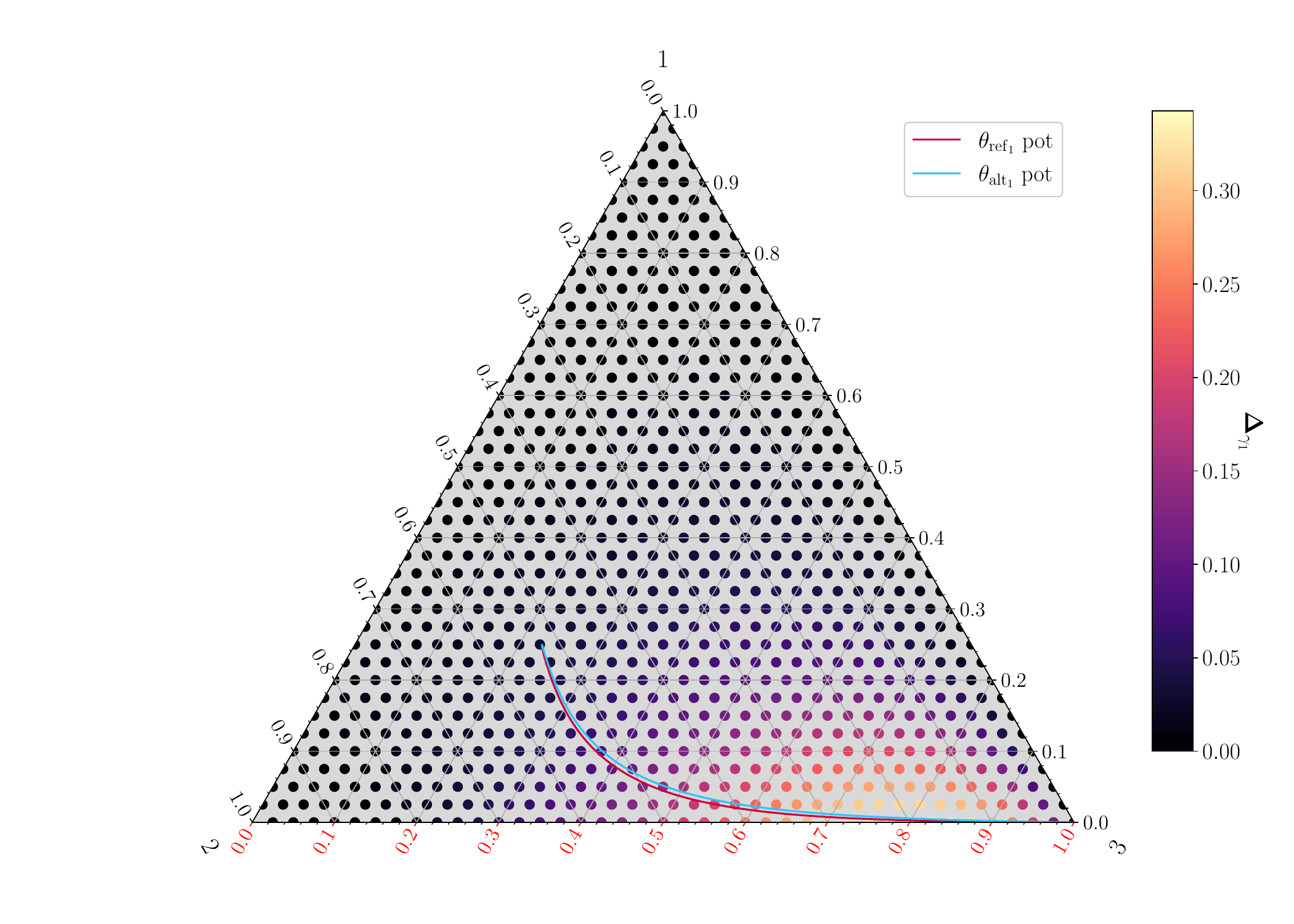}}
	\caption{Ternary diagram of mixture (1), (2), and (3); the metric $\bm{\Delta}_{\gamma_1}$ is plotted as a heatmap within the ternary diagram; the liquid mole fractions trajectories in the pot of simulations $\text{ref}_1$ and $\text{alt}_1$ are visualized by the red and blue line-plots; the red-marked axis indicates that a high $\bm{\Delta}_{\gamma_1}$ can be expected for the binary sub-system (2)-(3)}
	\label{fig:ternary_plots_example1}
\end{figure}

From Figure \ref{fig:ternary_plots_example1}, it is evident that the trajectories of $\bm{x}^{1,\pref}$ and $\bm{x}^{1,\theta_{\text{alt1}}}$ differ, as it is already visualized in Figure \ref{fig:dynamic_diagram_example1} in Section \ref{sec:dynamic_results_example1}. A second important observation is that these trajectories pass through the region where $\bm{\Delta}_{\gamma_1}$ reaches its maximum value of $0.34$, which corresponds to the differences shown in Figure \ref{fig:dynamic_diagram_example1}, \revtext{and the worst deviations listed in Table \ref{tab:delta_dynamic_results}} for the activity coefficient trajectories of component (1) in the pot and head stages. Moreover, Figure \ref{fig:ternary_plots_example1} illustrates that $\bm{\Delta}_{\gamma_1}$ equals $0$ along the simplex edge for the binary submixtures (1)-(2) and (1)-(3), but it does not equal $0$ along the simplex edge for the binary submixture (2)-(3) nor within the simplex itself.

This behavior can also be observed for the metrics $\bm{\Delta}_{\gamma_2}$, and $\bm{\Delta}_{\gamma_3}$ in Figures \ref{fig:ternary_plots_example1_2} and \ref{fig:ternary_plots_example1_3} in \ref{app:additional_results_mixture1}, although $\bm{\Delta}_{\gamma_2}$ and $\bm{\Delta}_{\gamma_3}$ are considerably smaller than $\bm{\Delta}_{\gamma_1}$. This corresponds to the results we observe in Figure \ref{fig:dynamic_diagram_example1}, where the trajectories of the activity coefficients for components (2), and (3) show also considerably smaller differences to the reference.

For the example mixtures 2 and 3 with four and five components, this cannot be visualized as easily as for the first mixture. Nonetheless, we can sample the quaternary and quinary mixture space and compare the results. For this sake, we introduce the following mean deviation and worst-case deviation metrics:

\begin{equation}
    \devmeangamma{i} = \frac{1}{N} \sum_{l=1}^N |\gamma_i^{\pref}\left(T^{\pref}(\bm{x}^l), \bm{x}^l) \right) - \gamma_i^{\palt}\left(T^{\palt}(\bm{x}^l), \bm{x}^l) \right)| \qquad (\C),
\end{equation}

\begin{equation}
    \devmeanT = \frac{1}{N} \sum_{l=1}^N |T^{\pref}(\bm{x}^l) - T^{\palt}(\bm{x}^l)|,
\end{equation}

\begin{equation}
    \devworstgamma{i} = \max_{l=1,...,N}|\gamma_i^{\pref}\left(T^{\pref}(\bm{x}^l), \bm{x}^l) \right) - \gamma_i^{\palt}\left(T^{\palt}(\bm{x}^l), \bm{x}^l) \right)| \qquad (\C),
\end{equation}

\begin{equation}
    \devworstT = \max_{l=1,...,N} |T^{\pref}(\bm{x}^l) - T^{\palt}(\bm{x}^l)|,
\end{equation}
where $N$ refers to the sample points in the simplex, and is $860$, $1\,752$, and $10\,422$ for mixtures 1, 2, and 3, respectively. We sample the ternary compositions equidistantly by $41$ points on each edge of the ternary simplex boundary, and examples 2 and 3 with $21$ points on each edge of the corresponding quaternary and quinary simplex boundaries. The resulting deviation metrics for the different alternative simulations compared to the corresponding reference simulation are listed in Table \ref{tab:metrics1} and show similar error metrics.

\begin{table}[h]
    \centering
        \begin{tabular}{|c|c|c|c|c|c|c|c|c|c|c|c|c|}
            \hline
            Id & $\devmeangamma{1}$ & $\devmeangamma{2}$ & $\devmeangamma{3}$ & $\devmeangamma{4}$ & $\devmeangamma{5}$ & $\devworstgamma{1}$ & $\devworstgamma{2}$ & $\devworstgamma{3}$ & $\devworstgamma{4}$ & $\devworstgamma{5}$ & $\devmeanT$ & $\devworstT$\\
            \hline
            $\text{alt}_1$ & $0.051$ & $0.029$ & $0.030$ & - & - & $0.343$ & $0.117$ & $0.034$ & - & - & $0.334$ & $2.700$\\
            $\text{alt}'_1$ & $0.075$ & $0.035$ & $0.040$ & - & - & $0.573$ & $0.155$ & $0.131$ & - & - & $0.598$ & $3.650$ \\
            $\text{alt}_2$ & $0.045$ & $0.065$ & $0.029$ & $0.025$ & - & $0.570$ & $0.0760$ & $0.150$ & $0.110$ & - & $0.399$ & $3.650$\\ 
            $\text{alt}_3$ & $0.028$ & $0.113$ & $0.129$ & $0.024$ & $0.025$ & $0.346$ & $0.528$ & $0.491$ & $0.331$ & $0.558$ & $0.511$ & $2.140$\\
            \hline
        \end{tabular}
    \caption{Metrics $\devmeangamma{i}$, $\devworstgamma{i}$, $\devmeanT$, and $\devworstT$ for simulations $\text{alt}_1$, $\text{alt}'_1$, $\text{alt}_2$, and $\text{alt}_3$ compared to the respective reference simulations $\text{ref}_1$, $\text{ref}_2$, and $\text{ref}_3$}
    \label{tab:metrics1}
\end{table}

\section{Summary and conclusions}

In this article, we developed a simulation for batch distillation processes featuring a rigorous novel yet simple index reduction method. An important advantage over other index reduction techniques known from the literature is that our approach leads to a relatively simple index-reduced system. In particular, our index-reduced system does not require the computation of the partial derivatives of the underlying vapor-liquid equilibrium model. Another advantage of the index reduction method proposed here is that it allows for the rigorous proof of the index-$1$ property under fairly general conditions. With the approaches from the literature, this is considerably more difficult.

We then utilize the developed batch distillation simulation to analyze the effects of multiplicities in activity coefficients on batch distillation by using alternative NRTL parameter sets leading to indistinguishable stationary activity coefficient and VLE diagrams of the binary submixtures. The study shows that, in general, there are considerable differences in the simulation results concerning the activity coefficient, temperature, and liquid and vapor mole fraction trajectories. Since the variables are all linked through the DAE system, the differences can in fact be observed in all trajectories. While there are particularly high differences in the activity coefficients for some components, we can conclude from the examples analyzed in this study, that the simulation model still provides \revtext{qualitatively} the same dynamic profiles with alternative NRTL parameters. \revtext{The maximum worst deviation from the reference simulation regarding the liquid mole fractions is $\devaltx{2}{10}{\mathrm{alt}_3} = 0.1590$ (see Table \ref{tab:delta_dynamic_results}).} 

Another important finding of the study is that these variations occur only in mixtures consisting of more than two components, and the extent of the variations depends on the components present in the mixture. This can be explained by the fact that the NRTL parameters are fitted only to the activity coefficients of the binary submixtures, which is a common practice. This means that the alternative NRTL parameters with certainty only perform well on a subset of all possible concentrations in the ternary, quaternary, or quinary simplex of the three examples, namely the subset of binary concentrations within the mixture. Yet, as demonstrated in this study, it is not guaranteed that the alternative NRTL parameters perform well for all possible concentrations of a multi-component mixture.

In light of this, future work should revise the common practice of the NRTL parameter estimation problem. We think that the differences we show in this paper can be improved by fitting NRTL parameters for multi-component mixtures not only on pair-wise experimental data, but for the entire mixture space, including experimental data for more possible concentrations of non-binary mixtures.

\section*{Acknowledgment}
We gratefully acknowledge funding from the Deutsche Forschungsgemeinschaft (DFG, German
Research Foundation) within the research unit ``FOR 5359: Deep Learning on Sparse Chemical Process Data''.

\section*{Author contributions}
Jennifer Werner: Conceptualization, Software, Validation, Investigation, Data curation, Writing -- orignal draft, Visualization. 
Jochen Schmid: Conceptualization, Methodology, Formal analysis, Validation, Investigation, Data curation, Writing -- orignal draft, Supervision. 
Lorenz T. Biegler: Conceptualization, Validation, Writing -- review \& editing.
Michael Bortz: Conceptualization, Resources, Writing -- review \& editing, Supervision, Project administration, Funding acquisition. 

\bibliographystyle{unsrtnat}
\bibliography{references} 

\begin{thebibliography}{48}
\providecommand{\natexlab}[1]{#1}
\providecommand{\url}[1]{\texttt{#1}}
\expandafter\ifx\csname urlstyle\endcsname\relax
  \providecommand{\doi}[1]{doi: #1}\else
  \providecommand{\doi}{doi: \begingroup \urlstyle{rm}\Url}\fi

\bibitem[Miyahara et~al.(1970)Miyahara, Sadotomo, and
  Kitamura]{miyahara1970evaluation}
Koreatsu Miyahara, Hideo Sadotomo, and Kazuhiko Kitamura.
\newblock {Evaluation of the Wilson parameters by nomographs}.
\newblock \emph{Journal of Chemical Engineering of Japan}, 3\penalty0
  (2):\penalty0 157--160, 1970.

\bibitem[Silverman and Tassios(1977)]{silverman1977number}
Norman Silverman and Dimitrios Tassios.
\newblock {The number of roots in the Wilson equation and its effect on
  vapor-liquid equilibrium calculations}.
\newblock \emph{Industrial \& Engineering Chemistry Process Design and
  Development}, 16\penalty0 (1):\penalty0 13--20, 1977.

\bibitem[Gau et~al.(2000)Gau, Brennecke, and Stadtherr]{gau2000reliable}
Chao-Yang Gau, Joan~F Brennecke, and Mark~A Stadtherr.
\newblock {Reliable nonlinear parameter estimation in VLE modeling}.
\newblock \emph{Fluid Phase Equilibria}, 168\penalty0 (1):\penalty0 1--18,
  2000.

\bibitem[Tassios(1979)]{tassios1979number}
Dimitrios Tassios.
\newblock {The number of roots in the NRTL and LEMF equations and the effect on
  their performance}.
\newblock \emph{Industrial \& Engineering Chemistry Process Design and
  Development}, 18\penalty0 (1):\penalty0 182--186, 1979.

\bibitem[Vasiliu et~al.(2021)Vasiliu, G{\"o}ttl, Br{\"o}cker, and
  Burger]{vasiliu2021multiple}
Dan Vasiliu, Quirin G{\"o}ttl, S{\"o}nke Br{\"o}cker, and Jakob Burger.
\newblock Multiple solutions when fitting excess gibbs energy models and
  implications for process simulation.
\newblock \emph{Chemie Ingenieur Technik}, 93\penalty0 (3):\penalty0 490--496,
  2021.

\bibitem[Renon and Prausnitz(1968)]{renon1968local}
Henri Renon and John~M Prausnitz.
\newblock Local compositions in thermodynamic excess functions for liquid
  mixtures.
\newblock \emph{{AIChE Journal}}, 14\penalty0 (1):\penalty0 135--144, 1968.

\bibitem[Gmehling et~al.(2019)Gmehling, Kleiber, Kolbe, and
  Rarey]{gmehling2019chemical}
J{\"u}rgen Gmehling, Michael Kleiber, B{\"a}rbel Kolbe, and J{\"u}rgen Rarey.
\newblock \emph{Chemical thermodynamics for process simulation}.
\newblock John Wiley \& Sons, 2019.

\bibitem[Werner et~al.(2023)Werner, Seidel, Jafar, Heese, Hasse, and
  Bortz]{werner2023multiplicities}
Jennifer Werner, Tobias Seidel, Ramzi Jafar, Raoul Heese, Hans Hasse, and
  Michael Bortz.
\newblock Multiplicities in thermodynamic activity coefficients.
\newblock \emph{{AIChE Journal}}, 69\penalty0 (12):\penalty0 e18251, 2023.

\bibitem[Meadows(1963)]{meadows1963multicomponent}
EL~Meadows.
\newblock Multicomponent batch-distillation calculations on a digital computer.
\newblock In \emph{Chem. Eng. Symp. Ser}, volume~59, pages 46--48, 1963.

\bibitem[Distefano(1968)]{distefano1968mathematical}
GP~Distefano.
\newblock Mathematical modeling and numerical integration of multicomponent
  batch distillation equations.
\newblock \emph{{AIChE Journal}}, 14\penalty0 (1):\penalty0 190--199, 1968.

\bibitem[Boston et~al.(1981)Boston, Britt, Jirapongphan, and
  Shah]{boston1981foundations}
Joseph~F Boston, Herbert~I Britt, Siri Jirapongphan, and VB~Shah.
\newblock An advanced system for the simulation of batch distillation
  operations.
\newblock \emph{Foundations of Computer-Aided Chemical Process Design},
  2:\penalty0 203--237, 1981.

\bibitem[Doherty and Perkins(1978)]{Doherty1978}
M.F. Doherty and J.D. Perkins.
\newblock {On the dynamics of distillation processes—I: The simple
  distillation of multicomponent non-reacting, homogeneous liquid mixtures}.
\newblock \emph{{Chemical Engineering Science}}, 33\penalty0 (3):\penalty0
  281--301, 1978.
\newblock ISSN 0009-2509.
\newblock \doi{https://doi.org/10.1016/0009-2509(78)80086-4}.
\newblock URL
  \url{https://www.sciencedirect.com/science/article/pii/0009250978800864}.

\bibitem[{Van Dongen} and Doherty(1985)]{VanDongen1985}
David~B. {Van Dongen} and Michael~F. Doherty.
\newblock On the dynamics of distillation processes—vi. batch distillation.
\newblock \emph{{Chemical Engineering Science}}, 40\penalty0 (11):\penalty0
  2087--2093, 1985.
\newblock ISSN 0009-2509.
\newblock \doi{https://doi.org/10.1016/0009-2509(85)87026-3}.
\newblock URL
  \url{https://www.sciencedirect.com/science/article/pii/0009250985870263}.

\bibitem[Cervantes and Biegler(1998)]{cervantes1998}
A.~Cervantes and L.~T. Biegler.
\newblock {Large-scale DAE optimization using a simultaneous NLP formulation}.
\newblock \emph{{AIChE Journal}}, 44\penalty0 (5):\penalty0 1038--1050, 1998.
\newblock \doi{https://doi.org/10.1002/aic.690440505}.
\newblock URL
  \url{https://aiche.onlinelibrary.wiley.com/doi/abs/10.1002/aic.690440505}.

\bibitem[Cervantes et~al.(2000)Cervantes, Wächter, Tütüncü, and
  Biegler]{cervantes2000}
Arturo~M. Cervantes, Andreas Wächter, Reha~H. Tütüncü, and Lorenz~T.
  Biegler.
\newblock A reduced space interior point strategy for optimization of
  differential algebraic systems.
\newblock \emph{{Computers \& Chemical Engineering}}, 24\penalty0 (1):\penalty0
  39--51, 2000.
\newblock ISSN 0098-1354.
\newblock \doi{https://doi.org/10.1016/S0098-1354(00)00302-1}.
\newblock URL
  \url{https://www.sciencedirect.com/science/article/pii/S0098135400003021}.

\bibitem[Cervantes and Biegler(2000)]{CERVANTES200041}
Arturo~M Cervantes and Lorenz~T Biegler.
\newblock A stable elemental decomposition for dynamic process optimization.
\newblock \emph{Journal of Computational and Applied Mathematics}, 120\penalty0
  (1):\penalty0 41--57, 2000.
\newblock ISSN 0377-0427.
\newblock \doi{https://doi.org/10.1016/S0377-0427(00)00302-2}.
\newblock URL
  \url{https://www.sciencedirect.com/science/article/pii/S0377042700003022}.

\bibitem[Biegler et~al.(2002)Biegler, Cervantes, and Wächter]{biegler2002}
Lorenz~T. Biegler, Arturo~M. Cervantes, and Andreas Wächter.
\newblock Advances in simultaneous strategies for dynamic process optimization.
\newblock \emph{{Chemical Engineering Science}}, 57\penalty0 (4):\penalty0
  575--593, 2002.
\newblock ISSN 0009-2509.
\newblock \doi{https://doi.org/10.1016/S0009-2509(01)00376-1}.
\newblock URL
  \url{https://www.sciencedirect.com/science/article/pii/S0009250901003761}.

\bibitem[Raghunathan et~al.(2004)Raghunathan, {Soledad Diaz}, and
  Biegler]{ragunathan2004}
Arvind~U Raghunathan, M~{Soledad Diaz}, and Lorenz~T Biegler.
\newblock {An MPEC formulation for dynamic optimization of distillation
  operations}.
\newblock \emph{{Computers \& Chemical Engineering}}, 28\penalty0
  (10):\penalty0 2037--2052, 2004.
\newblock ISSN 0098-1354.
\newblock \doi{https://doi.org/10.1016/j.compchemeng.2004.03.015}.
\newblock URL
  \url{https://www.sciencedirect.com/science/article/pii/S0098135404000857}.
\newblock Special Issue for Professor Arthur W. Westerberg.

\bibitem[Eckert and Van{\v{e}}k(2008)]{eckert2008mathematical}
Egon Eckert and Tom{\'a}{\v{s}} Van{\v{e}}k.
\newblock Mathematical modelling of selected characterisation procedures for
  oil fractions.
\newblock \emph{Chemical Papers}, 62:\penalty0 26--33, 2008.

\bibitem[Lopez-Saucedo et~al.(2016)Lopez-Saucedo, Grossmann, Segovia-Hernandez,
  and Hern{\'a}ndez]{lopez2016rigorous}
Edna~Soraya Lopez-Saucedo, Ignacio~E Grossmann, Juan~Gabriel Segovia-Hernandez,
  and Salvador Hern{\'a}ndez.
\newblock {Rigorous modeling, simulation and optimization of a conventional and
  nonconventional batch reactive distillation column: A comparative study of
  dynamic optimization approaches}.
\newblock \emph{Chemical Engineering Research and Design}, 111:\penalty0
  83--99, 2016.

\bibitem[Bortz et~al.(2019)Bortz, Heese, Scherrer, Gerlach, and
  Runowski]{bortz2019estimating}
Michael Bortz, Raoul Heese, Alexander Scherrer, Thomas Gerlach, and Thomas
  Runowski.
\newblock Estimating mixture properties from batch distillation using
  semi-rigorous and rigorous models.
\newblock In \emph{Computer Aided Chemical Engineering}, volume~46, pages
  643--648. Elsevier, 2019.

\bibitem[Mohring et~al.(2022)Mohring, Schmid, Wlazło, Heese, Gerlach,
  Kochenburger, and Bortz]{Mohring2022}
Jan Mohring, Jochen Schmid, Jarosław Wlazło, Raoul Heese, Thomas Gerlach,
  Thomas Kochenburger, and Michael Bortz.
\newblock {Modeling and optimizing dynamic networks: Applications in process
  engineering and energy supply}.
\newblock In Michael Bortz and Norbert Asprion, editors, \emph{Simulation and
  Optimization in Process Engineering}, pages 143--160. Elsevier, 2022.
\newblock ISBN 978-0-323-85043-8.
\newblock \doi{https://doi.org/10.1016/B978-0-323-85043-8.00013-1}.
\newblock URL
  \url{https://www.sciencedirect.com/science/article/pii/B9780323850438000131}.

\bibitem[Qian et~al.(2023)Qian, Lin, Jia, Biegler, and
  Huang]{qian2023nonlinear}
Xing Qian, Kuan-Han Lin, Shengkun Jia, Lorenz~T Biegler, and Kejin Huang.
\newblock Nonlinear model predictive control for dividing wall columns.
\newblock \emph{AIChE Journal}, 69\penalty0 (6):\penalty0 e18062, 2023.

\bibitem[Seader et~al.(2006)Seader, Henley, and Roper]{seader2006separation}
Junior~D Seader, Ernest~J Henley, and D~Keith Roper.
\newblock \emph{Separation process principles}.
\newblock wiley New York, 2006.

\bibitem[Biegler(2010)]{biegler2010nonlinear}
Lorenz~T Biegler.
\newblock \emph{Nonlinear programming: concepts, algorithms, and applications
  to chemical processes}.
\newblock SIAM, 2010.

\bibitem[Biegler et~al.(1997)Biegler, Grossmann, and
  Westerberg]{biegler1997systematic}
Lorenz~T Biegler, Ignacio~E Grossmann, and Arthur~W Westerberg.
\newblock Systematic methods for chemical process design.
\newblock 1997.

\bibitem[King(2013)]{king2013separation}
C~Judson King.
\newblock \emph{Separation processes}.
\newblock Courier Corporation, 2013.

\bibitem[Diwekar(2011)]{diwekar2011batch}
Urmila Diwekar.
\newblock \emph{Batch distillation: simulation, optimal design, and control}.
\newblock CRC press, 2011.

\bibitem[Gani and Cameron(1992)]{gani1992modelling}
Rafiqul Gani and Ian~T Cameron.
\newblock Modelling for dynamic simulation of chemical processes: the index
  problem.
\newblock \emph{{Chemical Engineering Science}}, 47\penalty0 (5):\penalty0
  1311--1315, 1992.

\bibitem[Peng et~al.(2003)Peng, Edgar, and Eldridge]{peng2003dynamic}
Jianjun Peng, Thomas~F Edgar, and R~Bruce Eldridge.
\newblock Dynamic rate-based and equilibrium models for a packed reactive
  distillation column.
\newblock \emph{{Chemical Engineering Science}}, 58\penalty0 (12):\penalty0
  2671--2680, 2003.

\bibitem[Hairer et~al.(1991)Hairer, Wanner, and Solving]{hairer1991ii}
E~Hairer, G~Wanner, and O~Solving.
\newblock {Ii: Stiff and differential-algebraic problems}.
\newblock \emph{Berlin [etc.]: Springer}, 1991.

\bibitem[Kunkel(2006)]{kunkel2006differential}
Peter Kunkel.
\newblock \emph{Differential-algebraic equations: analysis and numerical
  solution}, volume~2.
\newblock European Mathematical Society, 2006.

\bibitem[Bachmann et~al.(1990)Bachmann, Br{\"u}ll, Mrziglod, and
  Pallaske]{bachmann1990methods}
R~Bachmann, L~Br{\"u}ll, Th~Mrziglod, and U~Pallaske.
\newblock On methods for reducing the index of differential algebraic
  equations.
\newblock \emph{Computers \& Chemical Engineering}, 14\penalty0 (11):\penalty0
  1271--1273, 1990.

\bibitem[Fletcher and Morton(2000)]{fletcher2000initialising}
Roger Fletcher and William Morton.
\newblock Initialising distillation column models.
\newblock \emph{{Computers \& Chemical Engineering}}, 23\penalty0
  (11-12):\penalty0 1811--1824, 2000.

\bibitem[Bynum et~al.(2021)Bynum, Hackebeil, Hart, Laird, Nicholson, Siirola,
  Watson, and Woodruff]{bynum2021pyomo}
Michael~L. Bynum, Gabriel~A. Hackebeil, William~E. Hart, Carl~D. Laird,
  Bethany~L. Nicholson, John~D. Siirola, Jean-Paul Watson, and David~L.
  Woodruff.
\newblock \emph{Pyomo--optimization modeling in python}, volume~67.
\newblock Springer Science \& Business Media, third edition, 2021.

\bibitem[Hart et~al.(2011)Hart, Watson, and Woodruff]{hart2011pyomo}
William~E Hart, Jean-Paul Watson, and David~L Woodruff.
\newblock Pyomo: modeling and solving mathematical programs in python.
\newblock \emph{Mathematical Programming Computation}, 3\penalty0 (3):\penalty0
  219--260, 2011.

\bibitem[Nicholson et~al.(2018)Nicholson, Siirola, Watson, Zavala, and
  Biegler]{Nicholson2018}
Bethany Nicholson, John~D. Siirola, Jean-Paul Watson, Victor~M. Zavala, and
  Lorenz~T. Biegler.
\newblock pyomo.dae: a modeling and automatic discretization framework for
  optimization with differential and algebraic equations.
\newblock \emph{Mathematical Programming Computation}, 10\penalty0
  (2):\penalty0 187--223, 2018.

\bibitem[Ascher and Petzold(1998)]{ascher1998computer}
Uri~M Ascher and Linda~R Petzold.
\newblock \emph{Computer methods for ordinary differential equations and
  differential-algebraic equations}.
\newblock SIAM, 1998.

\bibitem[Brenan et~al.(1995)Brenan, Campbell, and Petzold]{brenan1995numerical}
Kathryn~Eleda Brenan, Stephen~L Campbell, and Linda~Ruth Petzold.
\newblock \emph{Numerical solution of initial-value problems in
  differential-algebraic equations}.
\newblock SIAM, 1995.

\bibitem[May-V{\'a}zquez et~al.(2022)May-V{\'a}zquez, G{\'o}mez-Castro,
  Rawlings, Rico-Ram{\'\i}rez, and Rodr{\'\i}guez-{\'A}ngeles]{may2022optimal}
Mayra~Margarita May-V{\'a}zquez, Fernando~Israel G{\'o}mez-Castro, Edna~Soraya
  Rawlings, Vicente Rico-Ram{\'\i}rez, and Mario~Alberto
  Rodr{\'\i}guez-{\'A}ngeles.
\newblock {Optimal control of a rate-based modelled batch distillation column:
  Initialization strategy}.
\newblock \emph{{Computers \& Chemical Engineering}}, 162:\penalty0 107811,
  2022.

\bibitem[Grossmann et~al.(2005)Grossmann, Aguirre, and
  Barttfeld]{grossmann2005optimal}
Ignacio~E Grossmann, P{\'\i}o~A Aguirre, and Mariana Barttfeld.
\newblock Optimal synthesis of complex distillation columns using rigorous
  models.
\newblock \emph{{Computers \& Chemical Engineering}}, 29\penalty0 (6):\penalty0
  1203--1215, 2005.

\bibitem[Powell(1994)]{powell1994direct}
Michael~JD Powell.
\newblock \emph{A direct search optimization method that models the objective
  and constraint functions by linear interpolation}.
\newblock Springer, 1994.

\bibitem[Johnson(2007)]{NLopt}
Steven~G. Johnson.
\newblock The {NLopt} nonlinear-optimization package.
\newblock \url{https://github.com/stevengj/nlopt}, 2007.

\bibitem[Byrd et~al.(2006)Byrd, Nocedal, and Waltz]{byrd2006k}
Richard~H Byrd, Jorge Nocedal, and Richard~A Waltz.
\newblock {K nitro: An integrated package for nonlinear optimization}.
\newblock \emph{Large-scale nonlinear optimization}, pages 35--59, 2006.

\bibitem[W{\"a}chter and Biegler(2006)]{wachter2006implementation}
Andreas W{\"a}chter and Lorenz~T Biegler.
\newblock On the implementation of an interior-point filter line-search
  algorithm for large-scale nonlinear programming.
\newblock \emph{Mathematical programming}, 106:\penalty0 25--57, 2006.

\bibitem[Cuthrell and Biegler(1989)]{cuthrell1989simultaneous}
James~E Cuthrell and Lorenz~T Biegler.
\newblock Simultaneous optimization and solution methods for batch reactor
  control profiles.
\newblock \emph{{Computers \& Chemical Engineering}}, 13\penalty0
  (1-2):\penalty0 49--62, 1989.

\bibitem[Poling et~al.(2001)Poling, Prausnitz, O'connell,
  et~al.]{poling2001properties}
Bruce~E Poling, John~M Prausnitz, John~P O'connell, et~al.
\newblock \emph{The properties of gases and liquids}, volume~5.
\newblock Mcgraw-hill New York, 2001.

\bibitem[Colwell(1981)]{colwell1981}
Charles~J. Colwell.
\newblock Clear liquid height and froth density on sieve trays.
\newblock \emph{Ind. Eng. Chem. Process Des. Dev.}, 20:\penalty0 298--307,
  1981.

\end{thebibliography}

\appendix
\section{Supplementary information on the simulation model}

\subsection{Submodels}
\label{app:submodels}

\revtext{
In this section, we record the specific thermodynamic submodels that we use in our simulation examples from Section~\ref{sec:study}. As pointed out before in Assumption~\ref{ass:submodels}, our general solution approach -- and in particular our index-reduction approach -- is valid also for other choices of the submodels. 
}

\subsubsection{Vapor-liquid equilibria}
\label{app:vle}

According to Raoult's law \cite{poling2001properties}, we assume that the vapor mole fractions of a $C$-component mixture at vapor-liquid equilibrium are related to the liquid mole fractions in the following way:
\begin{align} \label{eq:Raoult}
\fvlei(P,T,\bm{x}) := \frac{\Psati(T)}{P} \gamma_i(T, \bm{x}) x_i \qquad (i \in \{1,\dots,C\}). 
\end{align}
In this definition, $\Psati$ and $\gamma_i$ stand for the saturation pressure and the activity coefficient of component $i$, respectively. We model the saturation pressure (unit: Pa) according to the extended Antoine equation \cite{gmehling2019chemical}, that is,
\begin{align}
\Psati(T) := \exp\left(A_i + \frac{B_i}{C_i + T} + D_iT + E_i \ln T + F_i T^{G_i} \right),
\end{align}
where $A_i$, $B_i$, $C_i$, $D_i$, $E_i$, $F_i$, and $G_i$ are component-dependent fixed parameters. Also, for the activity coefficients, we use the non-random two-liquid model \cite{renon1968local, poling2001properties}, that is,
\begin{align}
    \ln \gamma_i(T, \bm{x}) & := \frac{\sum_{j=1}^C x_j \tau_{ji} G_{ji}}{\sum_{k=1}^C x_k G_{ki}} + \sum_{j=1}^C \frac{x_j G_{ij}}{\sum_{k=1}^C x_k G_{kj}} \left(\tau_{ij} - \frac{\sum_{m=1}^C x_m \tau_{mj} G_{mj}}{\sum_{k=1}^C x_k G_{kj}} \right) \\
    G_{ij}(T) & := \exp\left(-\alpha_{ij} \tau_{ij} \right) \\
    \tau_{ij}(T) & := a_{ij} + \frac{b_{ij}}{T} \qquad (i \neq j) \qquad \text{and} \qquad \tau_{ii}(T) := 0.
\end{align}
In the entire paper, we set the non-randomness parameter $\alpha_{ij} := 0.3$, following the recommendations in~\cite{renon1968local}. Consequently, the only remaining model parameters are the binary interaction parameters $a_{ij}, b_{ij}$ for all binary submixtures $\{i,j\} \subset \{1,\dots,C\}$. We denote the collection of the binary interaction parameters for a given binary submixture $\{i,j\} \subset \{1,\dots,C\}$ by
\begin{align}
\theta^{\{i,j\}} := (a_{ij},b_{ij},a_{ji},b_{ji}) \in \R^4
\end{align}
Analogously, we denote the collection of all binary interaction parameters by
\begin{align}
\theta := (\theta^B)_{B \in \mathcal{B}_C} \in \R^{4 d_C},
\end{align}  
where $\mathcal{B}_C$ denotes the set of all two-element subsets of $\{1,\dots,C\}$ and $d_C := |\mathcal{B}_C| = \binom{C}{2}$ is the number of all two-element subsets of $\{1,\dots,C\}$. 
The extended Antoine parameters used for the components in Section \ref{sec:study} are listed in Table \ref{tab:extended_antoine_paras} and they were obtained through Aspen.

\begin{table}[h]
    \centering
    \begin{tabular}{|c|c|c|c|c|c|c|c|}
        \hline
        Component & $A$ & $B$ & $C$ & $D$ & $E$ & $F$ & $G$ \\
        \hline
        acetone & $69.006$ & $-5599.6$ & $0.0$ & $0.0$ & $-7.0985$ & $6.2237\cdot 10^{-6}$ & $2.0$ \\
        methanol & $82.718$ & $-6904.5$ & $0.0$ & $0.0$ & $-8.8622$ & $7.4664\cdot 10^{-6}$ & $2.0$ \\
        butanol & $106.29$ & $-9866.4$ & $0.0$ & $0.0$ & $-11.655$ & $1.08 \cdot 10^{-17}$ & $6.0$ \\
        chloroform & $146.43$ & $-7792.3$ & $0.0$ & $0.0$ & $-20.614$ & $0.024578$ & $1.0$ \\
        2-butanone & $84.53$ & $-6787.2$ & $0.0$ & $0.0$ & $-9.2336$ & $9.0891 \cdot 10^{-9}$ & $3.0$ \\
        acetic acid & $53.27$ & $-6304.5$ & $0.0$ & $0.0$ & $-4.2985$ & $8.89 \cdot 10^{-18}$ & $6.0$ \\
        ethanol & $73.304$ & $-7122.3$ & $0.0$ & $0.0$ & $-7.1424$ & $2.89 \cdot 10^{-6}$ & $2.0$ \\
        ethyl acetate & $66.824$ & $-6227.6$ & $0.0$ & $0.0$ & $-6.41$ & $1.79\cdot 10^{-17}$ & $6.0$ \\
        toluene & $76.945$ & $-6729.8$ & $0.0$ & $0.0$ & $-8.179$ & $5.30\cdot 10^{-6}$ & $2.0$\\
        \hline 
    \end{tabular}
    \caption{Extended Antoine parameters of the components in Section \ref{sec:study}}
    \label{tab:extended_antoine_paras}
\end{table}

\revtext{ 
	In order to indicate the dependence of the activity coefficient models on the interaction parameters $\theta$, we write $\gamma_k = \gamma_k^\theta$ in Section~\ref{sec:study}. In this context, it should be noticed that for composition vectors
	\begin{align}
		\bm{x} \in \mathcal{X}_{ij} = \{\bm{x} \in \mathcal{X}: x_l = 0 \text{ for all } l \notin \{i,j\}\}
	\end{align}
	corresponding to the binary submixture $\{i,j\}$, 
	the model predictions $\gamma_i(T,\bm{x})$ and $\gamma_j(T,\bm{x})$ actually 
	depend only on the binary interaction parameters $\theta^{\{i,j\}}$ for the considered submixture $\{i,j\}$ (but not on any other submixtures). 
	In short,
	\begin{align} \label{eq:multi-component-NRTL-model-reduces-to-binary-NRTL-model-on-binary-submixtures}
		\gamma_k^\theta(T,\bm{x}) = \gamma_k^{\theta^{\{i,j\}}}(T,\bm{x}) \qquad (\bm{x} \in \mathcal{X}_{ij})
	\end{align}
	for $k \in \{i,j\}$. 
}

\subsubsection{Vapor and liquid molar enthalpies}
\label{app:molar_liquid_vapor_enthalpies}

As usual, we assume the vapor and liquid molar enthalpies to be given in the following way:
\begin{align} \label{eq:fhv-and-fhl}
\fhv(T,\bm{y}) := \bm{y} \cdot \hv(T) = \sum_{i=1}^C y_i \hvi(T)
\qquad \text{and} \qquad
\fhl(T,\bm{x}) := \bm{x} \cdot \hl(T) = \sum_{i=1}^C x_i \hli(T).
\end{align}
We use the following model for the partial molar liquid enthalpy (unit: J/mol) of component $i \in \{1, \dots, C\}$ at temperature $T$:
\begin{equation} \label{eq:hli}
    \hli(T) = \int_{T^{\text{ref}}}^T \cli(T') \text{d}T',
\end{equation}
where $T^\text{ref}$ is an arbitrary but fixed reference temperature in K and $\cli(T)$ is the liquid molar heat capacity (unit: J/(mol K)) of component $i \in  \{1,\dots,C\}$ at temperature $T$. We use the following fourth-degree polynomial model to describe it (DIPPR model 100):
\begin{equation} \label{eq:cli}
    \cli(T) := C_{1i} + C_{2i}T + C_{3i}T^2 + C_{4i}T^3 + C_{5i} T^4
\end{equation}
Also, for the partial molar vapor enthalpy (unit: J/mol) of component $i \in \{1,\dots,C\}$ at temperature $T$, we use the following model:
\begin{equation} \label{eq:hvi}
   \hvi(T) := \hli(T) + h_{\text{vap}, i}(T) = \int_{T^\text{ref}}^T \cli(T') \text{d}T' + h_{\text{vap}, i}(T),
\end{equation}
where $h_{\text{vap}, i}(T)$ is the vaporization enthalpy for component $i$, for which we use the Clausius-Clapeyron model \cite{poling2001properties}:
\begin{equation}
    \hvapi(T) := \text{R}T^2 \frac{\text{d}}{\text{d}T} \ln \Psati(T).
\end{equation}
As usual, $\text{R} := \SI{8.31}{\kilogram \cdot \meter^2 \cdot \second^{-2} \cdot \kelvin^{-1} \cdot \mole^{-1}}$ denotes the universal gas constant. The DIPPR 100 parameters we used for the components in Section \ref{sec:study} are listed in Table \ref{tab:dippr_100_parameters} and they were obtained through Aspen.

\begin{table}[h]
    \centering
    \begin{tabular}{|c|c|c|c|c|c|}
        \hline
        Component & $C_1$ & $C_2$ & $C_3$ & $C_4$ & $C_5$ \\
        \hline
        acetone & $135.6$ & $-0.177$ & $0.00028367$ & $6.89\cdot10^{-7}$ & $0.0$ \\
        methanol & $256.04$ & $-2.7414$ & $0.014777$ & $-0.000035078$ & $3.27\cdot10^{-8}$ \\
        butanol & $191.2$ & $-0.7304$ & $0.0022998$ & $0.0$ & $0.0$ \\
        chloroform & $124.85$ & $-0.16634$ & $0.00043209$ & $0.0$ & $0.0$ \\
        2-butanone & $3162.314$ & $-32.5937$ & $0.1324$ & $-0.000238$ & $1.62\cdot10^{-7}$ \\
        acetic acid & $695.074$ & $-6.3595$ & $0.0254$ & $-4.73\cdot10^{-5}$ & $3.51\cdot10^{-8}$ \\
        ethanol & $8412.349$ & $-89.4774$ & $0.36054$ & $-0.000646$ & $4.38\cdot10^{-7}$ \\
        ethyl acetate & $4633.987$ & $-48.718$ & $0.1978$ & $-0.000355$ & $2.41\cdot10^{-7}$ \\
        toluene & $451.303$ & $-4.222301$ & $0.01928$ & $-3.60\cdot10^{-5}$ & $2.57\cdot10^{-8}$ \\
        \hline
    \end{tabular}
    \caption{DIPPR 100 parameters of the components in Section \ref{sec:study}}
    \label{tab:dippr_100_parameters}
\end{table}

\subsubsection{Stage holdup}
\label{app:holdup}

We assume that the liquid holdup $n^j$ on every stage $j \in \{2,\dots,S\}$ except the pot is proportional to the liquid downflow rate $L^{j-1}$ to stage $j - 1$. In other words, this means that our holdup model is given by 
\begin{equation} \label{eq:fhold}
    \fhold(L^{j-1}) := c_{\text{holdup}}\frac{h}{S-1} L^{j-1} \qquad (j \in \{2,\dots,S\})  
\end{equation}
where $h$ is the height of the column above the pot, so that $\frac{h}{S-1}$ is the height of a single stage. Also, $c_\text{holdup}$ is defined as follows:
\begin{equation}
    c_\text{holdup} := \frac{n_\text{ref} / 100}{B_{\mathrm{ref}} / 3600},
\end{equation}
where $n_\text{ref}$ (expressed in \%) is a reference volumetric holdup, and $B_{\mathrm{ref}}$ is a reference liquid load (expressed in \SI{}{\meter \per \hour}). In this study, the parameters $n_\text{ref}$ and $B_{\mathrm{ref}}$ are set to \SI{4.2}{\percent} and \SI{5}{\meter \per \hour}, respectively.

\subsection{Index reduction by differentiation of the vapor summation equations}
\label{app:index-reduction-y-sum-differentiation}

\revtext{
In this section, we outline the conventional index reduction approach from the literature~\cite{cervantes1998, cervantes2000, biegler2002, ragunathan2004, qian2023nonlinear} to make the differences to our approach clear. In essence, the conventional approach is based on differentiating the vapor summation equations 
to obtain algebraic versions of the enthalpy balance equations. 
Specifically, the conventional approach in its simplest form proceeds as follows: In order to get rid of the zero-row problem (first bullet point at the end of Section~\ref{sec:system-equations}), the component mole balance equations for component $C$ are dropped. As a consequence, the liquid mole fractions $x_C^1, \dots, x_C^S$ are no longer differential variables, but they become algebraic variables and the derivatives of the liquid summation equations w.r.t.~these new algebraic variables become non-zero. In particular, this fixes the zero-row problem. In return for the new algebraic variables $x_C^1, \dots, x_C^S$, however, $S$ new algebraic equations have to be added because otherwise the Jacobian of the algebraic equations w.r.t.~the algebraic variables is not even a square matrix. In the conventional approach, such new algebraic equations are obtained by combining the vapor summation equations~\eqref{eq:y-sum,tot-and-modif-comp-mole-balance} (Proposition~\ref{prop:equivalence-of-basic-and-modified-basic-system}) with the vapor-fraction-defining equations~\eqref{eq:y-def,tot-and-modif-comp-mole-balance}. In this manner, one arrives at the equilibrium summation equations~\eqref{eq:vlesum,y-differentiation} below. It should be noticed, however, that these new algebraic equations~\eqref{eq:vlesum,y-differentiation} still do not contain the vapor stream variables $V^1, \dots, V^S$ so that the zero-column problem (second bullet point at the end of Section~\ref{sec:system-equations}) is not addressed yet. 
In order to address this, the equilibrium summation equations~\eqref{eq:vlesum,y-differentiation} are differentiated along solutions, which leads to 
\begin{align} \label{eq:T-dot}
\dot{T}^j = -\frac{\sum_{i=1}^C \big( \partial_{\bm{x}}\fvlei(P,T^j,\bm{x}^j) \cdot \dot{\bm{x}}^j + \partial_P \fvlei(P,T^j,\bm{x}^j) \dot{P} \big)}{\sum_{i=1}^C \partial_T \fvlei(P,T^j,\bm{x}^j)}
\end{align}
provided that $\sum_{i=1}^C \partial_T \fvlei(P(t),T^j(t),\bm{x}^j(t)) \ne 0$ along the solution. Additionally, the enthalpy-defining equations~\eqref{eq:enth,tot-and-modif-comp-mole-balance} are differentiated along solutions and then~\eqref{eq:T-dot} is used to obtain
\begin{align} \label{eq:H-dot}
\dot{H}^j 
&= \dot{n}^j \fhl(T^j,\bm{x}^j) + n^j \big( \partial_T \fhl(T^j,\bm{x}^j) \dot{T}^j + \partial_{\bm{x}} \fhl(T^j,\bm{x}^j) \cdot \dot{\bm{x}}^j \big) \nonumber \\
&= \dot{n}^j \fhl(T^j,\bm{x}^j) + n^j \bm{a}(P,T^j,\bm{x}^j) \cdot \dot{\bm{x}}^j + n^j b(P,T^j,\bm{x}^j) \dot{P},
\end{align}
where $\bm{a}(P,T,\bm{x})$ and $b(P,T,\bm{x})$ are defined as follows:
\begin{align}
\bm{a}(P,T,\bm{x}) &:= \partial_{\bm{x}} \fhl(T,\bm{x}) - \partial_T \fhl(T,x) \frac{\sum_{i=1}^C \partial_{\bm{x}}\fvlei(P,T,\bm{x})}{\sum_{i=1}^C \partial_T \fvlei(P,T,\bm{x})} \in \R^C, \label{eq:a(P,T,x)-def}\\
b(P,T,\bm{x}) &:= - \partial_T \fhl(T,x) \frac{\sum_{i=1}^C \partial_P \fvlei(P,T,\bm{x})}{\sum_{i=1}^C \partial_T \fvlei(P,T,\bm{x})} \in \R. \label{eq:b(P,T,x)-def}
\end{align}
Inserting the equations~\eqref{eq:H-dot} into the enthalpy balance equations~\eqref{eq:enth,tot-and-modif-comp-mole-balance} and using the total and component mole balance equations~\eqref{eq:tot-mole,tot-and-modif-comp-mole-balance} and~\eqref{eq:comp-mole,tot-and-modif-comp-mole-balance} to eliminate the time derivatives $\dot{n}^j$ and $\dot{\bm{x}}^j$, one then obtains algebraic versions of the enthalpy balance equations, namely the algebraicized enthalpy balance  equations~\eqref{eq:enth,y-sum-differentiation} below. Summarizing, the conventional index reduction approach leads to the alternative system equations \eqref{eq:tot-mole,y-sum-differentiation}-\eqref{eq:vlesum,y-differentiation}. Specifically, the alternative system equations consist of the total mole balances
\begin{subequations} \label{eq:tot-mole,y-sum-differentiation}
	\begin{align}
		\dot{n}^1 & = L^1 - V^1  \label{eq:tot-mole-1,y-sum-differentiation} \\ 
		\dot{n}^j & = L^j - V^j - L^{j-1} + V^{j-1} \label{eq:tot-mole-j,y-sum-differentiation} \\
		\dot{n}^S &= -\epsilon V^S - L^{S-1} + V^{S-1}, \label{eq:tot-mole-S,y-sum-differentiation}     
    \end{align}
\end{subequations}
of the component mole balances for the first $C-1$ components 
\begin{subequations} \label{eq:comp-mole,y-sum-differentiation}
    \begin{align}
		\dot{x}_i^1 & = \big( L^1 (x_i^2 - x_i^1) - V^1 (y_i^1 - x_i^1)\big)/n^1 \qquad (i\in \{1,\dots,C-1\})  \label{eq:comp-mole-1,y-sum-differentiation} \\        
        \dot{x}_i^j & = \big( L^j (x_i^{j+1} - x_i^j) - V^j (y_i^j - x_i^j) + V^{j-1} (y_i^{j-1} - x_i^j) \big)/n^j \qquad (i\in \{1,\dots,C-1\}) \label{eq:comp-mole-j,y-sum-differentiation} \\
        \dot{x}_i^S &= \big( \epsilon V^S (x_i^S - y_i^S) + V^{S-1} (y_i^{S-1} - x_i^S) \big)/n^S \qquad (i\in \{1,\dots,C-1\}), \label{eq:comp-mole-S,y-sum-differentiation}
    \end{align}
\end{subequations}
and of the algebraicized enthalpy balances
\begin{subequations} \label{eq:enth,y-sum-differentiation}
    \begin{align} 
    		&(L^1-V^1) \fhl(T^1,\bm{x}^1) + \bm{a}(P,T^1,\bm{x}^1) \cdot \big( L^1 (\bm{x}^2 - \bm{x}^1) - V^1 (\bm{y}^1 - \bm{x}^1)\big) + n^1 b(P,T^1,\bm{x}^1) \dot{P} \nonumber \\
    		& \quad = L^1 \fhl(T^2,\bm{x}^2) - V^1 \fhv(T^1,\bm{y}^1) + Q \label{eq:enth-1,y-sum-differentiation} \\
    		&(L^j - V^j - L^{j-1} + V^{j-1}) \fhl(T^j,\bm{x}^j) + \bm{a}(P,T^j,\bm{x}^j) \cdot \big( L^j (\bm{x}^{j+1} - \bm{x}^j) - V^j (\bm{y}^j - \bm{x}^j) \nonumber \\
    		&\qquad \qquad + V^{j-1}(\bm{y}^{j-1} - \bm{x}^j) \big) + n^j b(P,T^j,\bm{x}^j) \dot{P} \nonumber \\
    		& \quad = L^j \fhl(T^{j+1},\bm{x}^{j+1}) - V^j \fhv(T^j,\bm{y}^j) - L^{j-1} \fhl(T^j,\bm{x}^j) + V^{j-1}\fhv(T^{j-1},\bm{y}^{j-1}) \label{eq:enth-j,y-sum-differentiation} \\
        &(-\epsilon V^S - L^{S-1} + V^{S-1}) \fhl(T^S,\bm{x}^S) + \bm{a}(P,T^S,\bm{x}^S) \cdot \big( -\epsilon V^S (\bm{y}^S - \bm{x}^S) + V^{S-1}(\bm{y}^{S-1} - \bm{x}^S) \big) \nonumber \\
        &\qquad \qquad + n^S b(P,T^S,\bm{x}^S) \dot{P} \nonumber \\
        & \quad = (1-\epsilon)V^S \fhl(\Tcond,\bm{y}^S) - V^S \fhv(T^S,\bm{y}^S) - L^{S-1} \fhl(T^S,\bm{x}^S) + V^{S-1}\fhv(T^{S-1},\bm{y}^{S-1}). \label{eq:enth-S,y-sum-differentiation} 
    \end{align}
\end{subequations}
Additionally, they consist of the liquid summation conditions, the equilibrium-defining equations, the enthalpy-defining equations, and the holdup equations:
\begin{align}
		x_C^j &= 1 - \sum_{i=1}^{C-1} x_i^j \qquad (j \in \{1,\dots,S\}) \label{eq:x-sum,y-sum-differentiation} \\
        y_i^j &= \fvlei(P,T^j,\bm{x}^j) \qquad (i \in \{1,\dots,C\}, j \in \{1,\dots,S\}) \label{eq:y-def,y-sum-differentiation} \\
        H^j & = n^j \fhl(T^j,\bm{x}^j) \qquad (j \in \{1,\dots,S\}) \label{eq:enth-def,y-sum-differentiation} \\
        n^j & = \fhold(L^{j-1}) \qquad (j \in \{2,\dots,S\}), \label{eq:holdup-j,y-sum-differentiation}
\end{align} 
And finally, the alternative system equations include the equilibrium summation equations, that is,
\begin{align} \label{eq:vlesum,y-differentiation}
	\sum_{i=1}^C \fvlei(P,T^j,\bm{x}^j) = 1 \qquad (j \in \{1,\dots,S\}). 
\end{align}
It is straightforward to see that this alternative system~\eqref{eq:tot-mole,y-sum-differentiation}-\eqref{eq:vlesum,y-differentiation} is equivalent to the original system~\eqref{eq:tot-mole,tot-and-modif-comp-mole-balance}-\eqref{eq:holdup,tot-and-modif-comp-mole-balance} (Proposition~\ref{prop:equivalence-of-modified-basic-and-alternative-system}).
} 

\subsection{Simple equivalences}
\label{app:simple-equivalences}

In this section, we record the simple and straightforward equivalences of different variants of the system equations referred to in Sections~\ref{sec:system-equations} and~\ref{sec:index-reduction}. In the following, we use the following notion of equivalence of two systems of differential-algebraic equations. We say that two systems of differential-algebraic equations are \emph{equivalent as long as a certain condition $c$ is satisfied} iff their sets of solutions satisfying $c$ are equal (that is, every solution of the first system of equations satisfying $c$ is also a solution of the second system of equations satisfying $c$, and vice versa). 
\revtext{
We begin by showing that the system equations~\eqref{eq:comp-mole,tot-and-modif-comp-mole-balance}-\eqref{eq:holdup,tot-and-modif-comp-mole-balance} are equivalent to the even more basic system which instead of the re-formulated component mole balance equations~\eqref{eq:comp-mole,tot-and-modif-comp-mole-balance} features the original 
component mole balance equations
\begin{subequations} \label{eq:comp-mole,original}
	\begin{align}
		\dot{(n^1 x_i^1)} & = L^1 x_i^2 -V^1 y_i^1 \qquad (i\in \{1,\dots,C\})  \label{eq:comp-mole-1,original} \\
		\dot{(n^j x_i^j)} & = L^j x_i^{j+1} - V^j y_i^j - L^{j-1}x_i^j + V^{j-1}y_i^{j-1} \qquad (i\in \{1,\dots,C\}) \label{eq:comp-mole-j,original} \\
		\dot{(n^S x_i^S)} &= -\epsilon V^S y_i^S - L^{S-1}x_i^S + V^{S-1} y_i^{S-1} \qquad (i\in \{1,\dots,C\}) \label{eq:comp-mole-S,original}
	\end{align}
\end{subequations}
and instead of the total mole balance equations~\eqref{eq:tot-mole,tot-and-modif-comp-mole-balance} features the vapor summation equations~\eqref{eq:y-sum,tot-and-modif-comp-mole-balance}.
}

\begin{prop} \label{prop:equivalence-of-basic-and-modified-basic-system}
As long as $\epsilon(t)$, $n^j(t)$ and $V^j(t)$ are non-zero for all $j \in \{1,\dots,S\}$, the system~\eqref{eq:tot-mole,tot-and-modif-comp-mole-balance}-\eqref{eq:holdup,tot-and-modif-comp-mole-balance} is equivalent to the variant which instead of~\eqref{eq:comp-mole,tot-and-modif-comp-mole-balance} features the original component mole balance equations~\eqref{eq:comp-mole,original} and instead of~\eqref{eq:tot-mole,tot-and-modif-comp-mole-balance} features the vapor summation equations~\eqref{eq:y-sum,tot-and-modif-comp-mole-balance}.
\end{prop}

\begin{proof}
\revtext{
We call the variant of system~\eqref{eq:tot-mole,tot-and-modif-comp-mole-balance}-\eqref{eq:holdup,tot-and-modif-comp-mole-balance} which replaces~\eqref{eq:comp-mole,tot-and-modif-comp-mole-balance} by \eqref{eq:comp-mole,original} and \eqref{eq:tot-mole,tot-and-modif-comp-mole-balance} by \eqref{eq:y-sum,tot-and-modif-comp-mole-balance} the basic system, for the sake of brevity. 
Suppose first that the equations~\eqref{eq:tot-mole,tot-and-modif-comp-mole-balance}-\eqref{eq:holdup,tot-and-modif-comp-mole-balance} are satisfied and that
\begin{align} \label{eq:eps,n,V-not-zero}
\epsilon(t) \ne 0, \qquad n^j(t) \ne 0, \qquad V^j(t) \ne 0 \qquad (j \in \{1,\dots,S\})
\end{align}
for all $t$ in the solution interval. It then follows by the liquid summation equations~\eqref{eq:x-sum,tot-and-modif-comp-mole-balance} and the component mole balance equations~\eqref{eq:comp-mole,tot-and-modif-comp-mole-balance} that
\begin{align} \label{eq:V*y-sum}
V^j \sum_{i=1}^C y_i^j = V^j \qquad (j \in \{1,\dots,S-1\})
\qquad \text{and} \qquad
\epsilon V^S \sum_{i=1}^C y_i^S = \epsilon V^S.
\end{align}
So, by~\eqref{eq:eps,n,V-not-zero} and~\eqref{eq:V*y-sum}, the vapor summation equations~\eqref{eq:y-sum,tot-and-modif-comp-mole-balance} follow. Inserting the total and component mole balance equations~\eqref{eq:tot-mole,tot-and-modif-comp-mole-balance}-\eqref{eq:comp-mole,tot-and-modif-comp-mole-balance} in
\begin{align} \label{eq:n*x-dot}
\dot{(n^jx_i^j)} = \dot{n}^j x_i^j + n^j \dot{x}_i^j \qquad (i \in \{1,\dots,C\}, j \in \{1,\dots,S\}),
\end{align}
it further follows that the original component mole balance equations~\eqref{eq:comp-mole,original} are satisfied as well. And therefore, the equations of the basic system are satisfied, as desired.
Suppose now conversely that the equations of the basic system are satisfied and that~\eqref{eq:eps,n,V-not-zero} holds true. It then follows by summing the equations~\eqref{eq:n*x-dot} over all components and by using the original component mole balance equations~\eqref{eq:comp-mole,original} together with the liquid and the vapor summation equations that the total mole balance equations~\eqref{eq:tot-mole,tot-and-modif-comp-mole-balance} are satisfied. Inserting the original  component mole balance equations~\eqref{eq:comp-mole,original} and the total mole balance equations~\eqref{eq:tot-mole,tot-and-modif-comp-mole-balance}, in turn, into~\eqref{eq:n*x-dot}, we also obtain the re-formulated component mole balance equations~\eqref{eq:comp-mole,tot-and-modif-comp-mole-balance}. And therefore, the equations~\eqref{eq:tot-mole,tot-and-modif-comp-mole-balance}-\eqref{eq:holdup,tot-and-modif-comp-mole-balance} are satisfied, as desired. 
}
\end{proof}

\begin{prop} \label{prop:equivalence-of-modified-basic-and-unperturbed-system}
As long as $n^j(t)$ are non-zero for all $j \in \{1,\dots,S\}$, the system~\eqref{eq:tot-mole,tot-and-modif-comp-mole-balance}-\eqref{eq:holdup,tot-and-modif-comp-mole-balance} is  equivalent to the intermediary unperturbed system, that is, to system~\eqref{eq:tot-mole,x-sum-differentiation-perturbed}-\eqref{eq:aux-alg-eq,x-sum-differentiation-perturbed} with perturbation parameters $\delta_i^j = 0$. 
\end{prop}

\begin{proof}
It is clear by the arguments in Section~\ref{sec:index-reduction} that if the equations~\eqref{eq:tot-mole,tot-and-modif-comp-mole-balance}-\eqref{eq:holdup,tot-and-modif-comp-mole-balance} are satisfied with
\begin{align} \label{eq:n-not-zero}
n^j(t) \ne 0 \qquad (j \in \{1,\dots,S\})
\end{align}
for all $t$ in the solution interval, then so are the equations of the unperturbed system. 
Suppose now conversely that the equations of the unperturbed system are satisfied and that~\eqref{eq:n-not-zero} holds true. In order to show that \eqref{eq:tot-mole,tot-and-modif-comp-mole-balance}-\eqref{eq:holdup,tot-and-modif-comp-mole-balance} are satisfied, we have only to establish the component mole balance equations~\eqref{eq:comp-mole,tot-and-modif-comp-mole-balance} for component $C$. (After all, the other equations of \eqref{eq:tot-mole,tot-and-modif-comp-mole-balance}-\eqref{eq:holdup,tot-and-modif-comp-mole-balance} are already part of the perturbed system.) In order to do so, we first observe by the new algebraic equations~\eqref{eq:aux-alg-eq,x-sum-differentiation} and the liquid summation equations that the following modified vapor summation equations
\begin{align} \label{eq:V*y-sum,2}
V^j \sum_{i=1}^C y_i^j = V^j \qquad (j \in \{1,\dots,S-1\})
\qquad \text{and} \qquad
\epsilon V^S \sum_{i=1}^C y_i^S = \epsilon V^S
\end{align}
hold true for all stages. It then follows by the liquid summation equations, the above modified vapor summation equations, and the component mole balance equations for the first $C-1$ components, that the component mole balance equations~\eqref{eq:comp-mole,tot-and-modif-comp-mole-balance} also hold for component $C$, as desired.
\end{proof}

\color{black}
\begin{prop} \label{prop:equivalence-of-modified-basic-and-alternative-system}
As long as $n^j(t)$ and $\sum_{i=1}^C \partial_T \fvlei(P(t),T^j(t),\bm{x}^j(t))$ are non-zero for all $j \in \{1,\dots,S\}$, the system~\eqref{eq:tot-mole,tot-and-modif-comp-mole-balance}-\eqref{eq:holdup,tot-and-modif-comp-mole-balance} is  equivalent to the alternative system~\eqref{eq:tot-mole,y-sum-differentiation}-\eqref{eq:vlesum,y-differentiation}.
\end{prop}

\begin{proof}
It is clear by the arguments in \ref{app:index-reduction-y-sum-differentiation} that if the equations~\eqref{eq:tot-mole,tot-and-modif-comp-mole-balance}-\eqref{eq:holdup,tot-and-modif-comp-mole-balance} are satisfied with
\begin{align} \label{eq:n,partial_T-fvle-not-zero}
n^j(t) \ne 0 
\qquad \text{and} \qquad
\sum_{i=1}^C \partial_T \fvlei(P(t),T^j(t),\bm{x}^j(t)) \ne 0 \qquad (j \in \{1,\dots,S\})
\end{align}
for all $t$ in the solution interval, then so are the equations of the alternative system. 
Suppose now conversely that the equations \eqref{eq:tot-mole,y-sum-differentiation}-\eqref{eq:vlesum,y-differentiation} of the alternative system are satisfied and \eqref{eq:n,partial_T-fvle-not-zero} holds true. In order to show that \eqref{eq:tot-mole,tot-and-modif-comp-mole-balance}-\eqref{eq:holdup,tot-and-modif-comp-mole-balance} are satisfied, we have only to establish the component mole balance equations~\eqref{eq:comp-mole,tot-and-modif-comp-mole-balance} for component $C$ and the enthalpy balance equations~\eqref{eq:enth,tot-and-modif-comp-mole-balance}. (After all, the other equations of \eqref{eq:tot-mole,tot-and-modif-comp-mole-balance}-\eqref{eq:holdup,tot-and-modif-comp-mole-balance} are already part of the alternative system.) 
In order to see that the component mole balance equations~\eqref{eq:comp-mole,tot-and-modif-comp-mole-balance} are satisfied also for component $C$, notice that
\begin{align} \label{eq:dot-x_C}
\dot{x}_C^j = -\sum_{i=1}^{C-1} \dot{x}_i^j
\end{align}
by the liquid summation equations~\eqref{eq:x-sum,y-sum-differentiation}. Also, 
notice that
\begin{align} \label{eq:sums-of-mole-fraction-differences}
\sum_{i=1}^{C-1} (x_i^k - x_i^l) = -(x_C^k - x_C^l), \qquad
\sum_{i=1}^{C-1} (y_i^k - x_i^l) = -(y_C^k - x_C^l), \qquad
\sum_{i=1}^{C-1} (y_i^k - y_i^l) = -(y_C^k - y_C^l)
\end{align}
for arbitary $k,l \in \{1,\dots,S\}$ by the liquid and the equilibrium summation equations~\eqref{eq:x-sum,y-sum-differentiation} and~\eqref{eq:vlesum,y-differentiation} in conjunction with the vapor-fraction-defining equations~\eqref{eq:y-def,y-sum-differentiation}. Since now the right-hand sides of the component mole balance equations~\eqref{eq:comp-mole,y-sum-differentiation} are affine combinations of differences of the form $x_i^k - x_i^l$ or $y_i^k - x_i^l$ or $y_i^k - y_i^l$ for $i \in \{1,\dots,C-1\}$ and $k,l \in \{1,\dots,S\}$, it follows by combining~\eqref{eq:dot-x_C} and \eqref{eq:sums-of-mole-fraction-differences} that the component mole balance equations~\eqref{eq:comp-mole,tot-and-modif-comp-mole-balance} are satisfied also for component $C$, as desired. 
In order to see that the enthalpy balance equations~\eqref{eq:enth,tot-and-modif-comp-mole-balance} are satisfied, notice that
\begin{align} \label{eq:H-dot,alternative}
\dot{H}^j 
= \dot{n}^j \fhl(T^j,\bm{x}^j) + n^j \bm{a}(P,T^j,\bm{x}^j) \cdot \dot{\bm{x}}^j + n^j b(P,T^j,\bm{x}^j) \dot{P}
\end{align}
by the enthalpy-defining equations~\eqref{eq:enth-def,y-sum-differentiation} and by~\eqref{eq:a(P,T,x)-def} and~\eqref{eq:b(P,T,x)-def}. Also, notice that the right-hand side of~\eqref{eq:H-dot,alternative} is equal to the left-hand side of the $j$th algebraicized enthalpy balance equation~\eqref{eq:enth,y-sum-differentiation} by virtue of the total mole balance equations~\eqref{eq:tot-mole,y-sum-differentiation} and the component mole balance equations~\eqref{eq:comp-mole,y-sum-differentiation} for the first $C-1$ components and for component $C$ (which we have  established in the first step of the proof). And therefore, the enthalpy balance equations~\eqref{eq:enth,tot-and-modif-comp-mole-balance} follow, as desired. 
\end{proof}
\color{black}

\subsection{A physics-based motivation of the perturbed system}
\label{app:physical-motivation}

In this section, we show how one can derive the perturbed system equations~\eqref{eq:tot-mole,x-sum-differentiation-perturbed}-\eqref{eq:aux-alg-eq,x-sum-differentiation-perturbed} based on physical considerations by relaxing the zero vapor holdup assumption~\eqref{eq:zero-vapor-holdup}. Specifically, assume that there is a small vapor holdup $\nv^j \ne 0$ on each of the equilibrium stages so that the total holdup $n^j = \nl^j + \nv^j$ is no longer equal to the liquid holdup $\nl^j$. In this case, the component mole balances around the stages become
\begin{subequations} \label{eq:comp-mole,vapor-holdup}
	\begin{align}
		\dot{(n^1 z_i^1)} & = L^1 x_i^2 -V^1 y_i^1 \qquad (i\in \{1,\dots,C\})  \label{eq:comp-mole-1,vapor-holdup} \\
		\dot{(n^j z_i^j)} & = L^j x_i^{j+1} - V^j y_i^j - L^{j-1}x_i^j + V^{j-1}y_i^{j-1} \qquad (i\in \{1,\dots,C\}) \label{eq:comp-mole-j,vapor-holdup} \\
		\dot{(n^S z_i^S)} &= -\epsilon V^S y_i^S - L^{S-1}x_i^S + V^{S-1} y_i^{S-1} \qquad (i\in \{1,\dots,C\}) \label{eq:comp-mole-S,vapor-holdup}
	\end{align}
\end{subequations}
where $\bm{z}^j$ is the total composition vector on stage $j$. In terms of the total stage volume $\mathcal{V}$, the vapor volume fraction $\nu := \mathcal{V}_{\mathrm{v}}/\mathcal{V}$, and the molar densities $\rhol$ and $\rhov$, we can then write the total composition vector as
\begin{align}
z_i^j &= \frac{\nl^j x_i^j + \nv^j y_i^j}{\nl^j + \nv^j} = x_i^j + \beta (y_i^j-x_i^j) \label{eq:z-small-perturbation-of-x}\\
\beta &:= \frac{\nv^j}{\nl^j + \nv^j} = \frac{\mathcal{V}\rhov \nu}{\mathcal{V}(\rhol (1-\nu) + \rhov \nu)}
= \frac{\nu}{(\rhol / \rhov)(1-\nu) + \nu}. \label{eq:beta}
\end{align}
Since typically $\mu := \rhol / \rhov \approx 1000$ and \revtext{since at least for the data reported in~\cite{colwell1981} $\nu$ is} less than $0.5$, we see from~\eqref{eq:beta} that $\beta$ is a small number typically no larger than $1/1000$. So, by~\eqref{eq:z-small-perturbation-of-x}, the total composition vector $\bm{z}^j$ is a very small perturbation of the liquid composition vector $\bm{x}^j$. In short,
\begin{align} \label{eq:z=x+delta}
z_i^j = x_i^j + \beta (y_i^j-x_i^j) = x_i^j + \delta_i^j
\end{align}
with $|\delta_i^j| \ll 1$. 
If we additionally assume this perturbation $\delta_i^j$ to be time-independent and then plug~\eqref{eq:z=x+delta} into the component mole balances~\eqref{eq:comp-mole,vapor-holdup}, then we arrive precisely at the perturbed component mole balances~\eqref{eq:comp-mole,x-sum-differentiation-perturbed} of our perturbed system. And combining~\eqref{eq:comp-mole,x-sum-differentiation-perturbed}, in turn, with the total mole balances (that is, the summed versions of~\eqref{eq:comp-mole,vapor-holdup}), we arrive precisely at the perturbed additional algebraic equations~\eqref{eq:aux-alg-eq,x-sum-differentiation-perturbed}. 

\subsection{Index considerations}
\label{app:index-considerations}

\revtext{
In this section, we work out rigorous conditions under which our perturbed and our relaxed systems~\eqref{eq:ode,perturbed}-\eqref{eq:alg,perturbed} and~\eqref{eq:ode,relaxed}-\eqref{eq:alg,relaxed} are non-singular along their solutions (\ref{app:perturbation-approach} and \ref{app:relaxation-approach}). And, for the sake of completeness, we do the same for the alternative system~\eqref{eq:tot-mole,y-sum-differentiation}-\eqref{eq:vlesum,y-differentiation} (\ref{app:y-summation-approach}).
}

\subsubsection{Index reduction by differentiation of the liquid summation equations: perturbation-based approach}
\label{app:perturbation-approach}

We collect the differential and the algebraic variables of the perturbed system~\eqref{eq:ode,perturbed}-\eqref{eq:alg,perturbed} in the (column) vectors $\xi$ and $\eta$:
\begin{align*}
\xi &:= (n^1, \dots, n^S, \hat{\bm{x}}^1, \dots, \hat{\bm{x}}^S, H^1, \dots, H^S)^\top \in \R^{(2+C)S} \nonumber\\
\eta &:= (V^S, \dots, V^1, \bm{y}^1, \dots, \bm{y}^S, T^1, \dots, T^S, x_C^1, \dots, x_C^S, L^1, \dots, L^{S-1})^\top \in \R^{(4+C)S-1},
\end{align*}
where $\hat{\bm{x}}^j := (x_1^j, \dots, x_{C-1}^j)$. Additionally, we denote the algebraic equations of the system~\eqref{eq:ode,perturbed}-\eqref{eq:alg,perturbed} by $\hold^j$, $\xsum^j$, $\ydef_i^j$, $\enthdef^j$, $\aux_\delta^j$, respectively. Specifically, 
\begin{align*}
\hold^2, \dots, \hold^S, \qquad
\xsum^j, \qquad
\ydef_i^j, \qquad
\enthdef^j, \qquad
\aux_\delta^j \qquad (j \in \{1,\dots,S\})
\end{align*}
denote the functions of the state variables $\zeta := (\xi,\eta)$ defined as left-hand side minus right-hand side of the holdup equations~\eqref{eq:holdup-j,x-sum-differentiation-perturbed}, of the liquid summation equations~\eqref{eq:x-sum,x-sum-differentiation-perturbed}, of the vapor-fraction-defining equations~\eqref{eq:y-def,x-sum-differentiation-perturbed}, of the enthalpy-defining equations~\eqref{eq:enth-def,x-sum-differentiation-perturbed}, and of the auxiliary algebraic equations~\eqref{eq:aux-alg-eq,x-sum-differentiation-perturbed}, respectively. We collect these functions in the following order:
\begin{align*}
g_\delta := (\aux_\delta^S, \dots, \aux_\delta^1, \bm{\ydef}^1, \dots, \bm{\ydef}^S, \enthdef^1, \dots, \enthdef^S, \xsum^1, \dots, \xsum^S, \hold^2, \dots, \hold^S)^\top,
\end{align*}
where $\bm{\ydef}^j := (\ydef_1^j, \dots, \ydef_C^j)$, of course. With these orderings of the algebraic equations and the algebraic variables, the Jacobian $D_\eta g_\delta(\zeta)$ of the algebraic equations w.r.t.~the algebraic variables becomes upper triangular (Figure~\ref{fig:incidence-matrix-index-reduced-for-S=3}). 

\begin{lm} \label{lm:Jacobian-upper-triangular}
The Jacobian $D_\eta g_{\delta}(\zeta)$ is upper triangular for every state $\zeta \in Z$ and every perturbation parameter $\delta \in [0,\infty)^{S\cdot C}$. Also, the diagonal elements of $D_\eta g_{\delta}(\zeta)$ are given by~\eqref{eq:aux^S-dell-V^S}-\eqref{eq:holdup^j+1-dell-L^j}. 
\end{lm}

\begin{proof}
It is straightforward to verify from the definition of $g_\delta$ and the underlying algebraic equations of the perturbed system~\eqref{eq:tot-mole,x-sum-differentiation-perturbed}-\eqref{eq:aux-alg-eq,x-sum-differentiation-perturbed} that the Jacobian $D_\eta g_\delta(\zeta)$ is upper triangular. It is also straightforward to verify that its diagonal elements or diagonal blocks (in the case of the vapor-fraction-defining equations) are given by
\begin{align}
\partial_{V^S} \aux_\delta^S(\zeta) &= -\epsilon \sum_{i=1}^C \big( y_i^S - x_i^S - \delta_i^S)/n^S, \label{eq:aux^S-dell-V^S}\\
\partial_{V^j} \aux_\delta^j(\zeta) &= -\sum_{i=1}^C \big( y_i^j - x_i^j - \delta_i^j)/n^j \qquad (j \in \{S-1, \dots, 1\}), \label{eq:aux^j-dell-V^j}\\
D_{\bm{y}^j} \bm{\ydef}^j(\zeta) &= \operatorname{diag}(1,\dots,1) \in \R^{C \times C} \qquad (j \in \{1,\dots,S\}), \label{eq:equil^j-dell-y^j} \\
\partial_{T^j} \enthdef^j(\zeta) &= -n^j \partial_T \fhl(T^j,\bm{x}^j) \qquad (j \in \{1,\dots,S\}), \label{eq:enth^j-dell-T^j} \\
\partial_{x_C^j} \xsum^j(\zeta) &= 1 \qquad (j \in \{1,\dots,S\}), \label{eq:xsum^j-dell-x_C^j} \\
\partial_{L^j} \hold^{j+1}(\zeta) &= - \fhold'(L^j) \qquad (j \in \{1,\dots,S-1\}), \label{eq:holdup^j+1-dell-L^j}
\end{align}
which concludes the proof.
\end{proof}

\begin{figure}
	\centering
	\includegraphics[width=0.5\textwidth]{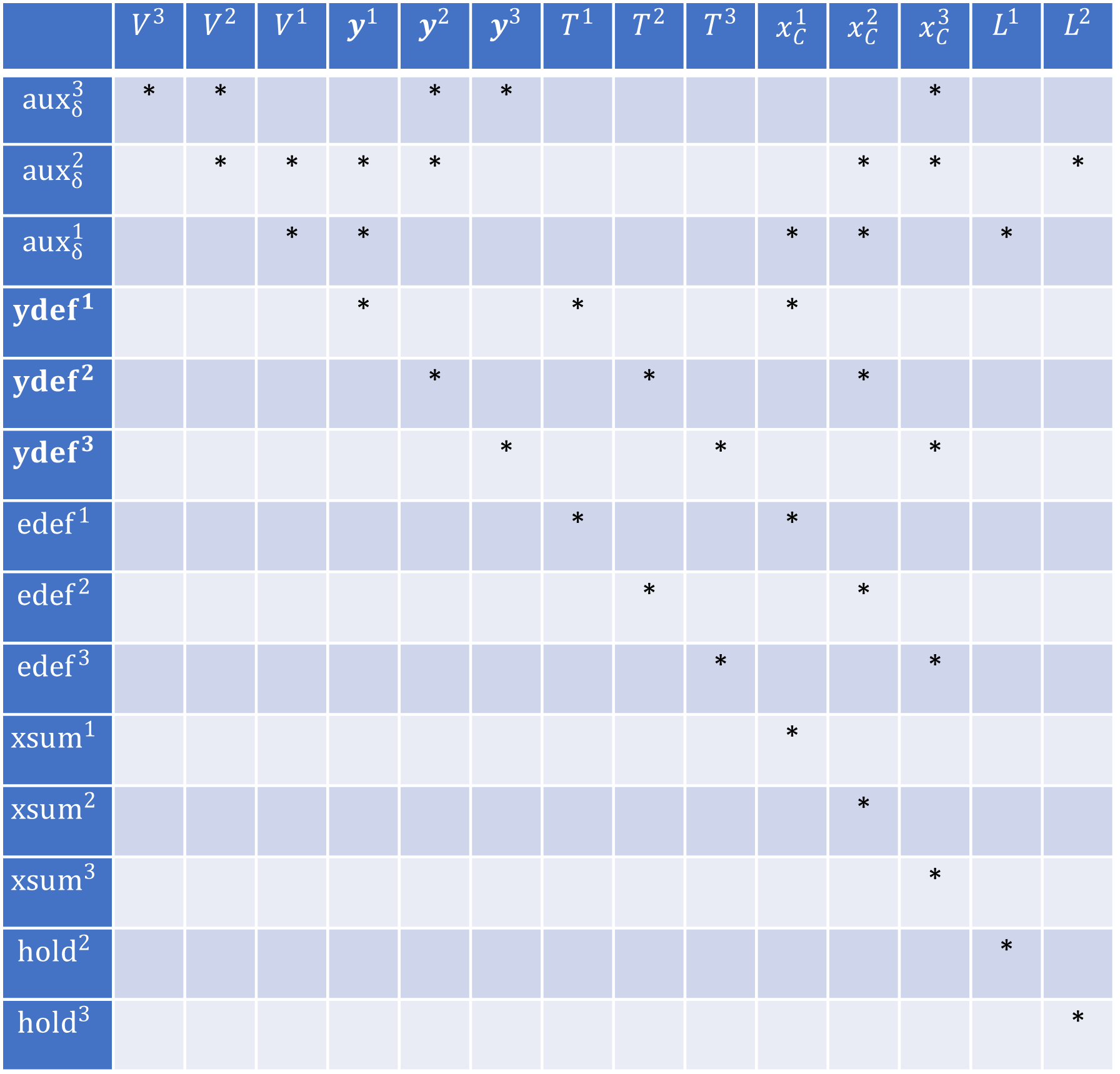}
	\caption{Structure of the Jacobian $D_\eta g_\delta(\zeta)$ of the algebraic equations of~\eqref{eq:ode,perturbed}-\eqref{eq:alg,perturbed} w.r.t.~the algebraic variables in the special case of $S=3$ stages. An empty cell stands for a zero entry or block of that matrix, while a cell with an asterisk stands for a generally non-zero entry or block}
	\label{fig:incidence-matrix-index-reduced-for-S=3}
\end{figure}

We now show that the unperturbed system, that is, system~\eqref{eq:tot-mole,x-sum-differentiation-perturbed}-\eqref{eq:aux-alg-eq,x-sum-differentiation-perturbed} with all $\delta_i^j = 0$, is singular along its solutions (Proposition~\ref{prop:singularity-of-unperturbed-system}). In contrast, the perturbed system~\eqref{eq:tot-mole,x-sum-differentiation-perturbed}-\eqref{eq:aux-alg-eq,x-sum-differentiation-perturbed} with perturbations chosen to satisfy $0 < \delta_i^1 \le \dotsb \le \delta_i^S$ is non-singular along its solutions (Proposition~\ref{prop:regularity-of-perturbed-system}).

\begin{prop} \label{prop:singularity-of-unperturbed-system}
Suppose $I \ni t \mapsto \zeta(t)$ is a solution of the unperturbed system equations~\eqref{eq:tot-mole,x-sum-differentiation-perturbed}-\eqref{eq:aux-alg-eq,x-sum-differentiation-perturbed} with $\delta_i^j = 0$ satisfying
\begin{align} \label{eq:n^j-and-V^j-nonzero}
\epsilon(t) \ne 0, \qquad
n^1(t), \dots, n^S(t) \ne 0, \qquad
V^1(t), \dots, V^S(t) \ne 0, \qquad
\end{align}
for all times $t$ in the solution interval $I$. Then the Jacobian $D_\eta g_0(\zeta(t))$ is singular for every $t \in I$.
\end{prop}

\begin{proof}
It follows by Propositions~\ref{prop:equivalence-of-basic-and-modified-basic-system} and~\ref{prop:equivalence-of-modified-basic-and-unperturbed-system} that $t \mapsto \zeta(t)$ is also a solution to the variant of the system equations~\eqref{eq:comp-mole,tot-and-modif-comp-mole-balance}-\eqref{eq:holdup,tot-and-modif-comp-mole-balance} featuring the vapor summation equations~\eqref{eq:y-sum,tot-and-modif-comp-mole-balance}. In particular, the liquid and the vapor summation equations are satisfied along the solution $t \mapsto \zeta(t)$. Combining this with~\eqref{eq:aux^S-dell-V^S}-\eqref{eq:aux^j-dell-V^j}, we see that the first $S$ diagonal elements of the upper triangular matrix $D_\eta g_0(\zeta(t))$ are all $0$ for every $t \in I$. So, by the upper triangular structure (Lemma~\ref{lm:Jacobian-upper-triangular}), the Jacobian $D_\eta g_0(\zeta(t))$ is singular for every $t \in I$, as claimed.
\end{proof}

\begin{prop} \label{prop:regularity-of-perturbed-system}
Suppose $I \ni t \mapsto \zeta(t)$ is a solution of the perturbed system equations~\eqref{eq:tot-mole,x-sum-differentiation-perturbed}-\eqref{eq:aux-alg-eq,x-sum-differentiation-perturbed} with perturbation parameters $\delta_i^j$ chosen 
\revtext{
such that 
\begin{align} \label{eq:perturbations-increasingly-ordered}
	0 < \delta^1 \le  \dotsb \le \delta^S, 
\end{align}
where $\delta^j := \sum_{i=1}^C \delta_i^j$.} 
Suppose further that this solution satisfies 
\begin{gather} 
\epsilon(t) \ne 0, \qquad
n^1(t), \dots, n^S(t) \ne 0, \qquad
\revtext{
V^1(t), \dots, V^S(t) > 0
}, \qquad
L^1(t), \dots, L^{S-1}(t) > 0 \label{eq:basic-non-singularity-condition-1,perturbed} \\
\revtext{
\fhold'(L^j(t)) \ne 0, \qquad \partial_T \fhl(T^j(t),\bm{x}^j(t)) \ne 0 
}
\label{eq:basic-non-singularity-condition-2,perturbed}
\end{gather}
for all times $t$ in the solution interval $I$. 
Then the Jacobian $D_\eta g_\delta(\zeta(t))$ is non-singular for every $t \in I$. Additionally, the vapor summation equations along the solution become~\eqref{eq:ysum-perturbed}.
\end{prop}

\begin{proof}
It follows from the auxiliary equations~\eqref{eq:aux-alg-eq,x-sum-differentiation-perturbed} and the liquid summation equations~\eqref{eq:x-sum,x-sum-differentiation-perturbed} by induction over $j \in \{1,\dots,S\}$ that
\begin{align}
V^1 \sum_{i=1}^C \big( y_i^1 - x_i^1 - \delta_i^1 \big) 
&= -L^1 \delta^1, \label{eq:y^1-sum-perturbed}\\
V^j \sum_{i=1}^C \big( y_i^j - x_i^j - \delta_i^j \big) 
&= -L^1 \delta^1 - \dotsb - L^j \delta^j + V^1(\delta^1-\delta^2) + \dotsb + V^{j-1} (\delta^{j-1} - \delta^j), \\
\epsilon V^S \sum_{i=1}^C (y_i^S - x_i^S - \delta_i^S) 
&= -L^1 \delta^1 - \dotsb - L^{S-1} \delta^{S-1} + V^1(\delta^1-\delta^2) + \dotsb + V^{S-1} (\delta^{S-1} - \delta^{S}). \label{eq:y^S-sum-perturbed}
\end{align} 
Inserting~\eqref{eq:basic-non-singularity-condition-1,perturbed}, \eqref{eq:basic-non-singularity-condition-2,perturbed} and \eqref{eq:y^1-sum-perturbed}-\eqref{eq:y^S-sum-perturbed} into \eqref{eq:aux^S-dell-V^S}-\eqref{eq:holdup^j+1-dell-L^j} and using~\eqref{eq:fhv-and-fhl}, \eqref{eq:hli}, \eqref{eq:fhold} and~\eqref{eq:perturbations-increasingly-ordered}, it follows that the diagonal elements of $D_\eta g_\delta(\zeta(t))$ are all non-zero for every $t \in I$. So, by the upper triangular structure (Lemma~\ref{lm:Jacobian-upper-triangular}), the Jacobian $D_\eta g_\delta(\zeta(t))$ is non-singular for every $t \in I$, as desired. Additionally, from \eqref{eq:y^1-sum-perturbed}-\eqref{eq:y^S-sum-perturbed} and the liquid summation equations~\eqref{eq:x-sum,x-sum-differentiation-perturbed}, it immediately follows that
\begin{subequations} \label{eq:ysum-perturbed}
\begin{align} 
\sum_{i=1}^C y_i^1 &= 1 + \delta^1 - \frac{L^1}{V^1} \delta^1 \label{eq:ysum^1-perturbed} \\
\sum_{i=1}^C y_i^j &= 1 + \delta^j - \frac{L^1 \delta^1 + \dotsb + L^j \delta^j - V^1(\delta^1-\delta^2) - \dotsb - V^{j-1} (\delta^{j-1} - \delta^j)}{V^j} \\
\sum_{i=1}^C y_i^S &= 1 + \delta^S - \frac{L^1 \delta^1 + \dotsb + L^{S-1} \delta^{S-1} - V^1(\delta^1-\delta^2) - \dotsb - V^{j-1} (\delta^{j-1} - \delta^j)}{\epsilon V^S}, \label{eq:ysum^S-perturbed}
\end{align}
\end{subequations}
which concludes the proof.
\end{proof}

It should be noticed that the assumptions~\eqref{eq:basic-non-singularity-condition-1,perturbed} and \eqref{eq:basic-non-singularity-condition-2,perturbed} are very natural, \revtext{at least if one chooses the thermodynamic submodels as we do for our simulation examples (\ref{app:submodels})}. Indeed, this is evident for all but the last condition~(\ref{eq:basic-non-singularity-condition-2,perturbed}b). In order to see that also the last condition is natural, notice that by~\eqref{eq:fhv-and-fhl}-\eqref{eq:cli} and the liquid summation equation~\eqref{eq:x-sum,x-sum-differentiation-perturbed} this last condition~(\ref{eq:basic-non-singularity-condition-2,perturbed}b) is already satisfied if only 
\begin{align} \label{eq:sufficient-condition-for-index-reduction}
\cli(T^j(t)) \ge \underline{c} > 0 
\qquad \text{and} \qquad
x_i^j(t) \ge 0 \qquad (i \in \{1,\dots,C\}, j \in \{1,\dots,S\}, t \in I). 
\end{align}
Since the $\cli$ are pure-component heat capacity models, the positivity assumption~(\ref{eq:sufficient-condition-for-index-reduction}a) is very natural. In particular, the positivity assumption (\ref{eq:sufficient-condition-for-index-reduction}a) is satisfied in all the case study examples considered in this paper, \revtext{see \ref{app:satisfaction-of-index-reduction-conditions}}. 

\subsubsection{Index reduction by differentiation of the liquid summation equations: relaxation-based approach}
\label{app:relaxation-approach}

\color{black}
We collect the differential and the algebraic variables of the relaxed system~\eqref{eq:ode,relaxed}-\eqref{eq:alg,relaxed} in the (column) vectors $\xi$ and $(\eta,s)$:
\begin{align*}
\xi &:= (n^1, \dots, n^S, \hat{\bm{x}}^1, \dots, \hat{\bm{x}}^S, H^1, \dots, H^S)^\top \in \R^{(2+C)S} \nonumber\\
(\eta,s) &:= (V^S, \dots, V^1, s^1, \dots, s^S, \bm{y}^1, \dots, \bm{y}^S, T^1, \dots, T^S, x_C^1, \dots, x_C^S, L^1, \dots, L^{S-1})^\top \in \R^{(5+C)S-1},
\end{align*}
where $\hat{\bm{x}}^j := (x_1^j, \dots, x_{C-1}^j)$. Additionally, we denote the algebraic equations of the system~\eqref{eq:ode,relaxed}-\eqref{eq:alg,relaxed} by $\hold^j$, $\xsum^j$, $\ydef_i^j$, $\enthdef^j$, $\aux_0^j$, and $\slack^j$, respectively. Specifically, 
\begin{align*}
\hold^2, \dots, \hold^S, \qquad
\xsum^j, \qquad
\ydef_i^j, \qquad
\enthdef^j, \qquad
\aux_0^j \qquad (j \in \{1,\dots,S\})
\end{align*}
are defined as in \ref{app:perturbation-approach} and $\slack^j$ denotes the function of $\zeta := (\xi,\eta,s)$ defined as left-hand side minus right-hand side of the relaxed vapor summation equations~\eqref{eq:ysum,relaxed}. We collect these functions in the following order:
\begin{align*}
g &:= (\aux_0^S, \dots, \aux_0^1, \slack^1, \dots, \slack^S, \bm{\ydef}^1, \dots, \bm{\ydef}^S, \enthdef^1, \dots, \enthdef^S, \xsum^1, \dots, \xsum^S, \\
&\qquad \hold^2, \dots, \hold^S)^\top,
\end{align*}
where $\bm{\ydef}^j := (\ydef_1^j, \dots, \ydef_C^j)$, of course. With these orderings of the algebraic equations and the algebraic variables, the Jacobian $D_{(\eta,s)} g(\zeta)$ of the algebraic equations w.r.t.~the algebraic variables becomes upper triangular (Figure~\ref{fig:incidence-matrix-index-reduced-for-S=3,relaxed}). 

\begin{lm} \label{lm:Jacobian-upper-triangular,relaxed}
The Jacobian $D_{(\eta,s)} g(\zeta)$ is upper triangular for every state $\zeta \in Z$. Also, the diagonal elements of $D_{(\eta,s)} g(\zeta)$ are given by~\eqref{eq:aux^S-dell-V^S,relaxed}-\eqref{eq:holdup^j+1-dell-L^j,relaxed}. 
\end{lm}

\begin{proof}
It is straightforward to verify from the definition of $g$ and the underlying algebraic equations of the relaxed system~\eqref{eq:ode,relaxed}-\eqref{eq:alg,relaxed} that the Jacobian $D_{(\eta,s)} g(\zeta)$ is upper triangular. It is also straightforward to verify that its diagonal elements or diagonal blocks (in the case of the vapor-fraction-defining equations) are given by
\begin{align}
\partial_{V^S} \aux_0^S(\zeta) &= -\epsilon \sum_{i=1}^C \big( y_i^S - x_i^S)/n^S, \label{eq:aux^S-dell-V^S,relaxed}\\
\partial_{V^j} \aux_0^j(\zeta) &= -\sum_{i=1}^C \big( y_i^j - x_i^j)/n^j \qquad (j \in \{S-1, \dots, 1\}), \label{eq:aux^j-dell-V^j,relaxed}\\
\partial_{s^j} \slack^j(\zeta) &= -1 \qquad (j \in \{1,\dots,S\}), \label{eq:slack^j-dell-s^j,relaxed} \\
D_{\bm{y}^j} \bm{\ydef}^j(\zeta) &= \operatorname{diag}(1,\dots,1) \in \R^{C \times C} \qquad (j \in \{1,\dots,S\}), \label{eq:equil^j-dell-y^j,relaxed} \\
\partial_{T^j} \enthdef^j(\zeta) &= -n^j \partial_T \fhl(T^j,\bm{x}^j) \qquad (j \in \{1,\dots,S\}), \label{eq:enth^j-dell-T^j,relaxed} \\
\partial_{x_C^j} \xsum^j(\zeta) &= 1 \qquad (j \in \{1,\dots,S\}), \label{eq:xsum^j-dell-x_C^j,relaxed} \\
\partial_{L^j} \hold^{j+1}(\zeta) &= - \fhold'(L^j) \qquad (j \in \{1,\dots,S-1\}), \label{eq:holdup^j+1-dell-L^j,relaxed}
\end{align}
which concludes the proof.
\end{proof}

\begin{figure}
	\centering
	\includegraphics[width=0.5\textwidth]{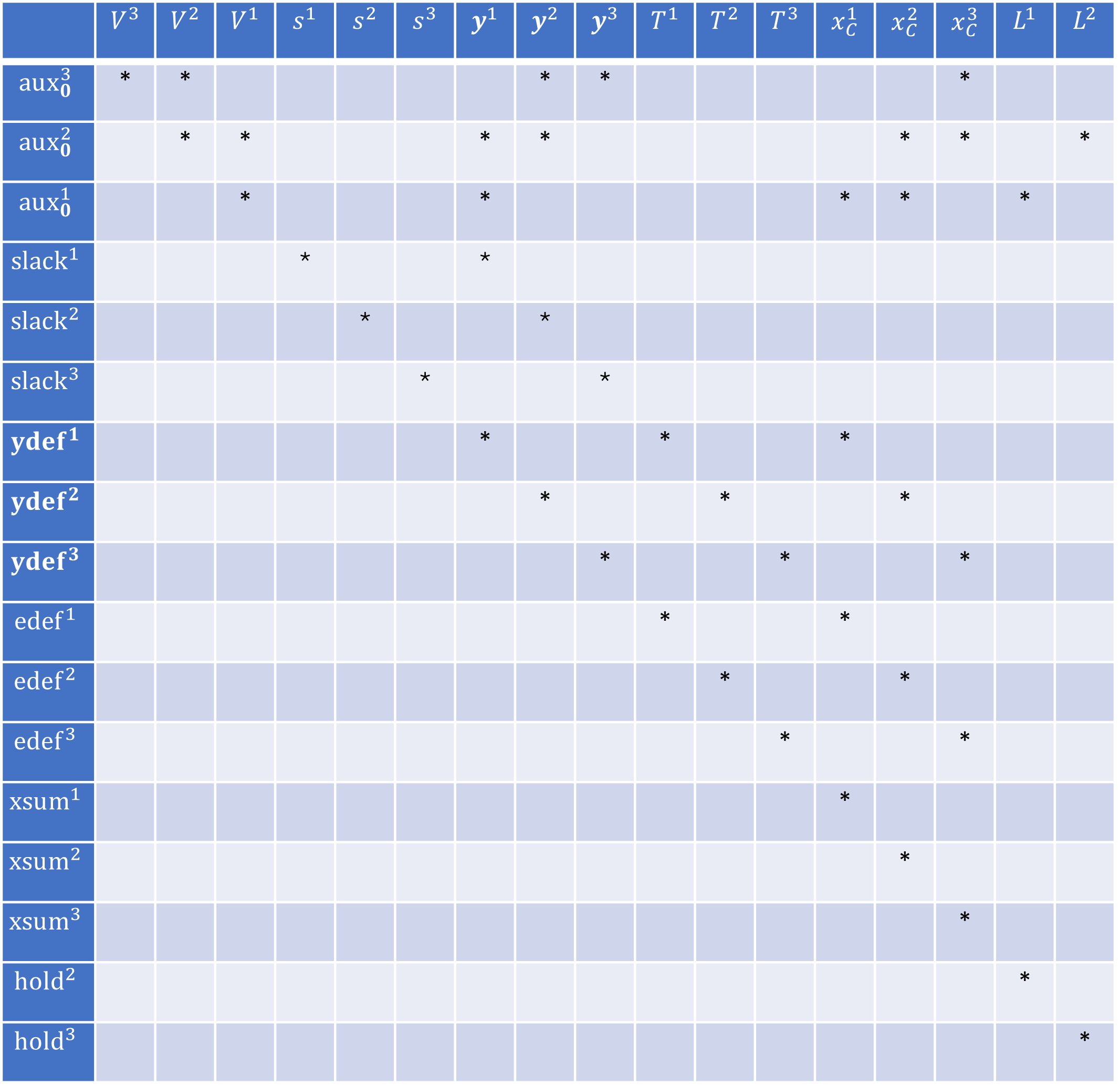}
	\caption{Structure of the Jacobian $D_{(\eta,s)} g(\zeta)$ of the algebraic equations of~\eqref{eq:ode,relaxed}-\eqref{eq:alg,relaxed} w.r.t.~the algebraic variables in the special case of $S=3$ stages. An empty cell stands for a zero entry or block of that matrix, while a cell with an asterisk stands for a generally non-zero entry or block}
	\label{fig:incidence-matrix-index-reduced-for-S=3,relaxed}
\end{figure}

\begin{prop} \label{prop:regularity-of-relaxed-system}
Suppose $I \ni t \mapsto \zeta(t)$ is a solution of the relaxed system equations~\eqref{eq:ode,relaxed}-\eqref{eq:alg,relaxed} satisfying 
\begin{gather} 
\epsilon(t) \ne 0, \qquad
n^1(t), \dots, n^S(t) \ne 0, \qquad \fhold'(L^j(t)) \ne 0, \qquad \partial_T \fhl(T^j(t),\bm{x}^j(t)) \ne 0 \label{eq:basic-non-singularity-condition,relaxed}
\end{gather}
for all times $t$ in the solution interval $I$. 
Then the Jacobian $D_{(\eta,s)} g(\zeta(t))$ is non-singular for every $t \in I$ if and only if
\begin{align}
s^1(t), \dots, s^S(t) \ne 0 \qquad (t \in I). 
\end{align}
\end{prop}

\begin{proof}
In view of the upper triangular structure of the Jacobian $D_{(\eta,s)} g(\zeta)$ (Lemma~\ref{lm:Jacobian-upper-triangular,relaxed}) and the formulas~\eqref{eq:aux^S-dell-V^S,relaxed}-\eqref{eq:holdup^j+1-dell-L^j,relaxed} for its diagonal elements and the relaxed vapor summation equations~\eqref{eq:ysum,relaxed}, the asserted equivalence is clear. 
\end{proof}
\color{black}

\subsubsection{Index reduction by differentiation of the vapor summation equations}
\label{app:y-summation-approach}

\color{black}
We collect the differential and the algebraic variables of the equations~\eqref{eq:tot-mole,y-sum-differentiation}-\eqref{eq:vlesum,y-differentiation} in the (column) vectors $\xi$ and $\eta$:
\begin{align*}
\xi &:= (n^1, \dots, n^S, \hat{\bm{x}}^1, \dots, \hat{\bm{x}}^S)^\top \in \R^{(1+C)S} \nonumber\\
\eta &:= (V^S, \dots, V^1, \bm{y}^1, \dots, \bm{y}^S, H^1, \dots, H^S, T^1, \dots, T^S, x_C^1, \dots, x_C^S, L^1, \dots, L^{S-1})^\top \in \R^{(5+C)S-1},
\end{align*}
where $\hat{\bm{x}}^j := (x_1^j, \dots, x_{C-1}^j)$. Additionally, we denote the algebraic equations of the system~\eqref{eq:tot-mole,y-sum-differentiation}-\eqref{eq:vlesum,y-differentiation} by $\hold^j$, $\aebal^j$, $\xsum^j$, $\ydef_i^j$, $\enthdef^j$, $\ysum^j$, respectively. Specifically, 
\begin{align*}
\hold^2, \dots, \hold^S, \qquad
\aebal^j, \qquad
\xsum^j, \qquad
\ydef_i^j, \qquad
\enthdef^j, \qquad
\ysum^j \qquad (j \in \{1,\dots,S\})
\end{align*}
denote the functions of the state variables $\zeta := (\xi,\eta)$ defined as left-hand side minus right-hand side of the holdup equations~\eqref{eq:holdup-j,y-sum-differentiation}, of the algebraicized enthalpy balance equations~\eqref{eq:enth,y-sum-differentiation}, of the liquid summation equations~\eqref{eq:x-sum,y-sum-differentiation}, of the vapor-fraction-defining equations~\eqref{eq:y-def,y-sum-differentiation}, of the enthalpy-defining equations~\eqref{eq:enth-def,y-sum-differentiation}, and of the equilibrium summation equations~\eqref{eq:vlesum,y-differentiation}, respectively. We collect these functions in the following order:
\begin{align*}
g &:= (\aebal^S, \dots, \aebal^1, \bm{\ydef}^1, \dots, \bm{\ydef}^S, \enthdef^1, \dots, \enthdef^S, \ysum^1, \dots, \ysum^S, \xsum^1, \dots, \xsum^S, \\
&\qquad \hold^2, \dots, \hold^S)^\top,
\end{align*}
where $\bm{\ydef}^j := (\ydef_1^j, \dots, \ydef_C^j)$, of course. With these orderings of the algebraic equations and the algebraic variables, the Jacobian $D_\eta g(\zeta)$ becomes upper triangular (Figure~\ref{fig:incidence-matrix-index-reduced-for-S=3}). 

\begin{lm} \label{lm:Jacobian-upper-triangular,y-sum-differentiation}
The Jacobian $D_\eta g(\zeta)$ is upper triangular for every state $\zeta \in Z$. Also, the diagonal elements of $D_\eta g(\zeta)$ are given by~\eqref{eq:aebal^S-dell-V^S}-\eqref{eq:holdup^j+1-dell-L^j,y-sum-differentiation}. 
\end{lm}

\begin{proof}
It is straightforward to verify from the definition of $g$ and the underlying algebraic equations of the alternative  system~\eqref{eq:tot-mole,y-sum-differentiation}-\eqref{eq:vlesum,y-differentiation} that the Jacobian $D_\eta g(\zeta)$ is upper triangular. It is also straightforward to verify that its diagonal elements or diagonal blocks (in the case of the vapor-fraction-defining equations) are given by
\begin{align}
\partial_{V^S} \aebal^S(\zeta) &= -\epsilon \fhl(T^S,\bm{x}^S) - \epsilon \bm{a}(P,T^S,\bm{x}^S) \cdot (\bm{y}^S - \bm{x}^S) \nonumber \\
&\quad - (1-\epsilon) \fhl(\Tcond,\bm{y}^S) + \fhv(T^S,\bm{y}^S), \label{eq:aebal^S-dell-V^S}\\
\partial_{V^j} \aebal^j(\zeta) &= -\fhl(T^j,\bm{x}^j) - \bm{a}(P,T^j,\bm{x}^j) \cdot (\bm{y}^j - \bm{x}^j) + \fhv(T^j,\bm{y}^j) \qquad (j \in \{S-1, \dots, 1\}), \label{eq:aebal^j-dell-V^j}\\
D_{\bm{y}^j} \bm{\ydef}^j(\zeta) &= \operatorname{diag}(1,\dots,1) \in \R^{C \times C} \qquad (j \in \{1,\dots,S\}), \label{eq:ydef^j-dell-y^j,y-sum-differentiation} \\
\partial_{H^j} \enthdef^j(\zeta) &= 1 \qquad (j \in \{1,\dots,S\}), \label{eq:edef^j-dell-H^j} \\
\partial_{x_C^j} \xsum^j(\zeta) &= 1 \qquad (j \in \{1,\dots,S\}) \label{eq:xsum^j-dell-x_C^j,y-sum-differentiation}  \\
\partial_{L^j} \hold^{j+1}(\zeta) &= - \fhold'(L^j) \qquad (j \in \{1,\dots,S-1\}), \label{eq:holdup^j+1-dell-L^j,y-sum-differentiation}
\end{align}
which concludes the proof.
\end{proof}

\begin{figure}
	\centering
	\includegraphics[width=0.5\textwidth]{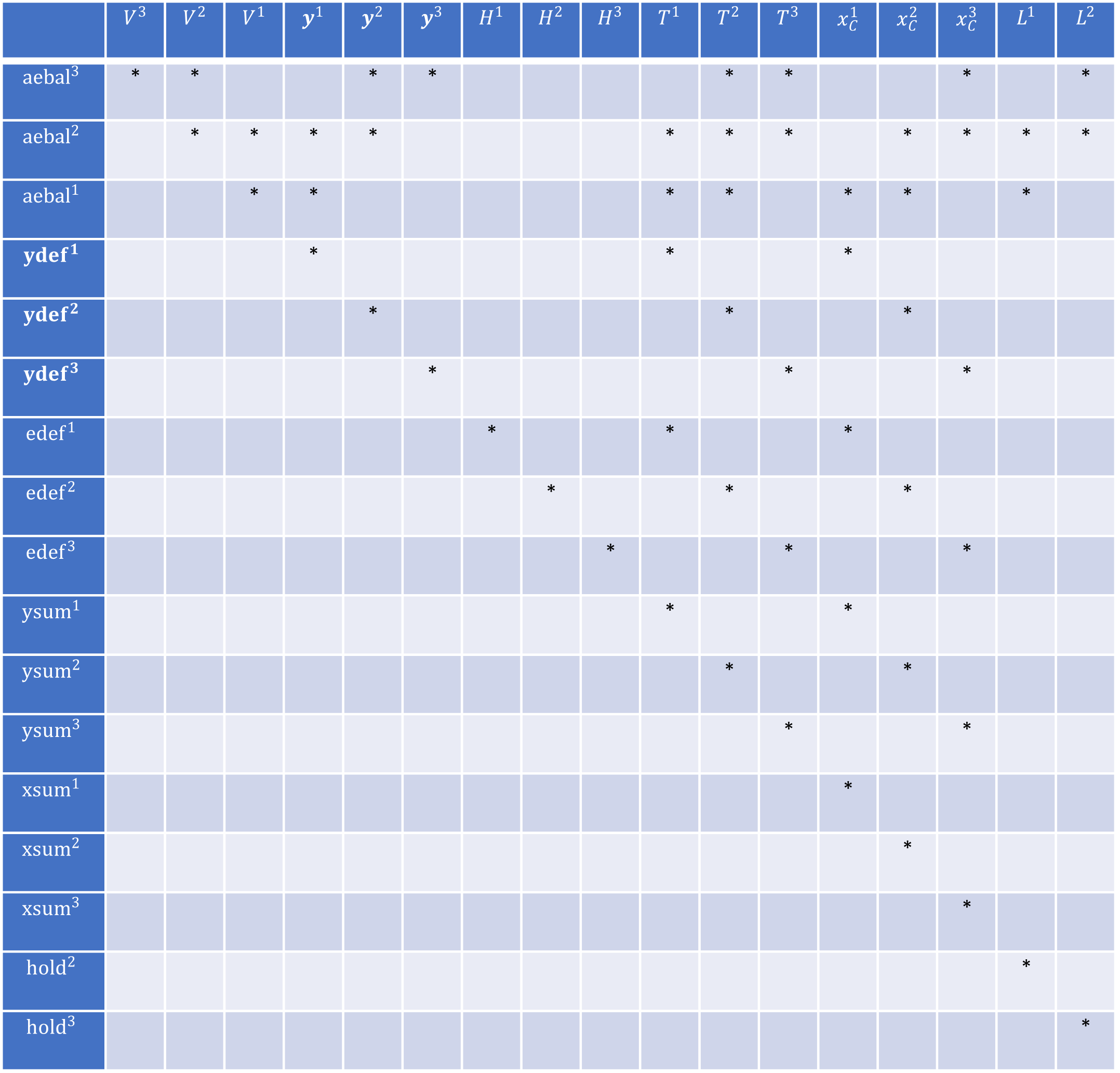}
	\caption{Structure of the Jacobian $D_\eta g(\zeta)$ of the algebraic equations of~\eqref{eq:tot-mole,y-sum-differentiation}-\eqref{eq:vlesum,y-differentiation} w.r.t.~the algebraic variables in the special case of $S=3$ stages. An empty cell stands for a zero entry or block of that matrix, while a cell with an asterisk stands for a generally non-zero entry or block}
	\label{fig:incidence-matrix-index-reduced-y-sum-differentiation-for-S=3}
\end{figure}

We now work out when exactly the alternative system~\eqref{eq:tot-mole,y-sum-differentiation}-\eqref{eq:vlesum,y-differentiation} is non-singular along its solutions.

\begin{prop} \label{prop:regularity-of-alternative-system}
Suppose $I \ni t \mapsto \zeta(t)$ is a solution of the alternative system equations~\eqref{eq:tot-mole,y-sum-differentiation}-\eqref{eq:vlesum,y-differentiation}. Suppose further that this solution satisfies 
\begin{align} 
\fhold'(L^j(t)) \ne 0
\qquad \text{and} \qquad
\sum_{i=1}^C \partial_T \fvlei(P(t),T^j(t),\bm{x}^j(t)) \ne 0 
\end{align}
for all times $t$ in the solution interval $I$ and all $j \in \{1,\dots,S\}$. Then the Jacobian $D_\eta g(\zeta(t))$ is non-singular for every $t \in I$ if and only if
\begin{subequations} \label{eq:regularity-characterization,y-sum-differentiation}
	\begin{align}
		 &-\epsilon(t) \fhl(T^S(t),\bm{x}^S(t)) - \epsilon(t) \bm{a}(P(t),T^S(t),\bm{x}^S(t)) \cdot (\bm{y}^S(t) - \bm{x}^S(t)) \nonumber \\
		 &\qquad - (1-\epsilon(t)) \fhl(\Tcond(t),\bm{y}^S(t)) + \fhv(T^S(t),\bm{y}^S(t)) 
\ne 0, \label{eq:regularity-characterization-1,y-sum-differentiation} \\ 
		 &-\fhl(T^j(t),\bm{x}^j(t)) - \bm{a}(P(t),T^j(t),\bm{x}^j(t)) \cdot (\bm{y}^j(t) - \bm{x}^j(t)) + \fhv(T^j(t),\bm{y}^j(t)) \ne 0 
		 \label{eq:regularity-characterization-j,y-sum-differentiation}     
    \end{align}
\end{subequations}
for every $t \in I$ and every $j \in \{S-1, \dots, 1\}$.
\end{prop}

\begin{proof}
In view of the upper triangular structure of the Jacobian $D_\eta g(\zeta)$ (Lemma~\ref{lm:Jacobian-upper-triangular,y-sum-differentiation}) and the formulas~\eqref{eq:aebal^S-dell-V^S}-\eqref{eq:holdup^j+1-dell-L^j,y-sum-differentiation} for its diagonal elements, the asserted equivalence is clear. 
\end{proof}

It should be noticed that in contrast to the non-singularity conditions~\eqref{eq:basic-non-singularity-condition-1,perturbed}-\eqref{eq:basic-non-singularity-condition-2,perturbed} for the perturbed system, the non-singularity conditions~\eqref{eq:regularity-characterization,y-sum-differentiation} for the alternative system~\eqref{eq:tot-mole,y-sum-differentiation}-\eqref{eq:vlesum,y-differentiation} are hard to verify rigorously for $S \ge 2$ stages. In fact, the conditions~\eqref{eq:regularity-characterization,y-sum-differentiation} seem to be easily verifiable only in the case $S = 1$ (which is not very interesting and which we therefore excluded from the very beginning to avoid case distinctions). Indeed, in the case of just $S = 1$ stage, 
the generally unwieldy conditions~\eqref{eq:regularity-characterization,y-sum-differentiation} greatly simplify to the  condition that $Q(t) \ne 0$ for every $t \in I$. Yet, for the case of $S \ge 2$ stages, a more convenient re-formulation of~\eqref{eq:regularity-characterization,y-sum-differentiation} seems difficult. 
In particular, it is not clear (to us) if, when, or why the conventional index reduction approach really does lead to an index-reduced system.
\color{black}

\subsection{Infinite-reflux initialization}
\label{app:infinite-reflux-initialization}

\color{black}
In this section, we collect some basic facts about the inner loop (Algorithm~\ref{alg:inner-loop}) of our stationary initialization. We begin with a derivation of Algorithm~\ref{alg:inner-loop}, namely we show that every stationary solution of the intermediary unperturbed system~\eqref{eq:ode,intermediary}-\eqref{eq:alg,intermediary} with constant control inputs~\eqref{eq:u-const-u_0} (infinite reflux and constant other control inputs) can be obtained as a result of Algorithm~\ref{alg:inner-loop}. 

\begin{lm} \label{lm:stationary-solution-is-output-of-inner-algo}
Suppose $t \mapsto (\xi(t),\eta(t)) \equiv (\xi,\eta)$ is a stationary solution of the unperturbed system~\eqref{eq:ode,intermediary}-\eqref{eq:alg,intermediary} with $\epsilon \equiv 0$ and $\hat{u} \equiv (P_0,Q_0,\Tcond_0)$. Suppose further that 
\begin{align} \label{eq:inner-algo-assumption}
n^j, V^j \ne 0 \qquad \text{and} \qquad 
\fhl(T^{j+1},\bm{y}^j) - \fhv(T^j,\bm{y}^j) \ne 0
\qquad (j \in \{1,\dots,S\}),
\end{align}
where $T^{S+1} := \Tcond_0$. Then $(\xi,\eta)$ satisfies the definition formulas from Algorithm~\ref{alg:inner-loop}, that is,
\begin{gather}
\sum_{i=1}^C \fvlei(P_0,T^j,\bm{x}^j) = 1, \qquad
y_i^j = \fvlei(P_0,T^j,\bm{x}^j), 
\label{eq:inner-algo-ysum-and-ydef}\\
V^j = \begin{cases}
-\frac{Q_0}{\fhl(T^2,\bm{y}^1) - \fhv(T^1,\bm{y}^1)}, \qquad j = 1 \\
V^{j-1} \frac{\fhl(T^j,\bm{y}^{j-1}) - \fhv(T^{j-1},\bm{y}^{j-1})}{\fhl(T^{j+1},\bm{y}^j) - \fhv(T^j,\bm{y}^j)}, \qquad j \ne 1
\end{cases}
\label{eq:inner-algo-V}
\end{gather}
for $j \in \{1,\dots,S\}$, 
\begin{align} \label{eq:inner-algo-x-and-L}
\bm{x}^{j+1} = \bm{y}^j, \qquad 
L^j = V^j 
\end{align}
for $j \in \{1,\dots,S-1\}$, and 
\begin{align} \label{eq:inner-algo-n-and-H}
n^j = \fhold(L^{j-1}), \qquad
H^j = n^j \fhl(T^j,\bm{x^j})
\end{align}
for $j \in \{2,\dots,S\}$. 
\end{lm}

\begin{proof}
Since $n^j, V^j \ne 0$ by assumption, it follows by Proposition~\ref{prop:equivalence-of-basic-and-modified-basic-system} and~\ref{prop:equivalence-of-modified-basic-and-unperturbed-system} that $t \mapsto (\xi(t),\eta(t))$ is also a solution to the system~\eqref{eq:tot-mole,tot-and-modif-comp-mole-balance}-\eqref{eq:holdup,tot-and-modif-comp-mole-balance} and  to the original system (as defined in Proposition~\ref{prop:equivalence-of-basic-and-modified-basic-system}). As a consequence of the total mole balance~\eqref{eq:tot-mole,tot-and-modif-comp-mole-balance} and the assumed constancy of the solution, it then follows that~(\ref{eq:inner-algo-x-and-L}b) is satisfied. As a consequence of the original component mole balance~\eqref{eq:comp-mole,original}, in turn, it follows that~(\ref{eq:inner-algo-x-and-L}a) is satisfied as well. Also, the relation~\eqref{eq:inner-algo-V} follows by the enthalpy balance~\eqref{eq:enth,tot-and-modif-comp-mole-balance}. All other identities of \eqref{eq:inner-algo-ysum-and-ydef}-\eqref{eq:inner-algo-n-and-H} are obvious. 
\end{proof}

We now show that, conversely, every result of Algorithm~\ref{alg:inner-loop} gives rise to a stationary solution of the intermediary unperturbed system~\eqref{eq:ode,intermediary}-\eqref{eq:alg,intermediary} with constant control inputs~\eqref{eq:u-const-u_0}.

\begin{prop} \label{prop:output-of-inner-algo-is-stationary-solution}
Suppose $\bm{x}^1 \in \mathcal{X} := \{ \bm{x} \in [0,1]^C: \sum_{i=1}^C x_i = 1\}$ and $(P_0,Q_0,\Tcond_0) \in (0,\infty)^3$ are given and that $\fhold(L) \ne 0$ for all $L \ne 0$. Suppose further that $T^1,\bm{y}^1, \bm{x}^2, T^2, \bm{y}^2, \dots, \bm{x}^S, T^S, \bm{y}^S$ and $V^1, \dots, V^S, L^1, \dots L^{S-1}$ and $n^2, \dots, n^S, H^2, \dots, H^S$ are computed according to Algorithm~\ref{alg:inner-loop}, that is, \eqref{eq:inner-algo-ysum-and-ydef}-\eqref{eq:inner-algo-n-and-H} are satisfied along with the well-definedness condition~(\ref{eq:inner-algo-assumption}b). And finally, choose $n^1 \in (0,\infty)$ arbitrarily and let $H^1 := n^1 \fhl(T^1,\bm{x}^1)$. Then the constant map $t \mapsto (\xi(t),\eta(t))$ defined by 
\begin{align}
\xi(t) &:= (n^1, \dots, n^S, \hat{\bm{x}}^1, \dots, \hat{\bm{x}}^S, H^1, \dots, H^S)^\top \\
\eta(t) &:= (V^S, \dots, V^1, \bm{y}^1, \dots, \bm{y}^S, T^1, \dots, T^S, x_C^1, \dots, x_C^S, L^1, \dots, L^{S-1})^\top
\end{align}
is a solution of the unperturbed system~\eqref{eq:ode,intermediary}-\eqref{eq:alg,intermediary} with $\epsilon \equiv 0$ and $\hat{u} \equiv (P_0,Q_0,\Tcond_0)$.
\end{prop}

\begin{proof}
It immediately follows from~\eqref{eq:inner-algo-ysum-and-ydef}-\eqref{eq:inner-algo-n-and-H} that $t \mapsto (\xi(t),\eta(t))$ is a solution to the original system equations (as defined in Proposition~\ref{prop:equivalence-of-basic-and-modified-basic-system}) with $\epsilon \equiv 0$ and $\hat{u} \equiv (P_0,Q_0,\Tcond_0)$. Since $Q_0, n^1 \ne 0$ and $\fhold(L) \ne 0$ for $L \ne 0$ by assumption, it further follows by~\eqref{eq:inner-algo-V}, (\ref{eq:inner-algo-n-and-H}a), (\ref{eq:inner-algo-x-and-L}b) that $V^j(t) \ne 0$ and $n^j(t) \ne 0$ for all $j \in \{1,\dots,S\}$. So, by Proposition~\ref{prop:equivalence-of-basic-and-modified-basic-system} and~\ref{prop:equivalence-of-modified-basic-and-unperturbed-system}, the constant map $t \mapsto (\xi(t),\eta(t))$ is also a solution to the unperturbed system~\eqref{eq:ode,intermediary}-\eqref{eq:alg,intermediary}, as desired.
\end{proof}

We close with a few remarks on the well-definedness condition~\eqref{eq:inner-algo-well-defined} or~(\ref{eq:inner-algo-assumption}b) for Algorithm~\ref{alg:inner-loop}. At least with our choices of the thermodynamic submodels (\ref{app:submodels}), this well-definedness condition is very natural. Indeed, with our molar enthalpy models~\eqref{eq:fhv-and-fhl}-\eqref{eq:hvi}, we have
\begin{align}
\fhl(T^{j+1},\bm{y}^j) - \fhv(T^j,\bm{y}^j) 
= -\bm{y}^j \cdot \bigg( \int_{T^{j+1}}^{T^j} \cli(T) \mathrm{d} T + \hvapi(T^j) \bigg)_{i \in \{1,\dots,C\}}.
\end{align}
And therefore, the well-definedness condition~(\ref{eq:inner-algo-assumption}b) is satisfied if, for instance, the following four conditions are all satisfied:
\begin{itemize}
\item[(i)] $y_i^j \ge 0$ for every $i \in \{1,\dots,C\}$, that is, our vapor-liquid equilibrium models $\fvlei$ predict non-negative vapor mole fractions at $(P_0,T^j,\bm{x}^j)$
\item[(ii)] $T^j > T^{j+1}$, that is the boiling temperature on stage $j$ is larger than on stage $j+1$ above
\item[(iii)] $\cli(T) > 0$ for every $T \in [T^{j+1},T^j]$ and $i \in \{1,\dots,C\}$, that is, the liquid heat capacity of each component is larger than $0$ on the temperature range $[T^{j+1},T^j]$ 
\item[(iv)] $\hvapi(T^j) > 0$, that is, our vaporization enthalpy models $\hvapi$ predict positive vaporization enthalpies for all components at $T^j$. 
\end{itemize}
Clearly, these conditions are very natural. Indeed, if any of these conditions was violated, this would indicate that our submodels are unphysical and thus inappropriate.
\color{black}

\section{Supplementary information on the case study}

\subsection{Satisfaction of the index-reduction conditions}
\label{app:satisfaction-of-index-reduction-conditions}

\revtext{In this section we show that the liquid molar heat capacity $\cli(T)$ of the components used in Section~\ref{sec:study} is always greater than zero in the relevant temperature range. The condition $\cli(T) > 0$ is necessary in order for our infinite reflux algorithm to work (refer to \ref{app:infinite-reflux-initialization}.)}
\revtext{This is shown in Figures~\ref{fig:cp_mixture1_2} and \ref{fig:cp_mixture_3}.}

\begin{figure}[H]
	\centering
	\makebox[\textwidth]{\includegraphics[width=0.9\textwidth]{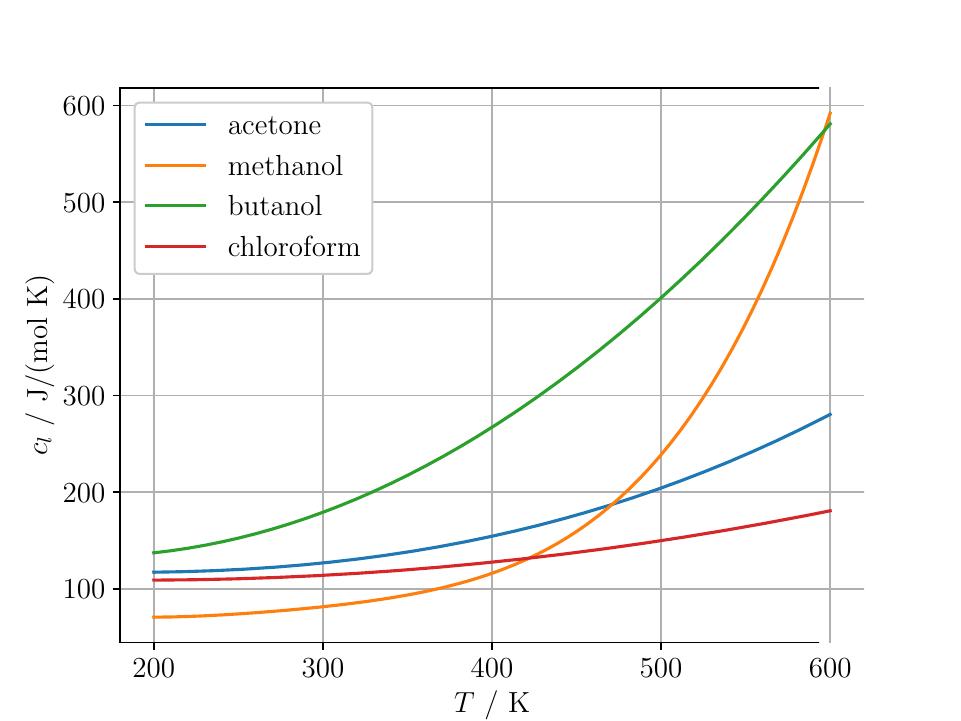}}
	\caption{\revtext{Liquid molar heat capacity $\cli$ of components acetone (1), methanol (2), butanol (3), and chloroform (4)}}
	\label{fig:cp_mixture1_2}
\end{figure}

\begin{figure}[H]
	\centering
	\makebox[\textwidth]{\includegraphics[width=0.9\textwidth]{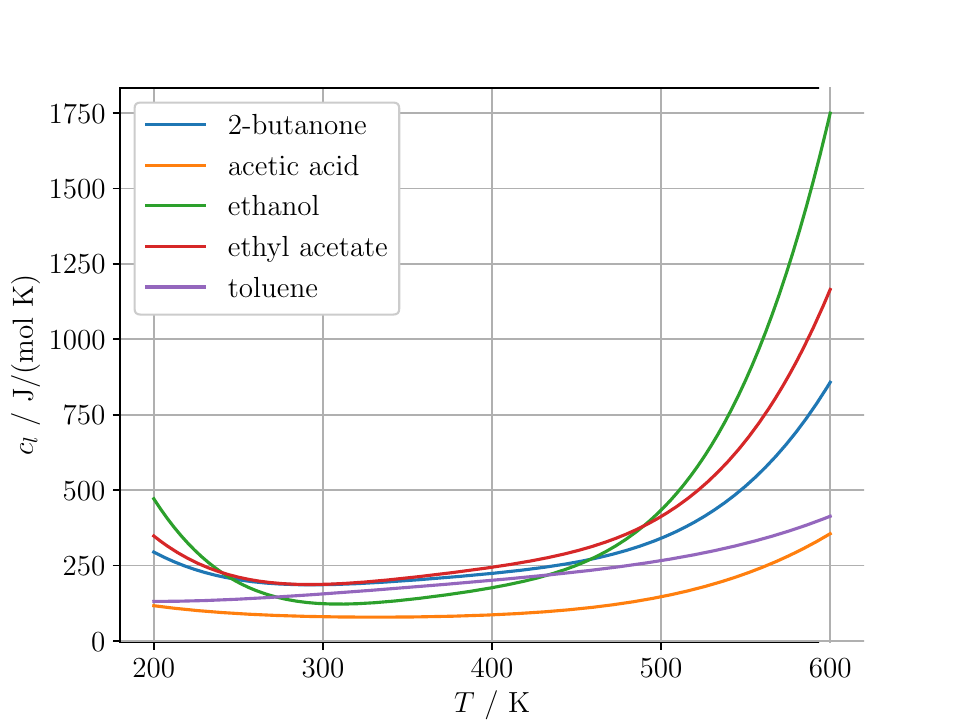}}
	\caption{\revtext{Liquid molar heat capacity $\cli$ of components 2-butanone (1), acetic acid (2), ethanol (3), ethyl acetate (4), toluene (5)}}
	\label{fig:cp_mixture_3}
\end{figure}

\subsection{Satisfaction of the vapor summation equations}
\label{app:satisfaction-of-vapor-summation-equation}

\begin{figure}[H]
	\centering
	\makebox[\textwidth]{\includegraphics[width=0.9\textwidth]{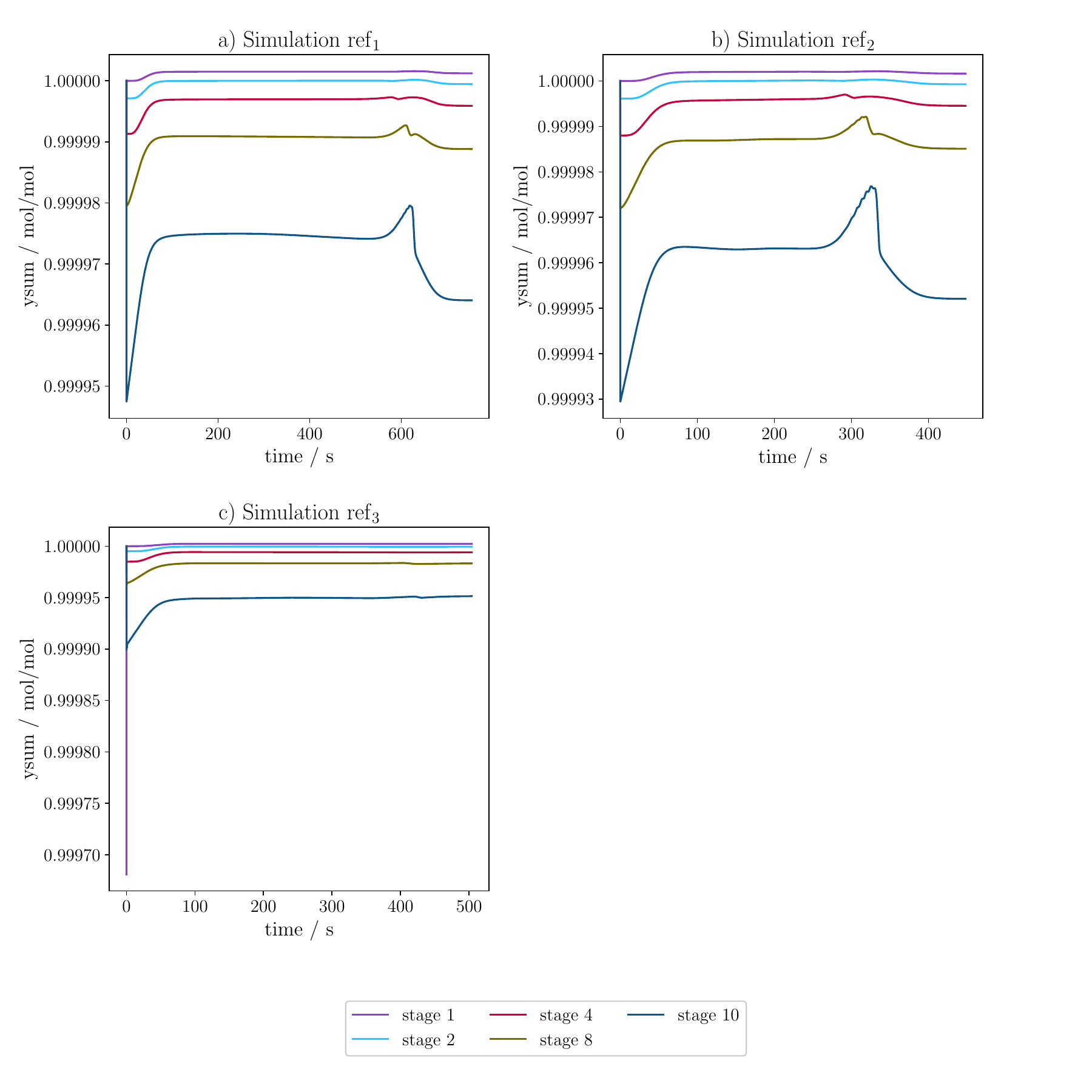}}
	\caption{\revtext{Summation of vapor mole fractions $\mathrm{ysum}$ on stages 1, 2, 4, 8, and 10; the DAE system that was used in the simulation is \eqref{eq:ode,perturbed}-\eqref{eq:alg,perturbed}; $\mathrm{ysum}$ is shown for the simulations $\mathrm{ref}_1$, $\mathrm{ref}_2$, and $\mathrm{ref}_3$ in a), b), and c) respectively}}
	\label{fig:vapor_summation_satisfaction}
\end{figure}

\subsection{Additional results for the binary submixtures} \label{app:additional_results_binary}
In the following we show the simulation results of the binary submixtures

\begin{itemize}
    \item acetone (1), methanol (2) (simulation $\text{ref}_4$ with $\prefidx{4} = \theta_\text{ref}^{\{1,2\}}$, and $\text{alt}_4$ with $\theta_{\text{alt}_4}=\palt^{\{1,2\}}$),
    \item acetone (1), chloroform (4) (simulation $\text{ref}_5$ with $\prefidx{5} = \theta_\text{ref}^{\{1,4\}}$, and $\text{alt}_5$ with $\theta_{\text{alt}_5}=\palt^{\{1,5\}}$),
    \item acetone (1), butanol (3) (simulation $\text{ref}_6$ with $\prefidx{6} = \theta_\text{ref}^{\{1,3\}}$, and $\text{alt}_6$ with $\theta_{\text{alt}_6}=\palt^{\{1,3\}}$),
    \item methanol (2), butanol (3) (simulation $\text{ref}_7$ with $\prefidx{7} = \theta_\text{ref}^{\{2,3\}}$, and $\text{alt}_7$ with $\theta_{\text{alt}_7}=\palt^{\{2,3\}}$),
    \item acetic acid (2), ethanol (3) (simulation $\text{ref}_8$ with $\prefidx{8} = \theta_\text{ref}^{\{2,3\}}$, and $\text{alt}_8$ with $\theta_{\text{alt}_8}=\palt^{\{2,3\}}$). 
\end{itemize}

\begin{figure}[H]
	\centering
	\makebox[\textwidth]{\includegraphics[width=\textwidth]{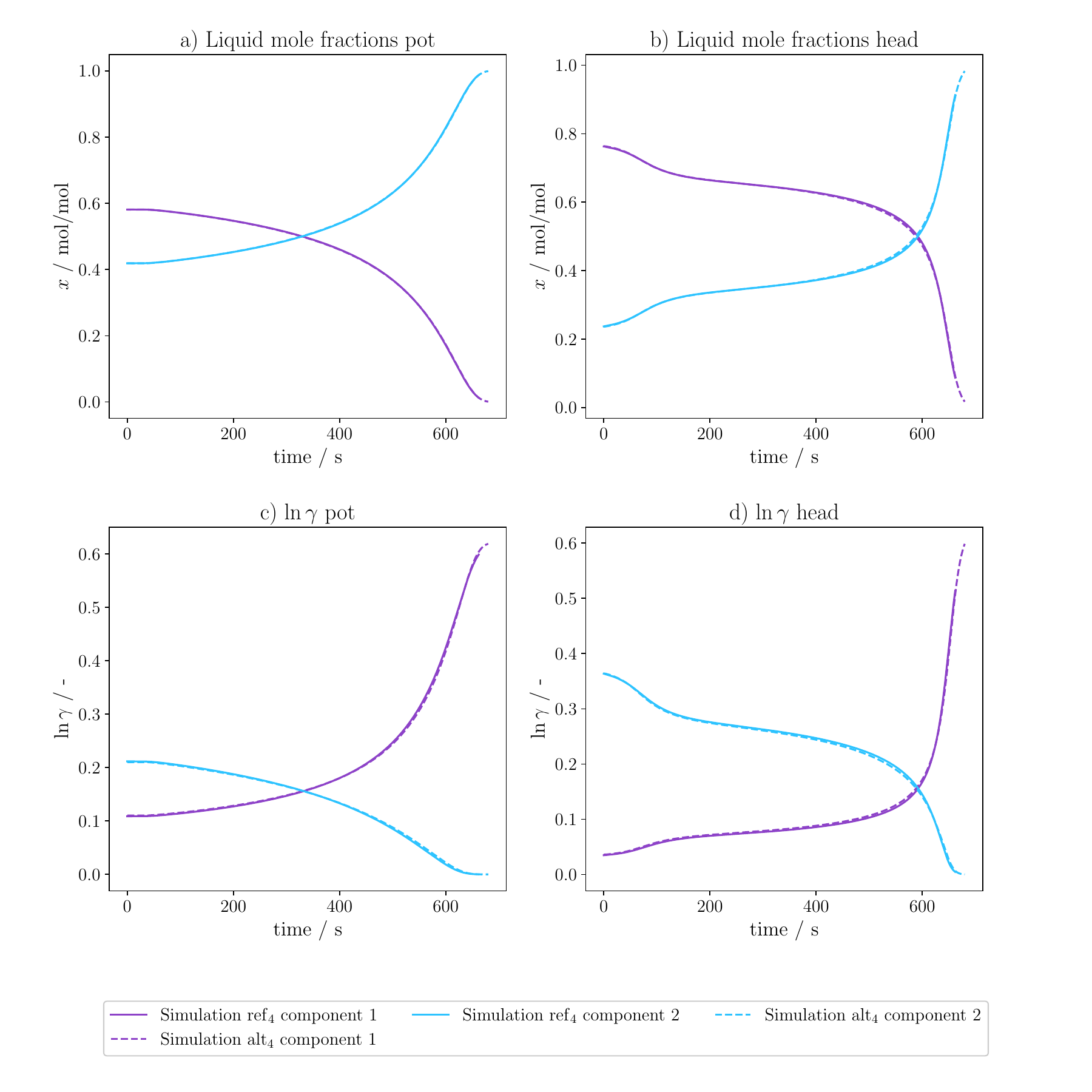}}
	\caption{Simulation results for mixture acetone (1), methanol (2) with initial composition $\bm{x}^{\text{app},0} = (0.6,0.4)$ \SI{}{\mole \per \mole}; liquid mole fractions $\bm{x}$ for pot and head stage are shown in a) and b), and activity coefficients $\bm{\gamma}$ for pot and head stage are shown in c) and d)}
	\label{fig:dynamic_diagram_acetone_methanol}
\end{figure}

\begin{figure}[H]
	\centering
	\makebox[\textwidth]{\includegraphics[width=\textwidth]{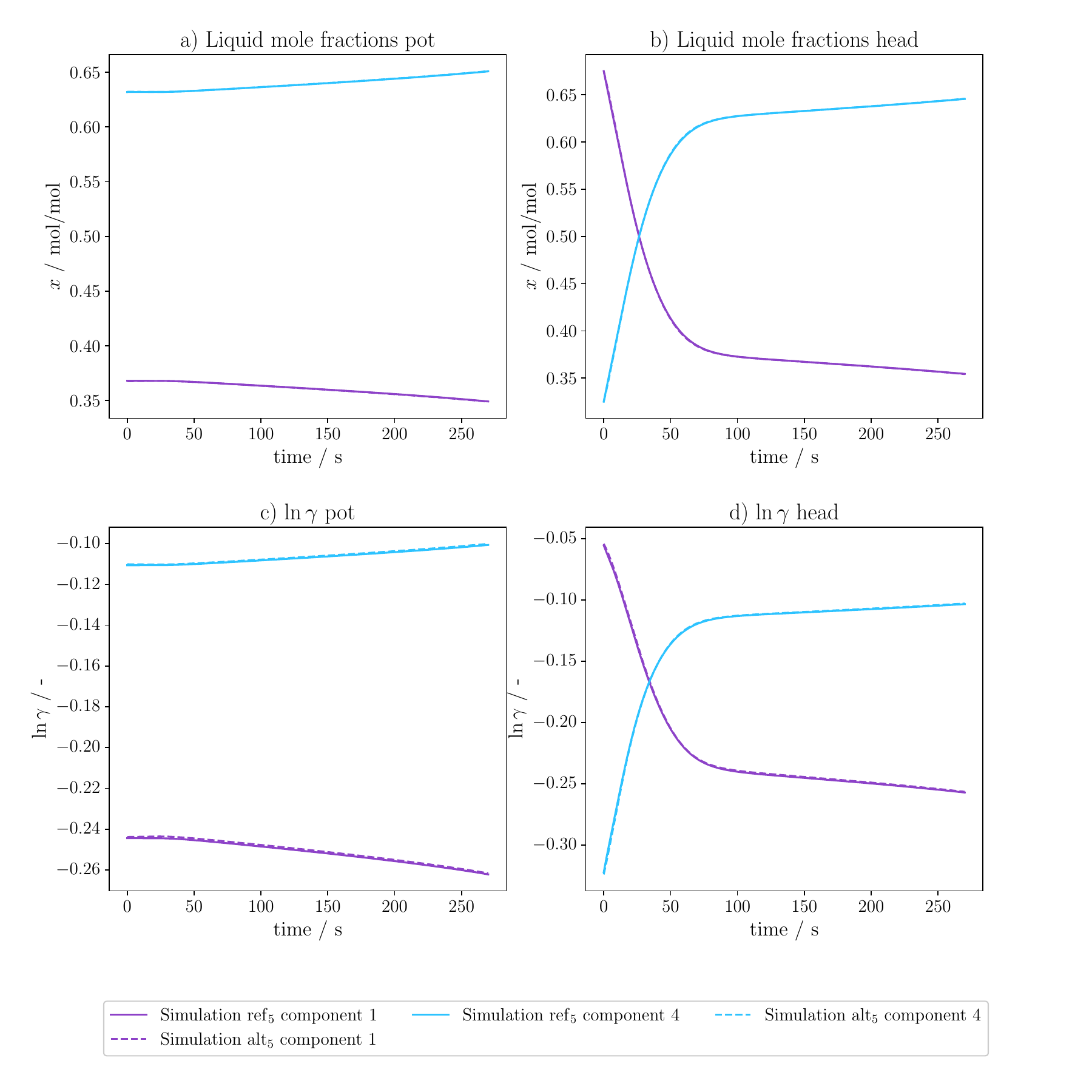}}
	\caption{Simulation results for mixture acetone (1), chloroform (4) with initial composition $\bm{x}^{\text{app},0} = (0.4,0.6)$ \SI{}{\mole \per \mole}; liquid mole fractions $\bm{x}$ for pot and head stage are shown in a) and b), and activity coefficients $\bm{\gamma}$ for pot and head stage are shown in c) and d)}
	\label{fig:dynamic_diagram_acetone_chloroform}
\end{figure}

\begin{figure}[H]
	\centering
	\makebox[\textwidth]{\includegraphics[width=\textwidth]{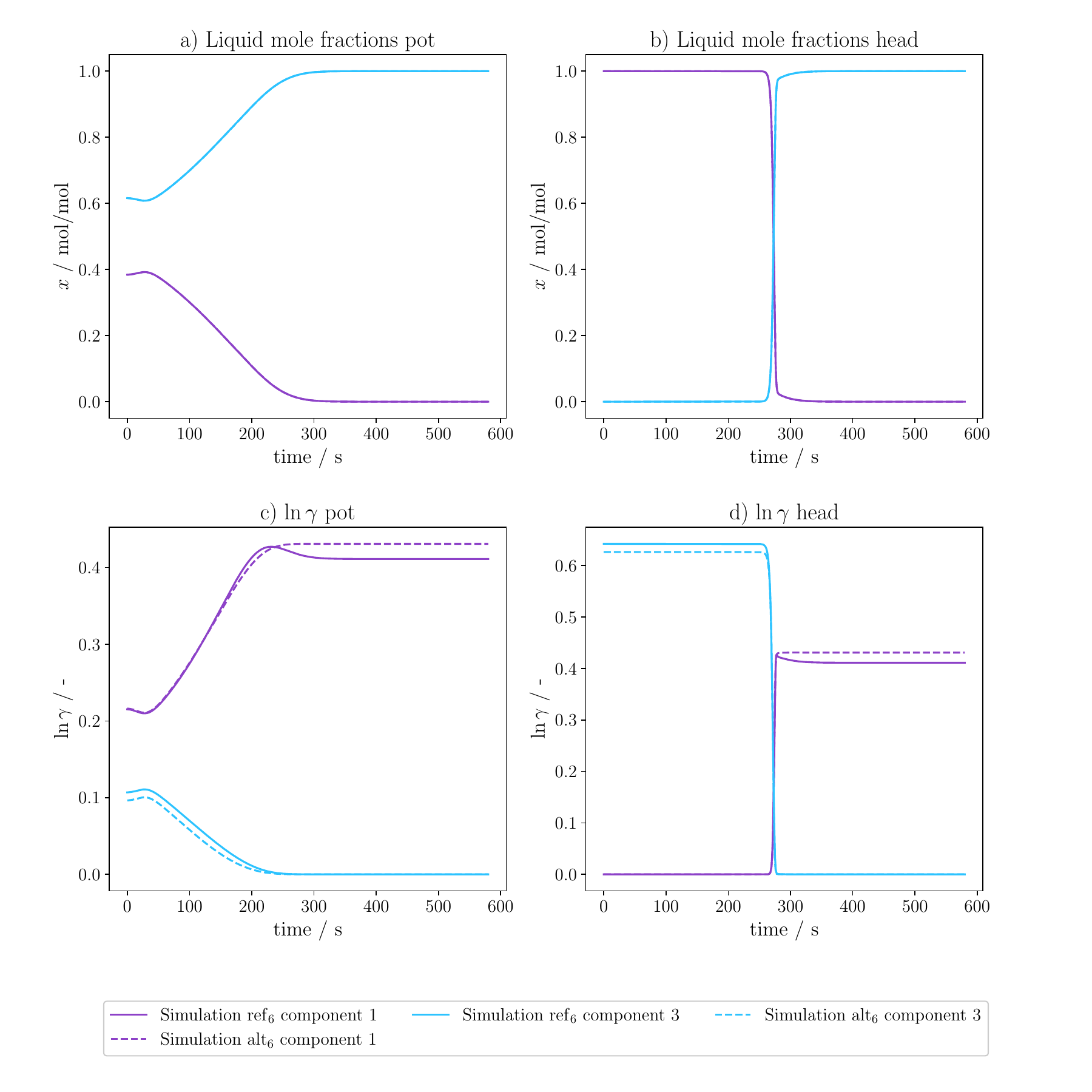}}
	\caption{Simulation results for mixture acetone (1), butanol (3) with initial composition $\bm{x}^{\text{app},0} = (0.5,0.5)$ \SI{}{\mole \per \mole}; liquid mole fractions $\bm{x}$ for pot and head stage are shown in a) and b), and activity coefficients $\bm{\gamma}$ for pot and head stage are shown in c) and d)}
	\label{fig:dynamic_diagram_acetone_butanol}
\end{figure}

\begin{figure}[H]
	\centering
	\makebox[\textwidth]{\includegraphics[width=\textwidth]{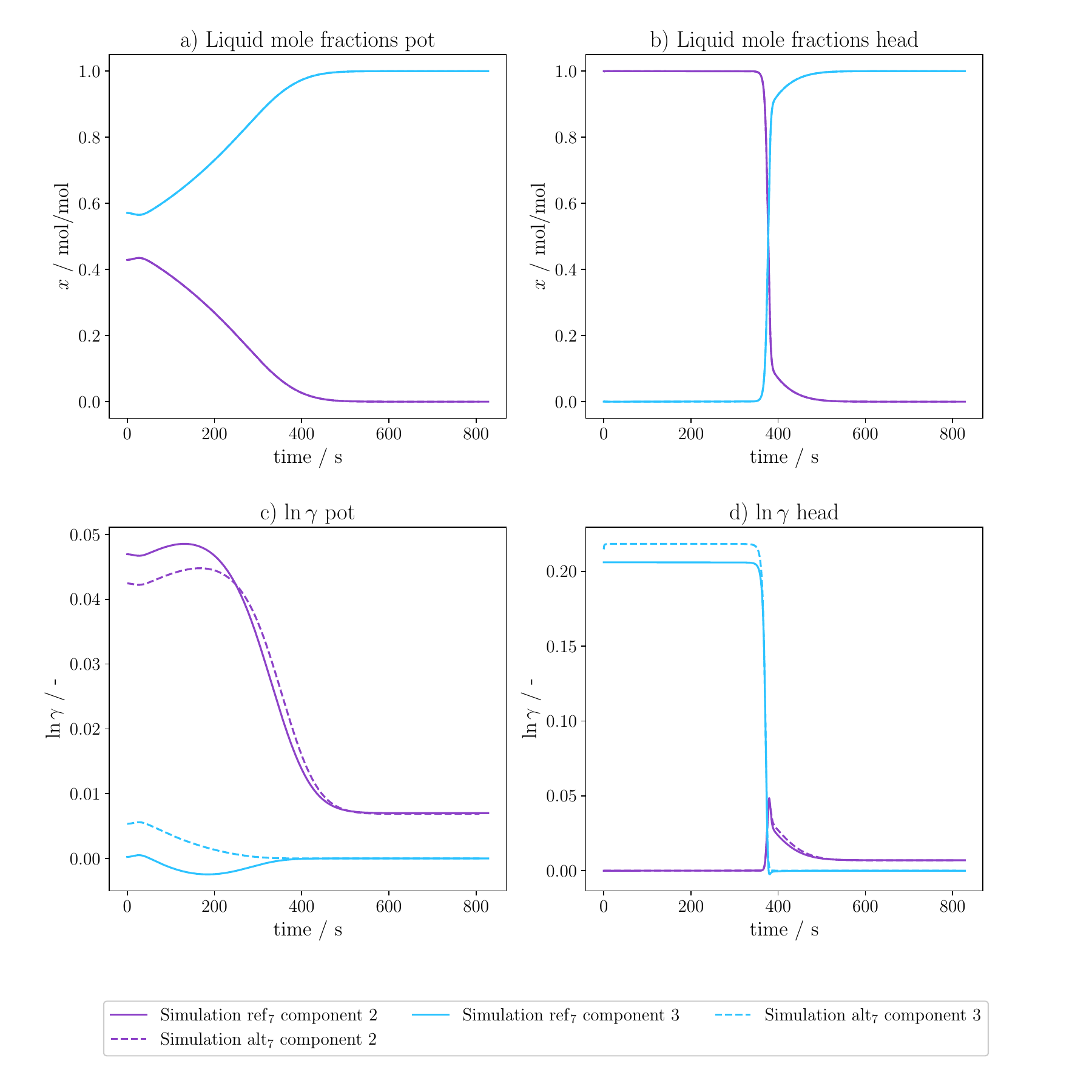}}
	\caption{Simulation results for mixture methanol (2), butanol (3) with initial composition $\bm{x}^{\text{app},0} = (0.5,0.5)$ \SI{}{\mole \per \mole}; liquid mole fractions $\bm{x}$ for pot and head stage are shown in a) and b), and activity coefficients $\bm{\gamma}$ for pot and head stage are shown in c) and d)}
	\label{fig:dynamic_diagram_methanol_butanol}
\end{figure}

\begin{figure}[H]
	\centering
	\makebox[\textwidth]{\includegraphics[width=\textwidth]{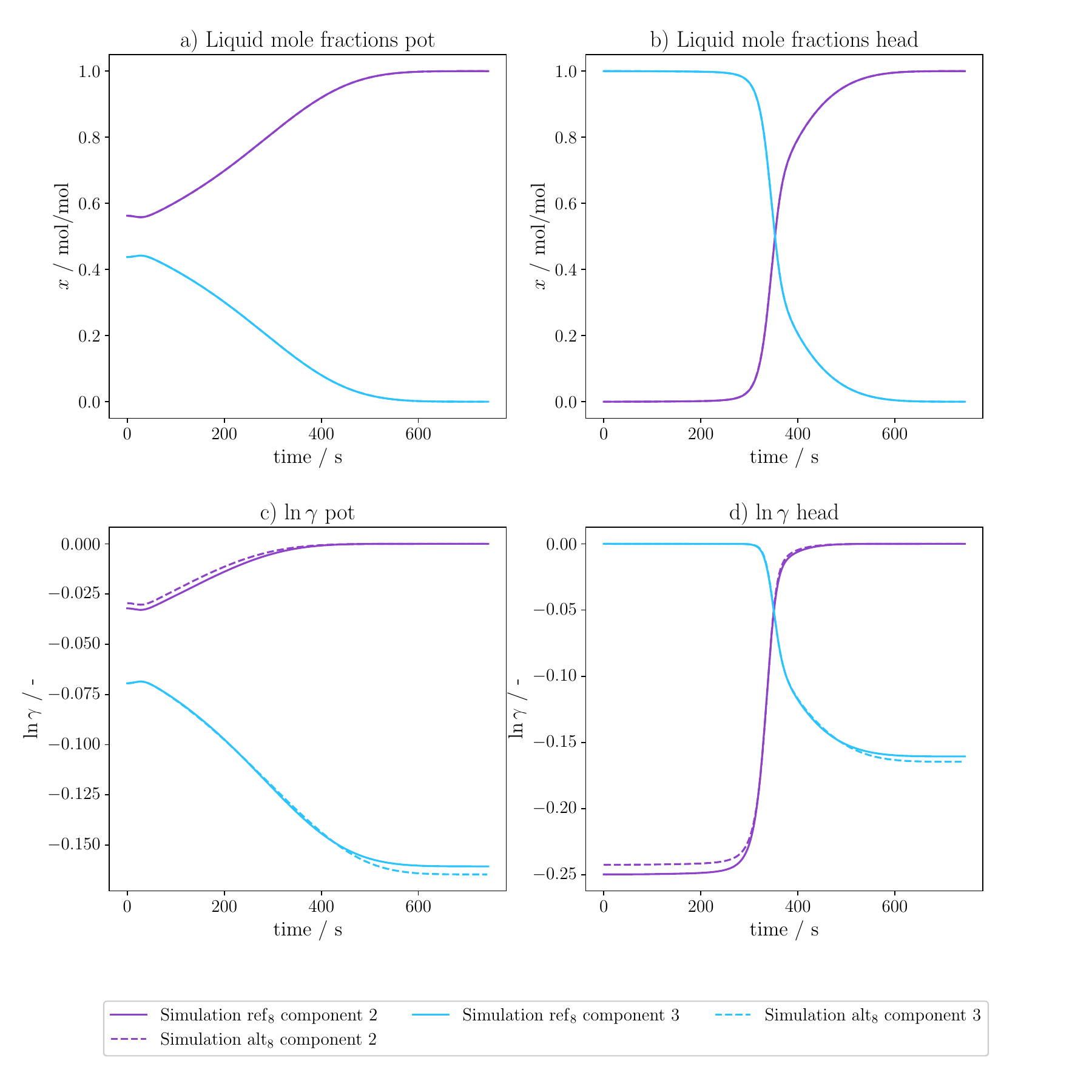}}
	\caption{Simulation results for mixture acetic acid (2), ethanol (3) with initial composition $\bm{x}^{\text{app},0} = (0.5,0.5)$ \SI{}{\mole \per \mole}; liquid mole fractions $\bm{x}$ for pot and head stage are shown in a) and b), and activity coefficients $\bm{\gamma}$ for pot and head stage are shown in c) and d)}
	\label{fig:dynamic_diagram_aceticAcid_ethanol}
\end{figure}

\subsection{Additional results for the three-component mixture} \label{app:additional_results_mixture1}

\begin{figure}[H]
	\centering
	\makebox[\textwidth]{\includegraphics[width=\textwidth]{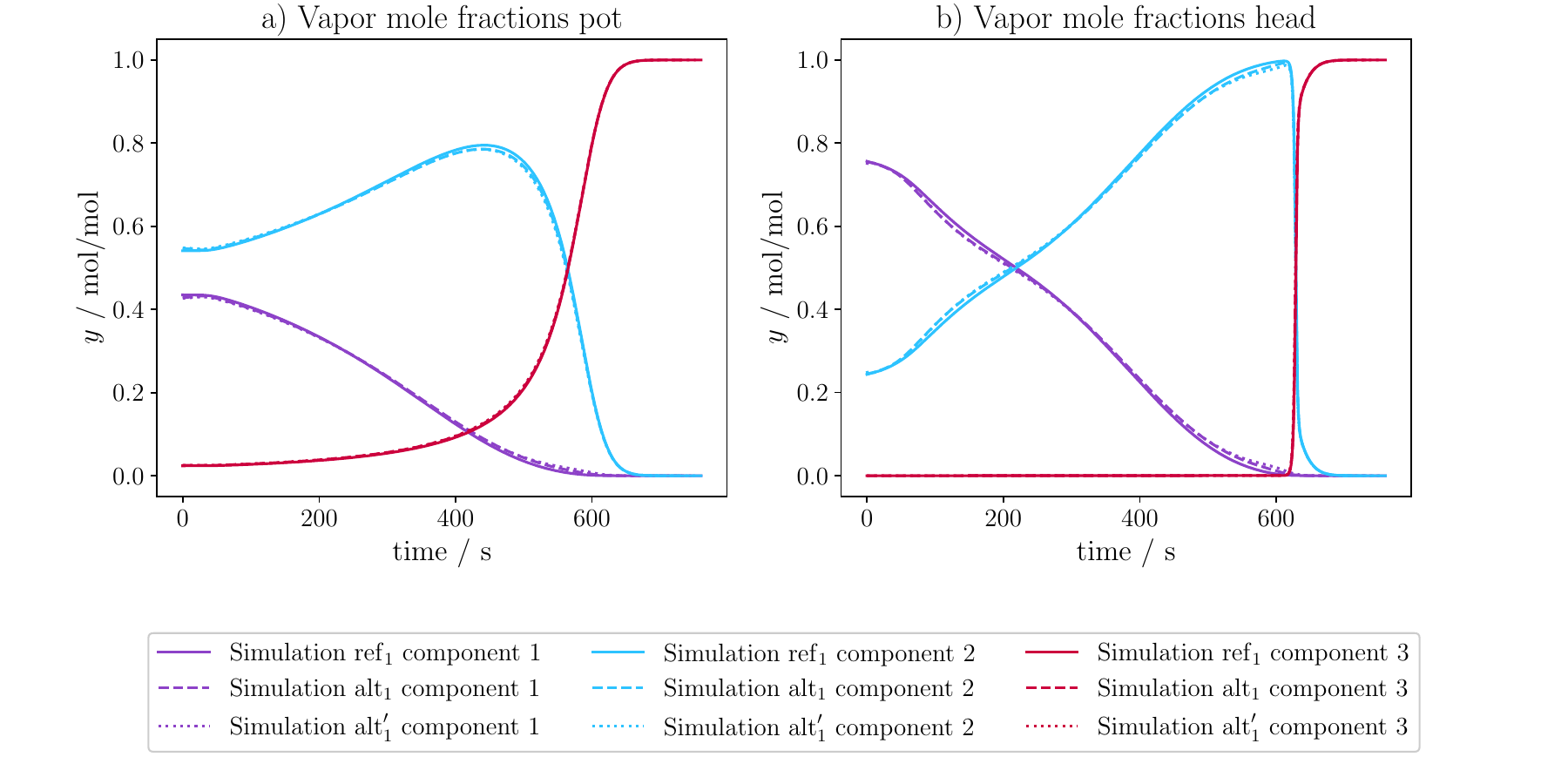}}
	\caption{Simulation results for mixture 1 acetone (1), methanol (2), and butanol (3); vapor mole fractions $\bm{y}$ for pot and head shown in a) and b)}
	\label{fig:dynamic_diagram_example1_vapor}
\end{figure}

\begin{figure}[H]
	\centering
	\makebox[\textwidth]{\includegraphics[width=\textwidth]{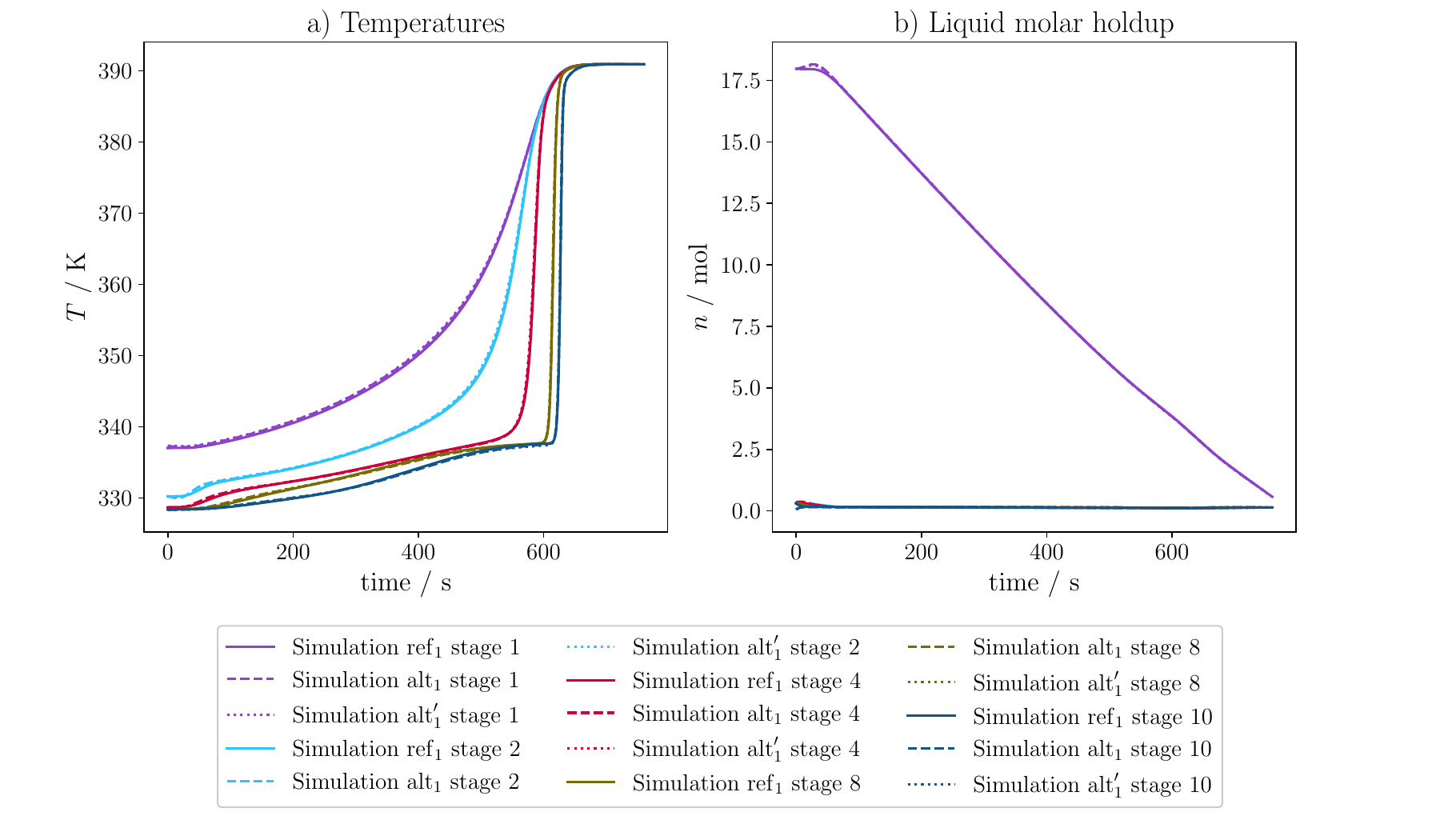}}
	\caption{Simulation results for mixture 1 acetone (1), methanol (2), and butanol (3); a) temperatures $T$ and b) molar liquid holdup $n$ on all stages}
	\label{fig:dynamic_diagram_example1_temp_n}
\end{figure}

\begin{figure}[H]
	\centering
	\makebox[\textwidth]{\includegraphics[width=\textwidth]{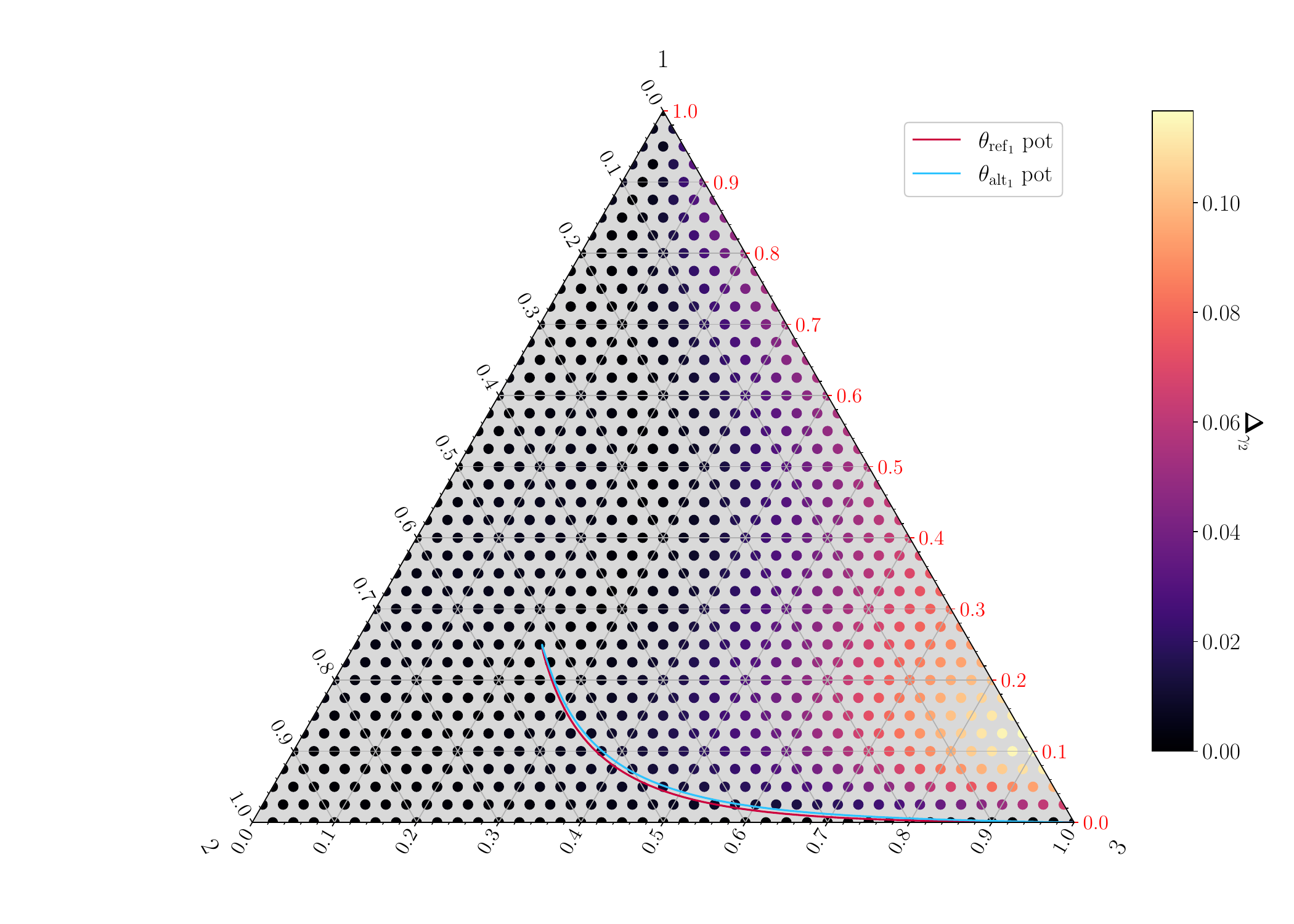}}
	\caption{Ternary diagram of mixture acetone (1), methanol (2), and butanol (3); the metric $\bm{\Delta}_{\gamma_2}$ is plotted as a heatmap within the ternary diagram; the liquid mole fraction trajectories of the pot stage for simulations 1 and $\text{alt}_1$ are visualized by the red and blue line-plots; the red-marked axis indicates that a high $\bm{\Delta}_{\gamma_2}$ can be expected for the binary sub-system (1)-(3)}
	\label{fig:ternary_plots_example1_2}
\end{figure}

\begin{figure}[H]
	\centering
	\makebox[\textwidth]{\includegraphics[width=\textwidth]{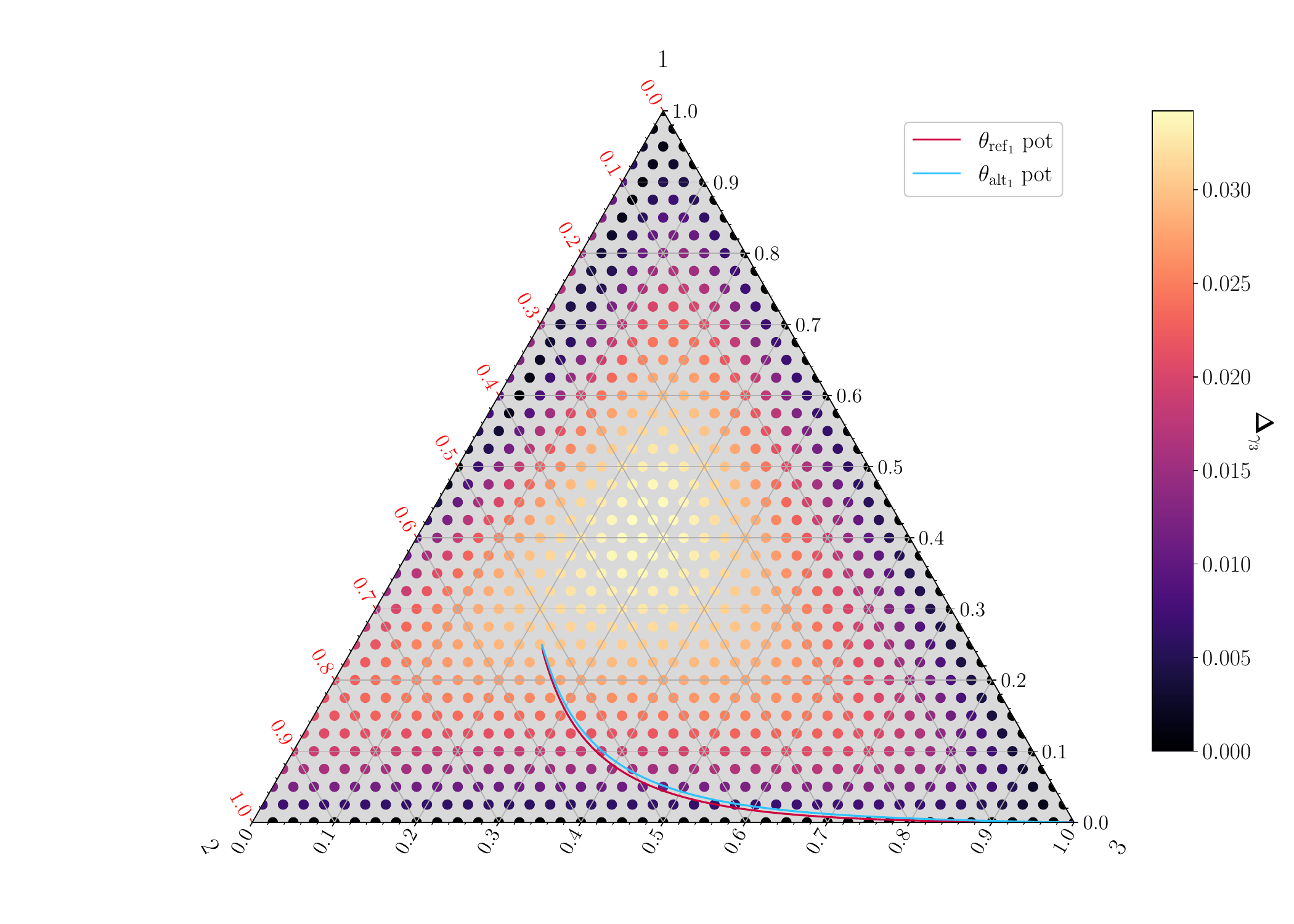}}
	\caption{Ternary diagram of mixture acetone (1), methanol (2), and butanol (3); the metric $\bm{\Delta}_{\gamma_3}$ is plotted as a heatmap within the ternary diagram; the liquid mole fraction trajectories of the pot stage for simulations 1 and $\text{alt}_1$ are visualized by the red and blue line-plots; the red-marked axis indicates that a high $\bm{\Delta}_{\gamma_3}$ can be expected for the binary sub-system (1)-(2)}
	\label{fig:ternary_plots_example1_3}
\end{figure}

\subsection{Additional results for the four-component mixture} \label{app:additional_restults_mixture2}

\begin{figure}[H]
	\centering
	\makebox[\textwidth]{\includegraphics[width=\textwidth]{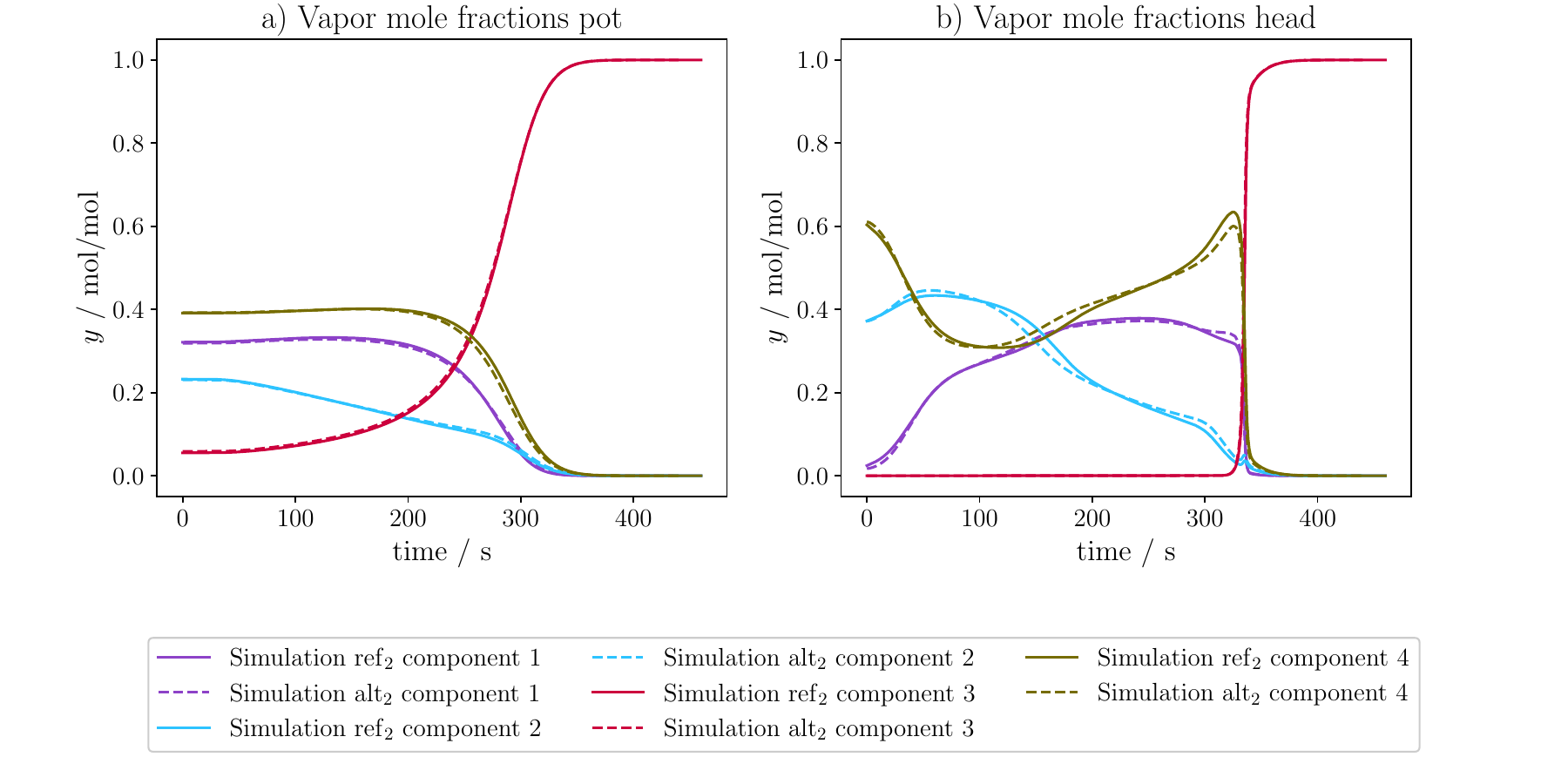}}
	\caption{Simulation results for mixture 2 acetone (1), methanol (2), butanol (3), and chloroform (4); vapor mole fractions $\bm{y}$ for pot and head shown in a) and b)}
	\label{fig:dynamic_diagram_example2_vapor}
\end{figure}

\begin{figure}[H]
	\centering
	\makebox[\textwidth]{\includegraphics[width=\textwidth]{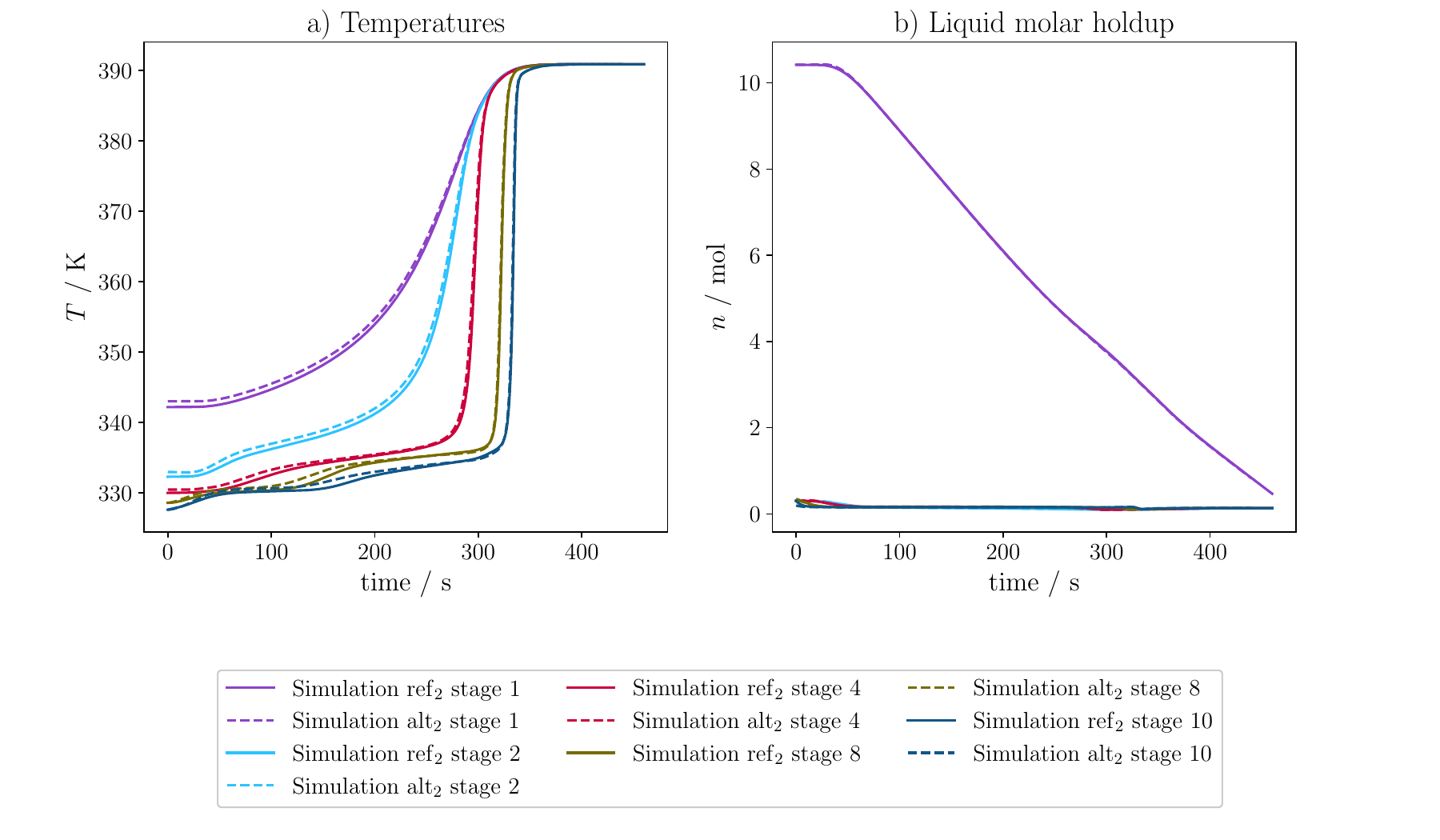}}
	\caption{Simulation results for mixture 2 acetone (1), methanol (2), butanol (3), and chloroform (4); a) temperatures $T$ and b) molar liquid holdup $n$ on all stages}
	\label{fig:dynamic_diagram_example2_temp_n}
\end{figure}

\subsection{Additional results for the five-component mixture} \label{app:additional_results_mixture3}

\begin{figure}[H]
	\centering
	\makebox[\textwidth]{\includegraphics[width=\textwidth]{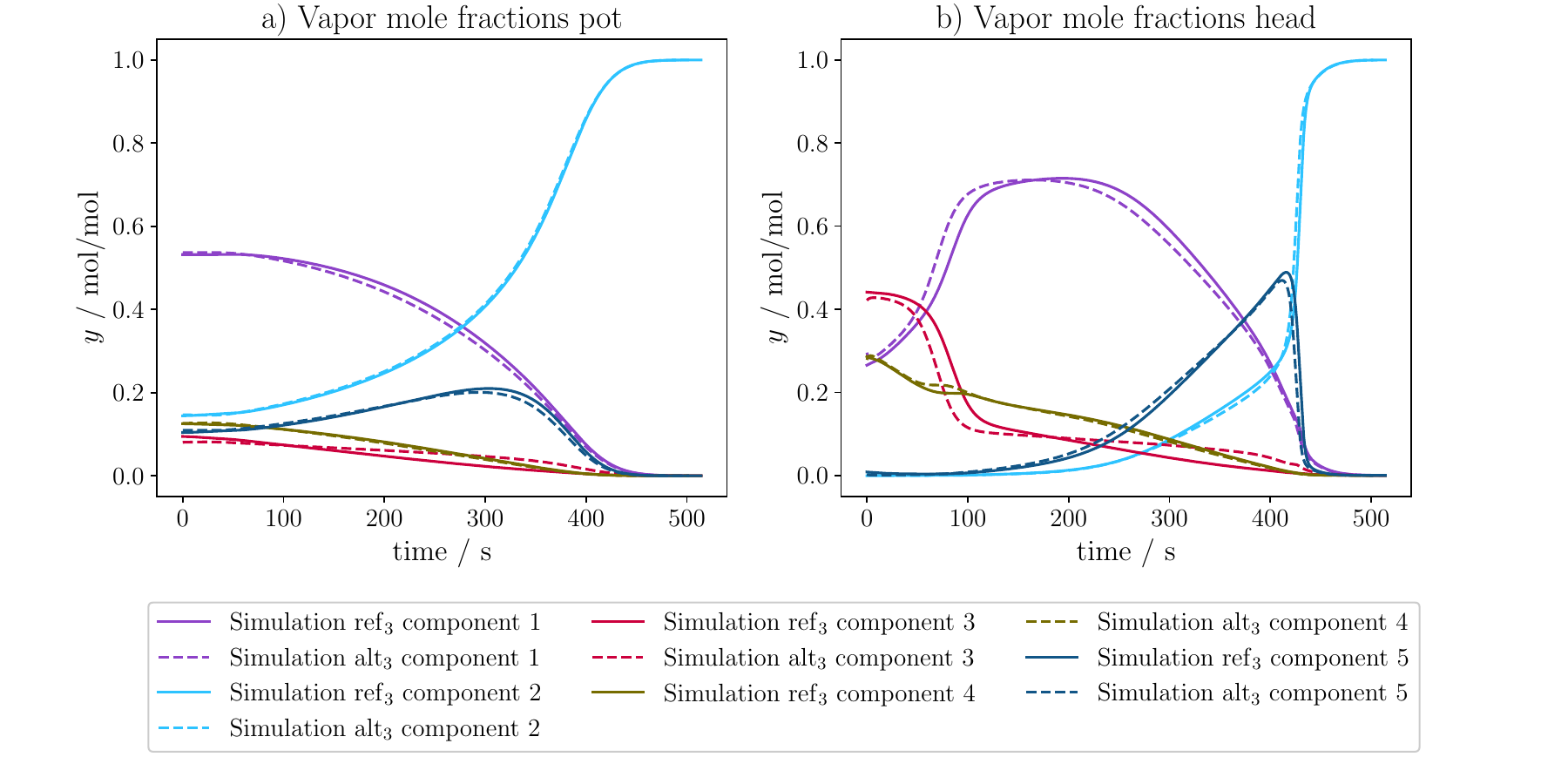}}
	\caption{Simulation results for mixture 2 2-butanone (1), acetic acid (2), ethanol (3), ethyl acetate (4) and toluene (5); vapor mole fractions $\bm{y}$ for pot and head shown in a) and b)}
	\label{fig:dynamic_diagram_example3_vapor}
\end{figure}

\begin{figure}[H]
	\centering
	\makebox[\textwidth]{\includegraphics[width=\textwidth]{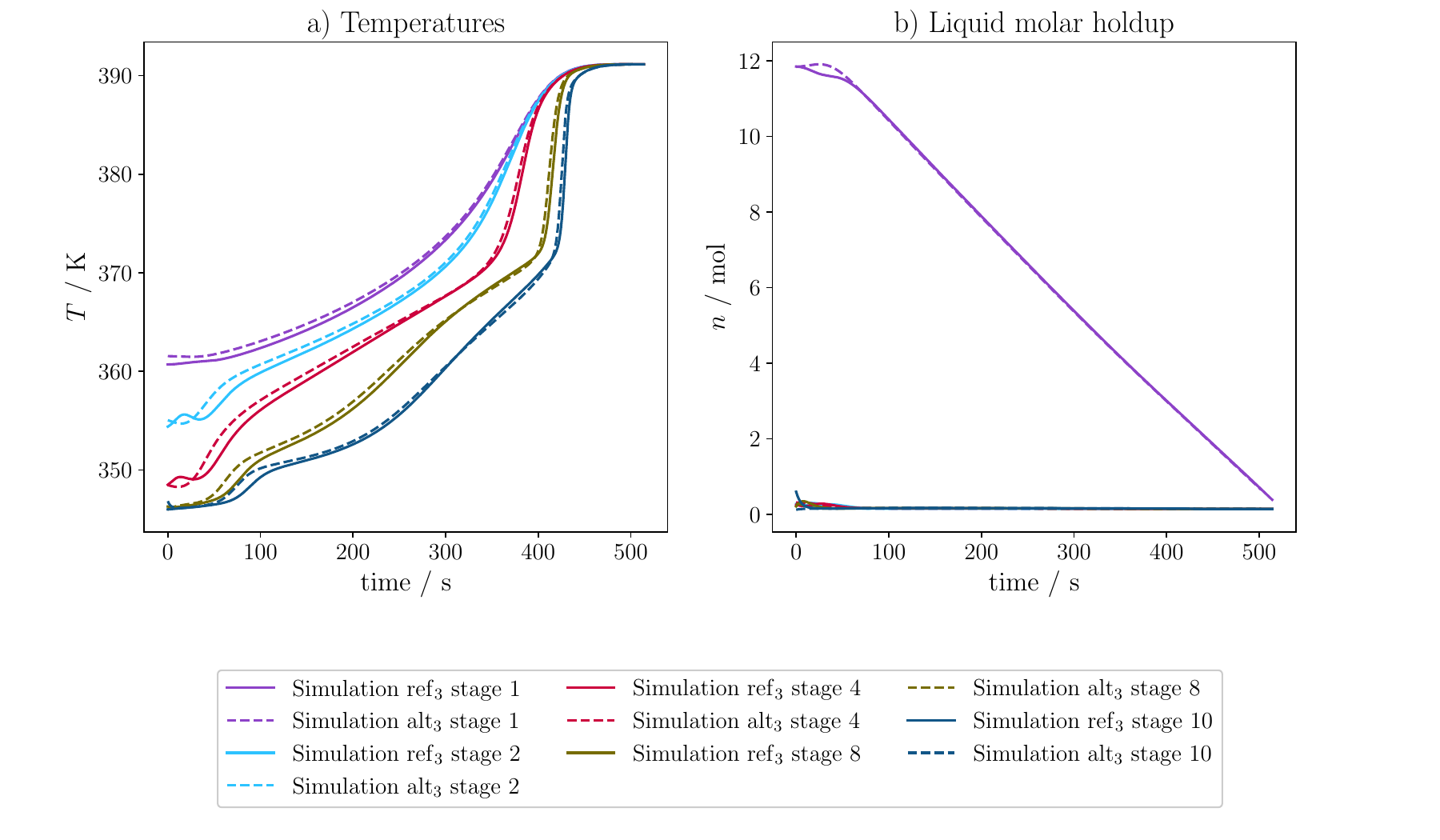}}
	\caption{Simulation results for mixture 2 2-butanone (1), acetic acid (2), ethanol (3), ethyl acetate (4) and toluene (5); a) temperatures $T$ and b) molar liquid holdup $n$ on all stages}
	\label{fig:dynamic_diagram_example3_temp_n}
\end{figure}


\end{document}